\documentclass
{amsart}

\usepackage{mathrsfs}

\usepackage{color}
\usepackage{amsmath}
\usepackage{amssymb}
\usepackage{amsfonts}
\usepackage[abbrev]{amsrefs}
\usepackage{stmaryrd}
\usepackage{graphicx}
\usepackage{bm}
\usepackage[colorlinks=true,linkcolor=blue,urlcolor=blue]{hyperref}

\definecolor{red}{rgb}{1.0,0.0,0.0}

\definecolor{blu}{rgb}{0.0,0.0,1.0}

\definecolor{gre}{rgb}{0.03,0.50,0.3}

\newtheorem{theorem}{Theorem}[section]

\newtheorem{lemma}[theorem]{Lemma}
\newtheorem{cor}[theorem]{Corollary}

\theoremstyle{definition}
\newtheorem{definition}[theorem]{Definition}
\newtheorem{example}[theorem]{Example}

\theoremstyle{definition}
\newtheorem{remark}[theorem]{Remark}
\newtheorem{assumption}[theorem]{Assumption}

\numberwithin{equation}{section}

\newcommand{\bfone}{{\mathbf{1}}}
\newcommand{\eps}{{\varepsilon}}

\newcommand{\abs}[1]{\left|{#1}\right|}
\newcommand{\norm}[1]{\lVert{#1}\rVert}

\newcommand{\F}{{\mathbb{F}}}
\newcommand{\N}{{\mathbb{N}}}

\newcommand{\R}{{\mathbb{R}}}

\newcommand{\Mean}{{{\mathbb{E}}}}
\newcommand{\Prob}{{{\mathbb{P}}}}

\newcommand{\cF}{{\mathcal{F}}}

\newcommand{\om}{{\omega}}

\newcommand{\dd}{{\mathrm{d}}}

\newcommand{\ee}{{\mathrm{e}}}

\newcommand{\mail}[1]{\href{mailto:#1}{\normalfont\texttt{#1}}}

\begin{document}
\title[PDP]
{Non-local Hamilton--Jacobi-Bellman equations  
for the stochastic optimal control of  path-dependent piecewise deterministic processes}

\thanks{The authors would like to thank the anonymous referees
for many  valuable suggestions and comments.}

\author[E.~Bandini]{Elena Bandini\textsuperscript{\MakeLowercase{a},1}}
\thanks{\noindent \textsuperscript{a} University of Bologna, Department of Mathematics, Piazza di Porta San Donato 5, 40126 Bologna (Italy).}
\author[C.~Keller ]{Christian Keller\textsuperscript{\MakeLowercase{b},2}}

\thanks{\noindent \textsuperscript{b} University of Central  Florida, Department of Mathematics,  4393 Andromeda Loop N
Orlando, FL 32816 (USA).
\\
\noindent \textsuperscript{1} E-mail: \mail{elena.bandini7@unibo.it}.
\\
\noindent \textsuperscript{2} E-mail: \mail{christian.keller@ucf.edu}.
\\
The research of the first named author was partially supported by the 2018 GNAMPA-INdAM project \textit{Controllo ottimo stocastico con osservazione parziale: metodo di randomizzazione ed equazioni di Hamilton-Jacobi-Bellman sullo spazio di Wasserstein} and the 2024 GNAMPA-INdAM \textit{Problemi di controllo ottimo in dimensione infinita.}
The research of the second named author was
 supported in part  by NSF-grant DMS-2106077.
}

\date{October 19, 2025}

\subjclass[2010]{}
\keywords{}

\begin{abstract}
We study the optimal control of path-dependent piecewise deterministic processes.
An appropriate dynamic programming principle is established.
We prove that  the associated value function is the unique minimax solution of the corresponding
non-local path-dependent Hamilton--Jacobi--Bellman equation. This is the first well-posedness result
for nonsmooth solutions of fully nonlinear non-local path-dependent partial differential equations.
\end{abstract}
\maketitle
\noindent \textbf{Keywords:} Path-dependent piecewise deterministic processes; non-local path-dependent HJB equations; stochastic optimal control.

\smallskip

\noindent \textbf{AMS 2020:}  45K05, 35D99, 49L20, 90C40, 93E20

\smallskip
\pagestyle{plain}
\tableofcontents

\section{Introduction}
{\color{black} 
Since B.~Dupire's introduction of a functional It\^o calculus in \cite{dupirefunctional},
a 
significant literature on 
{\color{black} second}
order path-dependent partial differential equations (PPDEs) of the form 
\begin{align}\label{E:PPDE1}
\partial_t u(t,x)+H(t,x,u(t,x),\partial_x u(t,x),\partial_{xx} u(t,x))=0,\quad (t,x)\in C(\R_+,\R^d),
\end{align}
has been developed
(see section~\ref{SS:LiteraturePPDE} below).
Note that the derivatives in \eqref{E:PPDE1} are to be understood in the sense of \cite{dupirefunctional}.
PPDEs are relevant for non-Markovian stochastic optimal control problems,
non-Markovian stochastic differential  games, and for  the pricing 
of path-dependent options in financial mathematics. Most works deal with problems
related to diffusion processes. There are no corresponding theories for  fully  nonlinear 
non-local PPDEs, which are relevant
for non-Markovian problems related to jump processes. This work aims to fill this gap.}

{\color{black} Without our research,
already relatively simple classes of problems involving jumps and path-dependence (see Remark~\ref{R:BasicCases} below) are out of reach. A possible specific application that can be addressed 
by our theory  is presented in section~\ref{SS:DelayedHH}.}

{\color{black} 
To be more precise, we investigate  non-local  path-dependent 
Hamilton--Jacobi--Bellman (HJB) equations 
that are related to 
non-Markovian counterparts of piecewise deterministic Markov processes (PDMPs).
PDMPs have been popularized by}
M.H.A.~Davis
(see especially \cite{Davis84} and \cite{DavisBook})
 and found many applications. PDMPs are 
 jump processes
 {\color{black} with a drift component}; between consecutive jump times, they
 are deterministic and solve  ordinary differential equations (ODEs).
{\color{black} We call our non-Markovian counterparts}  \emph{path-dependent} piecewise deterministic processes (PDPs).
{\color{black} They} differ from and extend  PDMPs inasmuch as
 all relevant data 
  are allowed to depend on the realized process history as opposed to  just the realized state as is the case for PDMPs. 
  E.g., between consecutive jump times, path-dependent PDPs solve delay functional differential equations.
Consequently, our theory broadens the realm of possible applications.\footnote{
General, not necessarily Markovian, PDPs 
are briefly mentioned  {\color{black} on pp.~25--26} in \cite{Jacobsen}. We are not aware of any 
thorough treatment of 
path-dependent PDPs. 
}

{
{\color{black} Our main results are existence and uniqueness for terminal value problems involving 
non-local path-dependent HJB equations of the form}
\begin{equation}\label{E:Intro:PPDE}
\begin{split}
&\partial_t u(t,x) + \inf_{a\in A} \Big \{\lambda(t,x,a)\,\int_{\R^d} [u(t, \bfone_{[0,t)}\,x+\bfone_{[t,\infty)}  e)-u(t,x)]\,Q(t,x,a,\dd e)\\
&\qquad +\ell(t,x,a)+(f(t,x,a),\partial_x u(t,x))\Big 
\}=0,\quad (t,x)\in [0,T)\times D(\R_+,\R^d),
\end{split}
\end{equation}
{\color{black} where $D(\R_+,\R^d)$ denotes the Skorokhod space 
{\color{black} and} all terms 
 are
to be understood as \emph{non-anticipating}, e.g., $u(t,x)=u(t,x_{\cdot\wedge t})$.}
{\color{black} Moreover,  we show that their solutions coincide with value functions of
appropriate  
stochastic optimal control problems related to path-dependent PDPs.
To this end, we establish a suitable  dynamic programming principle.}
\begin{remark}
{\color{black} A precise definition of the derivatives in \eqref{E:Intro:PPDE} is given in subsection~\ref{SS:classical}
(note that they are not Fr\'echet derivatives).}
{\color{black} Just as in the case of the (local) PPDE~\eqref{E:PPDE1},
these derivatives  can be understood in the sense of  B.~Dupire's
so-called horizontal and vertical derivatives  in \cite{dupirefunctional}.
However, as it is more common in the 
literature on 
{\color{black} first}
order PPDEs, we prefer to use the earlier notion
of the so-called coinvariant derivatives on path spaces, which are due to A.V.~Kim (see, e.g., \cite{Kim85}  and \cite{KimBook}).
Both notions of derivatives are essentially equivalent (see, e.g., section~3 in \cite{BK1} for more details).}
\end{remark}
Counterparts of \eqref{E:Intro:PPDE}  in the non-path-dependent case would be of the form
\begin{equation}\label{E:Intro:PDE}
\begin{split}
&\partial_t \tilde{u}(t,\tilde{x}) +\inf_{a\in A} \Big \{\tilde{\lambda}(t,\tilde{x},a)
\,\int_{\R^d} [\tilde{u}(t, e)-\tilde{u}(t,\tilde{x})]\,\tilde{Q}(t,\tilde{x},a,\dd e)\\
&\qquad +\tilde{\ell}(t,\tilde{x},a)+(\tilde{f}(t,\tilde{x},a),\partial_{\tilde{x}} \tilde{u}(t,\tilde{x}))
\Big \}=0,\quad (t,\tilde{x})\in [0,T)\times \R^d.
\end{split}
\end{equation}
As a rule, classical solutions to such HJB equations rarely exist and theories of non-smooth solutions such as 
viscosity solutions (see, e.g., \cite{CrandallIshiiLions}, \cite{BardiCapuzzoDolcetta}) or minimax solutions (see, e.g.,
 \cite{Subbotin}) are needed. 
{\color{black} In this work, we use minimax solutions, mainly due to a technical obstacle, which is explained in
section~\ref{SS:specificDifficulty}.}
 
 {\color{black}
\begin{remark}\label{R:BasicCases}
Note that our new theory is needed for relatively basic cases.

 (i) Optimal control of {\color{black} or 
 option pricing related to} standard 
 PDMPs (see, e.g., section~4 in \cite{BR10PDMP} and section~10 in \cite{Jacobsen})
  with the same data as the HJB equation~\eqref{E:Intro:PDE}
  but with path-dependent terminal cost {\color{black} or path-dependent payoff}, 
  e.g., $h(x)=\sup_{0\le t\le T} \abs{x(t)}$. Then the corresponding
  HJB equation is a PPDE of the form 
 \begin{align*}
 &\partial_t u(t,x) +\inf_{a\in A} \{\tilde{\lambda}(t,x(t),a)
\,\int_{\R^d} [u(t, \bfone_{[0,t)}\,x+\bfone_{[t,\infty)}  e)-u(t,x)]\,\tilde{Q}(t,x(t),a,\dd e)\\
&\qquad +\tilde{\ell}(t,x(t),a)+(\tilde{f}(t,x(t),a),\partial_{x} \tilde{u}(t,x))
\}=0,\quad (t,x)\in [0,T)\times  D(\R_+,\R^d),
\end{align*}
with terminal condition $u(T,x)=h(x)$, $x\in D(\R_+,\R^d)$.
Note that $A$ would be a singleton in the case of option pricing.
 
 (ii)  Change of the ODE in a PDMP to  a delay differential equation. Even 
 results on well-posedness for HJB equations
 associated to the optimal control of  delay differential equations  with  jumps that follow a Poisson process
 are new. If $d=1$, then the corresponding HJB equation would be of the form
 \begin{align*}
& \partial_t u(t,x) +\lambda [u(t, \bfone_{[0,t)}\,x+\bfone_{[t,\infty)}\,(x(t)+1))-u(t,x)]\\ &\qquad 
+\inf_{a\in A} 
(\tilde{f}(t,x((t-\tau)\vee 0),a),\partial_x u(t,x))
=0,\quad (t,x)\in [0,T)\times D(\R_+,\R^d),
 \end{align*}
 as a special case of \eqref{E:Intro:PPDE}.
 Similar and actually more general models with jumps and delays  are used in \cite{AH20} 
 for option pricing.
   \end{remark}}


\subsection{Related research}\label{SS:LiteraturePPDE}
 PDEs associated to PDMPs, i.e., variations of \eqref{E:Intro:PDE}, were studied in 
\cite{Vermes85Stoch}, \cite{Soner86SICON-II}, \cite{Ye94}
\cite{DempsterYe95AAP},   \cite{DavisFarid},
\cite{ForwickEtAl04},  \cite{Goreac12}, \cite{Bandini18COCV},
\cite{Bandini19SICON},  \cite{ACE19SIFIN}, \cite{BandiniThieullen20AMO},
{\color{black} and \cite{ColaneriEtAl20SPA},} to mention but a few works.
{\color{black}
\begin{remark}
The data in the above literature are either not time-dependent or 
required to be continuous with respect
to time. Our results are stronger in the sense that we require our data to be only
measurable with respect to time.
\end{remark}
}

{\color{black}
Next, we provide 
a short overview of the literature on PPDEs.

Existence and uniqueness of minimax solutions for  
{\color{black} first}
 order path-dependent Hamilton--Jacobi equations
was established in \cite{Lukoyanov03} and of viscosity solutions in \cite{Lukoyanov07}.
Wellposedness of viscosity solutions  for  
{\color{black} first}
order path-dependent HJB equations under significantly weaker
conditions (only continuity of the data in  the path variable $x$ with respect to sup norm is needed instead
of requiring continuity with respect to an $L_p$ norm) was proved in \cite{zhou2020viscosity1st}.
For more details as well as a comprehensive overview for  
{\color{black} first}
order PPDEs,
 we refer to the recent survey paper \cite{GL24survey}.
 
Wellposedness of viscosity solutions for  semilinear 
{\color{black} second}
order  PPDEs 
was established in \cite{EKTZ11}
and, subsequently, for semilinear 
{\color{black} second}
order path-de\-pen\-dent integro-differential equations in \cite{Keller16SPA}.
Note that \cite{Keller16SPA} is the only relevant work
for non-local PPDEs and that there have been no works on nonsmooth solutions for fully nonlinear
non-local PPDEs such as \eqref{E:Intro:PPDE}.

Many works in the literature on fully nonlinear 
{\color{black} second}
order PPDEs
 such as \cite{ETZ_I}, \cite{ETZ_II}, \cite{EkrenZhang16PUQR}, \cite{RTZ17},
\cite{RR20SIMA}, \cite{cosso21v2path}, \cite{zhou2020viscosity}, and \cite{Keller24}  
cover 
HJB equations of the form
\begin{equation}\label{E:2nd:HJB:PPDE}
\begin{split}
&\partial_tu(t,x)+\inf_{a\in A} \left\{\frac{1}{2}\sigma(t,x,a)^2\partial_{xx} u(t,x)+\ell(t,x,a)+f(t,x,a)\partial_xu(t,x)\right\}=0,\\
&\qquad\qquad(t,x)\in [0,T)\times C(\R_+,\R^d).
\end{split}
\end{equation}
Regarding existence and uniqueness of viscosity solutions for
{\color{black} second}
order path-de\-pen\-dent HJB equations such as \eqref{E:2nd:HJB:PPDE}, \cite{zhou2020viscosity} and \cite{Keller24} 
allow for the weakest conditions\footnote{
We refer to subsection~1.3 in \cite{Keller24} for a more detailed discussion.
} among the just mentioned works such as a fairly general
possibly degenerated controlled diffusion term as well as  continuity of the data
in the sup norm (as opposed to continuity in an $L_p$-norm).
It should be mentioned that the notion of viscosity solution in \cite{zhou2020viscosity} is 
a modification of
the usual one in the non-path-dependent case
(see \cite{CrandallLions83TAMS})
and uses pointwise tangency for the employed test functions
whereas the notion in  \cite{Keller24} uses tangency in mean (which was introduced in \cite{EKTZ11}).

\begin{remark}
Instead of the controlled diffusion term $(1/2)\sigma(t,x,a)^2\partial_{xx} u(t,x)$ in \eqref{E:2nd:HJB:PPDE},
we have  the comparable controlled non-local term
\begin{align*}
\int_{\R^d} [u(t, \bfone_{[0,t)}\,x+\bfone_{[t,\infty)}  e)-u(t,x)]\lambda(t,x,a)\,Q(t,x,a,\dd e)
\end{align*}
in \eqref{E:Intro:PPDE}.
Our work allows for the data to be continuous in $x$ with respect to the sup norm as well. 
Furthermore, our data need only be measurable with respect to time whereas in 
\cite{zhou2020viscosity} continuity in time is required.
\end{remark}
}

{\color{black}Last but not least, we want to note that, besides the PPDE approach
(based on path derivatives by B.~Dupire~\cite{dupirefunctional} or A.V.~Kim~\cite{KimBook}),
it is also very common for the treatment of 
stochastic optimal control problems with delays
 to use the Hilbert space $L^2(-\tau,0;\R^d)$ or the space $C([-\tau,0],\R^d)$ as state space and
consider HJB equations on those spaces with derivatives  understood as Fr\'echet derivatives: For a
corresponding overview, 
we refer  the reader to the monograph \cite{FabbriGozziSwiech}
(in particular to  its section~2.6.8  and the references therein).}

\subsection{A simplified presentation of the control problem} 
Given initial conditions $(t,x)\in [0,T)\times D(\R_+,\R^d)$, our stochastic optimal control problem is to minimize
a cost functional 
\begin{align*}
J(t,x,\alpha):=\Mean \left[\int_t^T \ell(s,X_{\cdot\wedge s}^{t,x,\alpha},\alpha(s))\,\dd s+h(X_{\cdot\wedge T}^{t,x,\alpha})\right]
\end{align*}
over an appropriate class of controls $\alpha$, {\color{black}which will be discussed shortly.}
The controlled stochastic process $X^{t,x,\alpha}$ is a path-dependent {\color{black} PDP}  
with jump times $T_n^{t,x,\alpha}$, $n\in\N$, and post-jump locations
$E_n^{t,x,\alpha}$, $n\in\N$, such that the 
 the following holds: 
\begin{align*}
&X^{t,x,\alpha}(s)=x(s)\quad\text{for each $s\in [0,t]$},\\
&T_0^{t,x,\alpha}=t,\quad\text{$T_1^{t,x,\alpha}$ is the first jump time of $X^{t,x,\alpha}$, etc.}\\
&E_0^{t,x,\alpha}=x(t),\quad\text{$E_1^{t,x,\alpha}$ is the first post-jump location of 
 $X^{t,x,\alpha}$, }\\ 
 &\qquad\qquad\qquad\qquad
 \text{i.e., $E_1^{t,x,\alpha}=X^{t,x,\alpha}(T_1^{t,x,\alpha})$, etc., and}\\
 & \frac{d}{ds} X^{t,x,\alpha}(s)=f(s,X_{\cdot\wedge s}^{t,x,\alpha},\alpha(s))\text{ for a.e.~$s\in (T_n^{t,x,\alpha},T_{n+1}^{t,x,\alpha})$.}
\end{align*}
The probability distributions of the jump times and post-jump locations are specified via a survival
function $F^{t,x,\alpha}$ and a random measure $Q(\cdot, \dd e)$ on $\R^d$, e.g.,
\begin{align*}
\Prob (T^{t,x,\alpha}_1>s)&=F^{t,x,\alpha}(s)\text{ for every $s>{\color{black}t}$ and}\\
\Prob( E^{t,x,\alpha}_1\in B\vert T^{t,x,\alpha}_1)&=
Q(T^{t,x,\alpha}_1, \phi_{\cdot\wedge T^{t,x,\alpha}_1}^{t,x,\alpha}, \alpha(T^{t,x,\alpha}_1) ,B)\text{ for each Borel set $B\subset\R^d$,}
\end{align*}
where  $\phi^{t,x,\alpha}$ is  the unique
solution of $\frac{d}{ds} \phi^{t,x,\alpha}(s)=f(s,\phi_{\cdot\wedge s}^{t,x,\alpha},\alpha(s))$ on $(t,\infty)$ with
initial condition $\phi^{t,x,\alpha}(s)=x(s)$ on $[0,t]$, i.e., $\phi^{t,x,\alpha}$ and $X^{t,x,\alpha}$
coincide before the first jump time.
\medskip

\textit{The class of controls: Differences between the Markovian and the path-dependent cases.}
In both cases, controls are open-loop between jump times.

In the Markovian case, admissible controls are of the form\footnote{For readability's sake, we omit the superscript $\phantom{}^{t,x,\alpha}$ or part of it.}
\begin{align}\label{E:Intro:AdmissibleControls0:Markov}
\tilde{\alpha}(t)=\tilde{\alpha}_n(t\vert T_n,E_n),\quad T_n\le t<T_{n+1}.
\end{align}
This structure leads naturally to an equivalent (discrete-time) Markov decision model
(see, e.g., \cite{DavisBook}).
{\color{black} Such an equivalency is of major importance in the theory of PDMPs and a suitable counterpart 
{\color{black} is}
crucial in this work as well. 

In the path-dependent case, {\color{black} our} 
 controls {\color{black}are}  of the form
\begin{align}\label{E:Intro:AdmissibleControls1}
\alpha(t)=\alpha_n(t\vert T_n, \{X^\alpha(s)\}_{0\le s\le {T_n}}),\quad T_n\le t<T_{n+1},
\end{align}
where $X^\alpha$ is our controlled  piecewise deterministic process. 
{\color{black}
In the spirit of the Markovian case, 
we establish a correspondence between our continuous-time control problem and 
 a related discrete-time model
(see section~\ref{SS:value} and especially Theorem~\ref{0L:value:DPP}).}
{\color{black} To this end, 
we 
{\color{black} express}
the history $\{X^\alpha(s)\}_{0\le s\le t}$ {\color{black} in \eqref{E:Intro:AdmissibleControls1}}
as an appropriate function 
{\color{black} that depends on} 
  the past jump times $T_1$, $\ldots$,
 $T_n$, 
 the past jump locations $E_1$, $\ldots$, $E_n$,
 {\color{black} and, in contrast to the Markovian case,  also on the control history, i.e., on all realized open-loop controls
 (see Definition~\ref{D:randomizedPolicy} for the precise definition of the policies for 
 our  control problem).}} 
  {\color{black}Doing so is needed for}  
the proof of the  dynamic programming principle (Theorem~\ref{0T:value:DPP}), which is crucial
 in establishing that solutions of the HJB equation \eqref{E:Intro:PPDE} coincide with the value function
 of our stochastic optimal control problem.

\subsection{Our approach  in dealing with specific difficulties concerning \label{SS:specificDifficulty} 
 the non-local path-dependent HJB equation  \eqref{E:Intro:PPDE}}
We follow the methodology by Davis and Farid in \cite{DavisFarid}, where existence and uniqueness of viscosity solutions
for HJB equations related to PDMPs are established via a fixed-point argument
{\color{black}(see the beginning of section~\ref{S:HJB} for a detailed outline)}. 
However, our situation in the path-dependent case
leads to several additional obstacles. We continue by describing some of those difficulties and how they will be circumvented.
First note that {\color{black} the}  main ingredients of the Davis--Farid methodology are   existence and uniqueness  results
of certain  \emph{local} HJB equations, in our case, these HJB equations are path-dependent and of the form
\begin{equation}\label{E:Intro:localPPDE}
\begin{split}
\partial_t u(t,x)+{\color{black}H^\psi(t,x,u(t,x),\partial_x u(t,x))}=0, \quad (t,x)\in [0,T)\times D(\R_+,\R^d),
\end{split}
\end{equation}
where $\psi:[0,T]\times D(\R_+,\R^d)\to\R$  and  ${\color{black}H^\psi(t,x,y,z)}$ is of the form
\begin{equation}\label{E:Intro:localPPDE:Hamiltonian}
\begin{split}
&{\color{black}H^\psi(t,x,y,z)}=
\inf_{a\in A} \{\lambda(t,x,a)\int_{\R^d} [\psi(t,\bfone_{[0,t)} x+\bfone_{[t,\infty)} e)-y]\,Q(t,x,a,\dd e)\\
&\quad +\ell(t,x,a)+(f(t,x,a),z)\},\quad (t,x,y,z)\in [0,T)\times D(\R_+,\R^d)\times\R\times\R^d.
\end{split}
\end{equation}
{\color{black} Note that \eqref{E:Intro:PPDE} is identical to   \eqref{E:Intro:localPPDE} in case  $\psi\equiv u$.}

One of our obstacles in the path-dependent case is a possible lack of regularity of the term
$\psi(t,\bfone_{[0,t)} x+\bfone_{[t,\infty)} e)$ in \eqref{E:Intro:localPPDE:Hamiltonian}.
Even if $\psi$ is continuous, the map $t\mapsto\psi(t,\bfone_{[0,t)} x+\bfone_{[t,\infty)} e)$ might not be 
continuous (see  appendix~\ref{S:Appendix:D}) and 
therefore we can expect the Hamiltonian $F_\psi$ to be at most measurable with respect to $t$.
{\color{black} To deal with this issue,}
we develop in our companion paper \cite{BK1}
 a minimax solution theory for path-dependent Hamilton--Jacobi equations with time-measurable coefficients
and also with ``$u$-dependence" of the Hamiltonian, {\color{black} i.e.,
we cover equations of the form $\partial_t u+H(t,x,u,\partial_x u)=0$, whereas the previous
minimax solutions literature on PPDEs covers equations of the form $\partial_t u+H(t,x,\partial_x u)=0$.}  There seems to have been no treatment
of any of those two aspects in the literature 
(see \cite{GLP21AMO} for an overview)
and thus the results in \cite{BK1} ought to be of independent interest besides
 the current matter. 
 {\color{black}   Also note that, while we use in \cite{BK1} a combination of minimax and viscosity
 solution techniques to establish uniqueness for minimax solutions and to verify that value functions
 of relevant optimal control problems are minimax solutions of the corresponding HJB equations,
  a  complete viscosity solution theory for PPDEs with time-measurable
 coefficients does not exist to the best of our knowledge. For further discussion on relevant background and difficulties, we refer to section~1 in \cite{BK1}.}
 {\color{black} Furthermore, note that there has been no minimax solution theory related to optimal control of PDMPs.}
 
 A second obstacle is the following. A standard proof of a comparison principle between
 appropriate sub- and supersolutions of \eqref{E:Intro:localPPDE}  (as done in \cite{BK1})
 requires at least some regularity of ${\color{black}H^\psi}$ with respect to its $x$-component. Typically, 
 some type of uniform continuity
 is required. For simplicity (even though it might be more difficult to prove), let us require Lipschitz continuity, i.e.,
 we would need
 \begin{align*}
& \abs{\int_{\R^d} [\psi(t,\bfone_{[0,t)} x_1+\bfone_{[t,\infty)} e)-
 \psi(t,\bfone_{[0,t)} x_2+\bfone_{[t,\infty)} e)\,Q(t,x_1,a,\dd e)}\\
 &\qquad\qquad\le L_\psi 
 \sup_{0\le s\le t} \abs{x_1(s)-x_2(s)},
 \end{align*}
 which, in the Markovian or state-dependent case, is trivially satisfied because
 \begin{align*}
 \psi(t,\bfone_{[0,t)} x_1+\bfone_{[t,\infty)} e)-
 \psi(t,\bfone_{[0,t)} x_2+\bfone_{[t,\infty)} e)=\tilde{\psi}(t,e)-\tilde{\psi}(t,e)=0
 \end{align*}
 for some function $\tilde{\psi}:[0,T]\times\R^d\to\R$ with $\tilde{\psi}(t,x(t))=\psi(t,x)$.
The slightly involved 
 proof of {\color{black}Theorem}~\ref{L:barpsi:reg:Step1} deals with this issue by establishing
 Lipschitz continuity for our value function $V$, which yields the desired regularity at least for 
  ${\color{black}H^\psi}$ with $\psi=V$.

 \subsection{The delayed Hodgkin--Huxley model as 
 possible application}\label{SS:DelayedHH}

 {\color{black} The Hodgkin--Huxley (HH) model, developed in 1952 by 
A. Hodgkin and 
A. Huxley in \cite{HH52},   
 describes how action potentials (nerve impulses) are initiated and propagated along axons.
It is based on the concept of voltage-gated ion channels that allow ions to flow through the membrane of the neuron, generating electrical signals in response to changes in the membrane potential. 
In the original HH model, the transitions between different states of ion channels (e.g., from closed to open or inactivated states) are deterministic and described by ODEs with specific kinetic parameters. 
In reality, the behaviour of the ion channels is subject to random fluctuations, and it is natural to model the state transitions through a continuous-time Markov chain $d_t$, see e.g., \cite{Kampen}. 

For such a model one can consider  optogenetics control problems. Optogenetics allows precise control of cellular activity by using light to manipulate genetically engineered ion channels that are sensitive to light.  
In the control problem new (rhodopsin) channels that are sensitive to light are inserted in the neuron.
 Experimentally, the channel is illuminated and the effect of
the illumination is to put the channel in one of its conductive states. From a mathematical point of view, this type of problems has been successfully studied (in the case of spatio-temporal models) by dealing with the theory of controlled infinite dimensional PDMPs, see e.g., \cite{BandiniThieullen20AMO}, \cite{RTT17PDMP}. 

On the other hand,  recent studies revealed that significant
delays could take place between voltage changes and ion channel activation/inactivation dynamics. As a matter of fact, the response of the gating variables to changes in the membrane potential is not instantaneous, but rather involves some time lag: 
ion channels may have intrinsic delays in their opening and closing, as well as delays in the response to changes in membrane voltage. 
The need to capture more realistic dynamics of neurons have lead to the development of the so-called  delayed  HH models, see e.g.  \cite{CessacVolkov}.
In this extended  model,   the membrane potential  $v$ would  follow a delayed ODE of the type 
$$
C_m v'(t) = - I_{ion}(v(t), m(t-\tau))+ I_{ext}(t), \quad v(0)=v \in [0, V_{max}].
$$
Here $I_{ion}(v(t), d(t-\tau))$, that represents the ionic current, would  depend on the delayed state of the gating variable $d(t-\tau)$, and   the stochastic gating state  $d$ would be  a pure jump process evolving  according to a jump measure  with compensator $\lambda(v(t-\tau), d(t-\tau))Q(v(t-\tau),d(t-\tau), \rm dy)\,{\color{black} \dd t}$.

Our theory is well suited to manage such a  model. Therefore, as a future application,   corresponding optogenetic control problems could be  studied   by suitably adapting our results to the bounded domain framework.

 }

\subsection{Organization of the rest of the paper}
Section~\ref{S:Notation} introduces frequently used notation.
In section~\ref{S:Canonical},  two canonical sample spaces are specified, the path space $\Omega$ 
(the Skorokhod space, previously denoted by $D(\R_+,\R^d)$)
and the canonical space of our controlled marked point processes.
Section~\ref{S:OptimalControl} provides a precise description of our stochastic optimal control problem
including assumptions for the data as well as statement and proof of the dynamic programming principle.
Section~\ref{S:HJB} contains the complete treatment of our non-local path-dependent HJB equation:
The fixed-point approach in the spirit of Davis and Farid, 
appropriate notions of minimax solutions, and, ultimately, existence, uniqueness, and a comparison principle for
\eqref{E:Intro:PPDE} together with relevant terminal conditions. The appendices include technical as well as 
fundamental material. 
Appendix~\ref{S:Appendix:MDP} is relevant in its own right (independent from its applications in this paper):
It connects a fairly general path-dependent discrete-time  decision model  to a standard discrete-time decision model.
Thereby, the Bellman equation for the path-dependent model is derived. Appendix~\ref{A:TechnicalProofs} contains technical proofs. Finally, appendix~\ref{S:Appendix:D}
provides results on a  possible lack of regularity of certain functionals on Skorokhod space.

\section{Notation and preliminaries}\label{S:Notation}
Let $\N$ be the set of all strictly positive integers and  $\N_0=\N\cup\{0\}$.
Fix $d\in\N$ and $T\in (0,\infty)$.
 Let 
  $\Delta\not\in \R^d$ be a cemetery state.
Given $(s,x)\in \R_+\times D(\R_+,\R^d)$ and $e\in \R^d$, define
\begin{align}\label{E:concatenation}
x\otimes_{s} e\quad\text{by\,\,\,\, $(x\otimes_{s} e)(t)=\begin{cases}
x(t),\,\textup{if}\,\, t\in [0,s)\\
e,\, \textup{if}\,\,t\in [s,\infty).
\end{cases}$}
\end{align}
For measurability properties in $(x,t,e)$, see Theorem 96, 146-IV, in 
\cite{DMprobPot}.

\begin{remark}\label{R:concatenation}
Clearly,
$x\otimes_s e=x(\cdot\wedge s)\otimes_s e$. This  identity will simplify notation.
\end{remark}


{\color{black} Given sets $S$ and $\tilde{S}$, we denote by $\tilde{S}^S$ the set of all functions from $S$ to $\tilde{S}$.}

Given a topological space $S$, we denote by $\mathcal{B}(S)$ its Borel $\sigma$-field,
 by $\mathscr{P}(S)$ the set of all probability measures on $\mathcal{B}(S)$,
 and by $B(S)$ the set of all bounded Borel-measurable functions from $S$ to $\R$.
 Moreover, we write $C(S)$, {\color{black} $C_b(S)$,} $\mathrm{USC}(S)$, 
{\color{black} $\mathrm{BUSC}(S)$},  
$\mathrm{LSC}(S)$,  {\color{black} and $\mathrm{BLSC}(S)$, resp.,} for the sets
 of all continuous, {\color{black}bounded continuous,}  
 upper semicontinuous,  {\color{black} bounded upper semicontinuous, }
lower semicontinuous,
 {\color{black}  and bounded lower semicontinuous }
  functions from $S$ to $\R$, resp.
 Given another topological space $\tilde{S}$, we write $C(S,\tilde{S})$ for the set of all continuous functions
 from $S$ to $\tilde{S}$. 
 {\color{black} If $S$ is a Borel space, then
 we write $\mathrm{BLSA}(S)$ for the set of all bounded lower semi-analytic functions from $S$
 to $\R$ (for a detailed treatment of analytic sets and lower semi-analytic functions,
 we refer the reader to sections~7.1 and 7.2 in \cite{BS}).}
 
 {\color{black}  
Let us  also recall the notion of universal measurability: 
 A subset $\underline{S}$ of
a Borel space $S$  is \emph{universally measurable} if, for every $\Prob\in\mathscr{P}(S)$,  the set $\underline{S}$
belongs to the completion of $\mathcal{B}(S)$ with respect to $\Prob$ 
(see Definition~7.18 of \cite{BS}).}
 
 We  frequently use the notation 
 $\bfone_S$ for indicator functions of some set $S$, i.e., $\bfone_S(t)=1$ if $t\in S$ and 
 $\bfone_S(t)=0$ if $t\not\in S$,  and the notation
 $\delta_r$ for a Dirac measure concentrated at some point $r$; the domains of $\bfone_S$ and 
 of $\delta_r$ will be clear from context.
 We  also always use the convention
 $\infty\cdot 0=0$.


{\color{black} Occasionally, we use 
$\norm{x}_s:=\sup_{0\le t\le s} \abs{x(t)}$
given $s\in [0,T]$ and $x\in D(\R_+,E)$.}
 {\color{black} Moreover, given a bounded real-valued function $\varphi$, we denote its sup norm
 by $\norm{\varphi}_\infty$.}

\section{The canonical setup}\label{S:Canonical}

\subsection{The canonical path space}\label{SS:CanonicalPathSpace}
Let $\Omega=D(\R_+,\R^d)$. The canonical process
$X=(X_t)_{t\ge 0}$ on $\Omega$ is defined by $X_t(\omega)=\omega(t)$ for each
$(t,\om)\in\R_+\times\Omega$. Let $\F^0=\{\cF^0_t\}_{t\ge 0}$ be the   filtration generated
by $X$ and let $\Omega$ be equipped  with the $\sigma$-field $\cF^0:=\vee_{t\ge 0} \cF^0_t$.

We consider non-empty subsets of $\R_+\times\Omega$  to be equipped with the pseudo-metric 
\begin{align}\label{E:SS:CanonicalPathSpace:PseudoMetric}
((t,x),(s,\tilde{x}))\mapsto\abs{t-s}+\sup_{0\le r <\infty} \abs{x(r\wedge t)-\tilde{x}(r\wedge s)}.
\end{align}

\subsection{The canonical space of controlled marked point processes}\label{SS:Canon:Marked}
First, consider a non-empty Borel subset $A$  of a Polish space and fix an element $a^\circ\in A$. 
The set  $A$  is 
 our control action space. 
Next,  define a  
set of open-loop controls
\begin{align*}
\mathcal{A}:=\{\alpha:\R_+\to A\text{ Borel-measurable}\}
\end{align*}
with cemetery state $\Delta^\prime\not\in\mathcal{A}$.
\begin{remark}\label{R:AAisBorelSpace}
Note that  $\mathcal{A}$ is a Borel space when it is  equipped with $\mathcal{B}(\mathcal{A})$, the smallest $\sigma$-algebra under which all functions 
$\alpha\mapsto \int_0^\infty \ee^{-t}\eta(t,\alpha(t))\,\dd t$,   $\mathcal{A}\to\R$,  
$\eta\in B(\R_+\times A)$,  are measurable (Lemma~1 in \cite{Yushkevich80}).
\end{remark}

Our canonical space $\check{\Omega}$ is  
 defined
as the set of all  sequences $\check{\omega}=(t_n,e_n,\alpha_n)_{n\in\N_0}$ 
for which the following holds (cf.~Remarque~III.3.43 in \cite{JacodBook}):


(i) $(t_n,e_n,\alpha_n)\in \left((0,\infty)\times \R^d\times\mathcal{A}\right)\cup\{(\infty,\Delta,\Delta^\prime)\}$.

(ii) $t_n\le t_{n+1}$.

(iii) $t_n<t_{n+1}$  unless $t_n=\infty$.





For each $n\in\N_0$, we define  canonical mappings  ${T}_n:\check{\Omega}\to (0,\infty]$, ${E}_n:
\check{\Omega}\to \R^d\cup\{\Delta\}$,  and $\mathcal{A}_n:\check{\Omega}\to \mathcal{A}\cup\{\Delta^\prime\}$  by
$${T}_n(\check{\omega})=t_n,  \quad {E}_n(\check{\omega})=e_n, \quad \mathcal{A}_n(\check{\omega})=\alpha_n,
\qquad \check{\omega}=(t_{\color{black}j},e_{\color{black}j},
\alpha_{\color{black}j})_{\color{black}{j}\in\N_0}\in\check{\Omega},$$
{\color{black} and  we also  write $\mathcal{A}_n(t)$  instead of
 $\check{\omega}\mapsto [\mathcal{A}_n(\check{\omega})](t)$, $\check{\Omega}\to A$,
$t\in\R_+$, i.e.,
\begin{align*}
[\mathcal{A}_n(t)] (\check{\omega})=\alpha_n(t),
\qquad \check{\omega}=(t_j,e_j,\alpha_j)_{j\in\N_0}\in\check{\Omega}.
\end{align*}
}
Consider the corresponding
random measure 
$$\check{\mu}(\dd t\,\dd e\,\dd \alpha)=\sum_{n\ge 0}\bfone_{\{T_n<\infty\}}\,\delta_{({T}_n,{E}_n,\mathcal{A}_n)} (\dd t\,\dd e\,\dd \alpha),
$$
 let $\check{\F}=\{\check{\cF}_t\}_{t \ge 0}$ be the  filtration generated by $\check{\mu}$, and
equip $\check{\Omega}$ with the $\sigma$-field $\check{\cF}:=\vee_{t\ge 0} \check{\cF}_t$.

\begin{remark}
The filtration $\check{\F}$ is right-continuous (Theorem~T25, p.~304, in \cite{BremaudBook})
{\color{black}  and, by definition,  is a raw filtration.   It  is emphasized in 
Remark~4.2.1 of \cite{Jacobsen} that   it is ``essential" not to use
 \emph{completions} of $\check{\F}$ in the context of the construction of marked point processes.}
\end{remark}

For each $n\in\N$, we shall also need the space\footnote{Our numbering is consistent with
the one of the marginals $r_N$ in chapters 8 and 9 of \cite{BS}.}
\begin{align*}
\check{\Omega}_n&:=\{ 
\check{\omega}_n \in((\R_+\times E\times \mathcal{A})\cup\{(\infty,\Delta,\Delta^\prime)\})^{n}:\\ 
&\qquad\qquad\exists (t_j,e_j,\alpha_j)_{j\in\N_0}\in\check{\Omega}\text{ with }
\check{\omega}_n= (t_j,e_j,\alpha_j)_{j=0}^{n-1}\}
\end{align*}
with corresponding filtration $\check{\F}_n=\{\check{\cF}_{n,t}\}_{t\ge 0}\}$ 
and $\sigma$-field $\check{\cF}_{n,\infty}:=\vee_{t\ge 0} \check{\cF}_{n,t}$.
With slight abuse of notation, we write $T_j$, $E_j$, and $\mathcal{A}_j$ for
the canonical mappings on $\check{\Omega}_n$ and do not indicate their domains
(i.e., $\check{\Omega}_n$ vs.~$\check{\Omega}$) if there is no danger of confusion.

\section{Optimal control of path-dependent PDPs}\label{S:OptimalControl}

\subsection{The data}\label{S:Data1}
First, consider mappings 
\begin{align*}
&f: \R_+\times\Omega\times A\to \R^d,\quad
\lambda:\R_+\times\Omega\times A\to\R_+,\quad
Q:\R_+\times\Omega\times A\times\mathcal{B}(\R^d)\to \R_+,
\end{align*}
which are the characteristics for our PDP{\color{black}, i.e., $f$ specifes its deterministic flow, $\lambda$ the
distribution of its jump times, and $Q$ the distribution of its post-jump locations}. Next, consider mappings
\begin{align*}
\ell:\R_+\times\Omega\times A\to\R_+\qquad\text{and}\qquad h:\Omega\to\R_+,
\end{align*}
which will serve as running and terminal costs.

The following standing assumptions are in force throughout the rest of 
this paper. 

\begin{assumption}\label{A:data:controlled} Let $f$, $\lambda$, $\ell$, and $h$ satisfy the following conditions
(the same as in section~{\color{black}7} in \cite{BK1}).

(i) For a.e.~$t\in\R_+$, the map $(x,a)\mapsto (f,\lambda,\ell)(t,x,a)$, $\Omega\times A\to \R^d\times\R_+\times\R_+$,
is continuous, where $\Omega$ is considered  here to be equipped with 
the seminorm $x\mapsto \sup_{ s\le t} \abs{x(s\wedge t)}$.

(ii) For every $(x,a)\in\Omega\times A$, the map $t\mapsto (f,\lambda,\ell)(t,x,a)$, $\R_+\to \R^d\times\R_+\times\R_+$,
is Borel-measurable.

(iii)  There are constants $C_f$, $C_\lambda\ge 0$, 
such that, 
for every $(s,x)\in \R_+\times \Omega$,
 \begin{align*}
& \sup_{a\in A}\abs{f(s,x,a))}
\le C_f(1+\sup_{t\le s}\abs{x(t)}),\quad
  \sup_{a\in A} \lambda(t,x,a)\le C_\lambda,\qquad\text{and}\\
&\sup_{a\in A} \ell(s,x,a)+h(x)\le C_f.
\end{align*}

(iv) There is a constant $L_f\ge 0$ such that, for every $(s,x,\tilde{x})\in \R_+\times \Omega\times\Omega$,
\begin{align*}
&\sup_{a\in A}\left[
\abs{f(s,x,a)-f(s,\tilde{x},a)}+
\abs{\ell(s,x,a)-\ell(s,\tilde{x},a)}+
\abs{\lambda(s,x,a)-\lambda(s,\tilde{x},a)}\right]\\
&\qquad\le L_f\sup_{t\in [0,s]}\abs{x(t)-\tilde{x}(t)}\qquad\text{and}\\
& \abs{h(x)-h(\tilde{x})}\le L_f\sup_{t\le T}\abs{x(t)-\tilde{x}(t)}.
\end{align*}
\end{assumption}

\begin{assumption}\label{A:Q}
Let $Q$ satisfy the following conditions.

{\color{black}
(i) For every $B\in\mathcal{B}(\R^d)$,  $(s,x,a)\mapsto Q(s,x,a,B)$ 
is Borel-measurable and non-anticipating, i.e., $Q(s,x,a,B)=Q(s,x(\cdot\wedge s),a,B)$.}



(ii) For every $s\in\R_+$, $x\in\Omega$, and $a\in A$,  $Q(s,x,a,\cdot)\in\mathscr{P}(\R^d)$.

(iii) For every  $s\in\R_+$, $x\in\Omega$, and $a\in A$,   $Q(s,x,a,\{x(s)\})=0$.

(iv) There is a constant $L_Q\ge 0$,  such that, for all $\psi\in C_b([0,T]\times\Omega)$ and $L_\psi\ge 0$, 
the statement
\begin{align}\label{E:4AQ}
\abs{\psi(s,x)-\psi(s,\tilde{x})}\le L_\psi
{\color{black} \sup_{t\in [0,s]}\abs{x(t)-\tilde{x}(t)}}
\text{ for all $(s,x,\tilde{x})\in [0,T]\times\Omega\times\Omega$}
\end{align}
implies
\begin{align*}
&\sup_{a\in A}\Big|\int_{\R^d} \psi(s,x\otimes_s e)\,Q(s,x,a,\dd e)-
\int_{\R^d} \psi(s,x\otimes_s e)\,Q(s,\tilde{x},a,\dd e)
 \Big|\\
  &\qquad\le L_\psi\,L_Q\,\sup_{t\in [0,s]}\abs{x(t)-\tilde{x}(t)}
  \text{ for all $(s,x,\tilde{x},e)\in [0,T]\times\Omega\times\Omega\times \R^d$.}
\end{align*}

{\color{black} (v) For each $s\in\R_+$ and $\eta\in C_b(\R^d)$, the map $(x,a)\mapsto \int_{\R^d} \eta(e)\,Q(s,x,a,\dd e)$,
$\Omega\times A\to\R$, is continuous.}
\end{assumption}

\begin{remark}\label{R:AQ}
By Proposition~7.26 in \cite{BS},  $Q$ is a Borel-measurable stochastic kernel on $\R^d$ given $\R_+\times\Omega\times A$.
{\color{black} By Theorem 97 (b), 147-IV, in 
\cite{DMprobPot}, the maps  $(s,x)\mapsto Q(s,x,a,B)$, $(a,B)\in A\times\mathcal{B}(\R^d)$,
are $\F^0$-optional.}
\end{remark}

{\color{black} We provide now two elementary examples concerning Assumption~\ref{A:Q} (iv).
\begin{example}
Let $d=1$.

(i) Assume that
\begin{align*}
Q(s,x,a,\dd e)=\frac{1}{2} [\delta_{x(s)-1}(\dd e)+\delta_{x(s)+1}(\dd e)].
\end{align*}
Consider a function $\psi\in C_b([0,T]\times\Omega{\color{black})}$ 
and a constant $L_\psi\ge 0$ such that \eqref{E:4AQ} holds.
Then, for every $(s,x,\tilde{x})\in [0,T]\times\Omega\times\Omega$, we have
\begin{align*}
&\Big|\int_{\R^d} \psi(s,x\otimes_s e)\,Q(s,x,a,\dd e)-
\int_{\R^d} \psi(s,x\otimes_s e)\,Q(s,\tilde{x},a,\dd e)
 \Big|\\
 &=\frac{1}{2}\Big|
 \psi(s,x\otimes_s (x(s)-1))+ \psi(s,x\otimes_s (x(s)+1))
 \\&\qquad\qquad - \psi(s,x\otimes_s (\tilde{x}(s)-1))- \psi(s,x\otimes_s (\tilde{x}(s)+1))
 \Big|\\
 &\color{black}\le  L_\psi \abs{x(s)-\tilde{x}(s)}\\
 &\le L_\psi \sup_{t\in [0,s]}\abs{x(t)-\tilde{x}(t)},
\end{align*}
i.e., Assumption~\ref{A:Q} (iv) holds.

(ii) Assume that 
\begin{align*}
Q(s,x,a,\dd e)&=\begin{cases}\frac{1}{2} [\delta_{x(0)-1}(\dd e)+\delta_{x(0)+1}(\dd e)]
&\text{ if $x(s)\not\in\{x(0)-1,x(0)+1\}$,}\\
\delta_{x(0)}(\dd e)
&\text{ if $x(s)\in\{x(0)-1,x(0)+1\}$.}
\end{cases}
\end{align*}
Similarly as in part (i), one can show that Assumption~\ref{A:Q} (iv) holds.
Also note that Assumption~\ref{A:Q} (iii) holds as well.
\end{example}}

\subsection{The flow and related notation}\label{S:Flow}
Given $(s,x,\alpha)\in\R_+\times\Omega\in\mathcal{A}$, denote by $\phi^{s,x,\alpha}=\phi$ the solution of
\begin{equation}\label{E:flow}
\begin{split}
\phi^\prime(t)&=f(t,\phi,\alpha(t))\text{ a.e.~on $(s,\infty)$,}\\
\phi(t)&=x(t)\text{ on $[0,s]$.}
\end{split}
\end{equation}
Establishing existence and uniqueness  for the initial value problem~\eqref{E:flow}
is standard given Assumption~\ref{A:data:controlled}
{\color{black}(see, e.g., Theorem~16.3.11 in \cite{CohenElliott} for a fairly general result,
which covers our case)}. 
Note that $\phi^{s,x,\alpha}\in\Omega$ and
$\phi^{s,x,\alpha}\vert_{[s,\infty)}\in C([s,\infty),\R^d)$.

The following result is  implicitly needed throughout most of what follows.
Its proof can be found in the appendix, section \ref{S:A1}.
\begin{lemma}\label{L:phi:meas}
The map $(s,x,\alpha)\mapsto \phi^{s,x,\alpha}$, $\R_+\times\Omega\times\mathcal{A}\to\Omega$,
is measurable from $\mathcal{B}(\R_+)\otimes \cF^0\otimes\mathcal{B}(\mathcal{A})$ to $\cF^0$.
\end{lemma}

We shall also  make use of the survival function $F^{s,x,\alpha}$ defined by 
\begin{align*}
F^{s,x,\alpha}(t)&:=\bfone_{[s,\infty)}(t)\,\exp\Big(
-\int_s^t \lambda(r,\phi^{s,x,\alpha},\alpha(r))\,\dd r
\Big)+\bfone_{[0,s)}(t),\quad t\in\R_+.
\end{align*}

For each $(s,x,\alpha)\in \R_+\times\Omega\times\mathcal{A}$ and $t\in [s,\infty)$, also define 
\begin{align*}
\lambda^{s,x,\alpha}(t)&:=\lambda(t,\phi^{s,x,\alpha},\alpha(t)),\\
\chi^{s,x,\alpha}(t)&:=\exp\left(
-\int_s^t \lambda^{s,x,\alpha}(r)\,\dd r
\right),\\
\ell^{s,x,\alpha}(t)&:=\ell(t,\phi^{s,x,\alpha},\alpha(t)),\\
Q^{s,x,\alpha}(t,\dd e)&:=Q(t,\phi^{s,x,\alpha},\alpha(t),\dd e).
\end{align*}
If more convenient, then we  write
\begin{align*}
Q(t,\dd e\vert s,x,\alpha):=Q^{s,x,\alpha}(t,\dd e),\text{ etc.}
\end{align*}

\begin{remark}\label{R:coeff:meas}
Following  Lemma~\ref{L:phi:meas} and the arguments in the proofs in section~\ref{S:A1},
one can deduce that
the mappings  
\begin{align*}
&(s,x,\alpha,t)\mapsto  \lambda^{s,x,\alpha}(t), \R_+\times\Omega\times\mathcal{A}\times\R_+\to\R_+,\\
&(s,x,\alpha,t)\mapsto \chi^{s,x,\alpha}(t), \R_+\times\Omega\times\mathcal{A}\times\R_+\to\R_+,\\
&(s,x,\alpha,t)\mapsto \ell^{s,x,\alpha}(t), \R_+\times\Omega\times\mathcal{A}\times\R_+\to\R_+,\\
& (s,x,\alpha,t)\mapsto Q^{s,x,\alpha}(t,B), \R_+\times\Omega\times\mathcal{A}\times\R_+\to\R_+, B\in\mathcal{B}(\R^d),
\end{align*}
are $\mathcal{B}(\R_+)\otimes\cF^0\otimes\mathcal{B}(\mathcal{A})\otimes\mathcal{B}(\R_+)$-measurable. 

{\color{black}
To this end, one can start by assuming
 first that $\lambda=\lambda(t,x,a)$ is of the form $\lambda_1(t)\,\lambda_2(x(\cdot\wedge t))\,\lambda_3(a)$,
where $\lambda_1$ is (Borel) measurable and $\lambda_2$ as well as $\lambda_3$ are continuous.
Then, by Lemma~\ref{L:phi:meas}, $(s,x,\alpha,t)\mapsto \lambda(t,\phi^{s,x,\alpha},\alpha(t))$ is clearly
$\mathcal{B}(\R_+)\otimes\cF^0\otimes \mathcal{B}(\mathcal{A})\otimes\mathcal{B}(\R_+)$-measurable.
It remains to apply a monotone-class argument.
}
\end{remark}

\subsection{The continuous-time optimal control problem} 
 {\color{black} First, let us define  \textcolor{blue}{recursively} the
  ``canonical" piecewise deterministic process}   $X^{s,x}:\R_+\times\check{\Omega}\to \R^d\cup\{\Delta\}$
  {\color{black} starting at $(s,x)\in (0,\infty)\times\Omega$}   by 
\begin{align}\label{E:ContrState}
X^{s,x}(t,\check{\omega}):=\begin{cases}
\Delta, &t_0=\infty,\\
\phi^{s,x,\alpha_0}(t), &0\le t<t_1\text{ and } t_0<\infty,\\
\phi^{t_n,X^{s,x}(\check{\omega})\otimes_{t_n}\,e_n,\alpha_n}(t),
& t_n\le t<t_{n+1}\text{ and } t_0<\infty,
\end{cases}
\end{align}
for each $\check{\omega}=(t_n,e_n,\alpha_n)_{n\in\N_0}\in\check{\Omega}$.

\begin{remark}
Keep in mind that, by Remark~\ref{R:concatenation}, the  $X^{s,x}$-term on the right-hand
side of \eqref{E:ContrState} depends only on its values on $[0,t_n]$,
 i.e., our definition of $X^{s,x}$ is not circular.
Note that ${X}^{s,x}(\check{\omega})\in\Omega$ for each $\check{\omega}\in\check{\Omega}$ with $t_0<\infty$.
Also note that in \eqref{E:ContrState}, $t_{n+1}=\infty$ for some $n\in\N_0$ is possible and if, in addition,
$t_n<\infty$, then $X^{s,x}(t,\check{\omega})=\phi^{t_n,X^{s,x}(\check{\omega})\otimes_{t_n}\,e_n,\alpha_n}(t)$ for all $t\in [t_n,\infty)$. In particular, $X^{s,x}(t,(t_0,e_0,\alpha_0;\infty,\Delta,\Delta^\prime;
\infty,\Delta,\Delta^\prime;\ldots))=\phi^{s,x}(t)$ if $t_0<\infty$.
\end{remark}

\begin{remark}\label{R:Xsx}
The process $X^{s,x}$ is $\check{\F}$-adapted. This can be shown via mathematical induction
together with Proposition~4.2.1~(b)~(iii) in \cite{Jacobsen},
Theorem 96 (d), 146-IV, in 
\cite{DMprobPot}, 
and Lemma~\ref{L:phi:meas}.
Since in addition $X^{s,x}$ has c\`adl\`ag paths, it is also $\check{\F}$-progressively measurable.
\end{remark}
We also consider the following ``restrictions" of $X^{s,x}$. Given $n\in\N_0$,
define $X^{s,x}_n:\R_+\times\check{\Omega}_{n+1}\to \R^d\cup\{\Delta\}$ by
\begin{align}\label{E:Xn}
X^{s,x}_n(t,\check{\omega}_{n+1}):=
X^{s,x}(t,(\check{\omega}_{n+1};\infty,\Delta,\Delta^\prime;
\infty,\Delta,\Delta^\prime;\ldots)),
\end{align}
where $\check{\omega}_{n+1}=(t_j,e_j,\alpha_j)_{j=0}^n$.

{\color{black} Next, we introduce classes of policies needed to formulate our optimal control problem. Note
that we consider  randomized as well as non-randomized policies. We use randomized policies mainly
for technical reasons: Doing so will allow us to directly apply results from discrete-time stochastic optimal control in \cite{BS}.
We will show that for our problem optimization over randomized policies and optimization over non-randomized policies
are equivalent.
%
}

\begin{definition}\label{D:randomizedPolicy}
Denote by $\mathbb{A}^\prime$ the space of all policies $\mathbf{a}=(\mathbf{a}_0,\mathbf{a}_1,\ldots)$
such that each 
$\mathbf{a}_n=\mathbf{a}_n(\dd \alpha_n\vert (t_0,e_0),\ldots, (t_n,e_n); \alpha_0,\ldots,\alpha_{n-1})$, $n\in\N_0$, is a 
universally measurable 
stochastic kernel
on $\mathcal{A}\cup\{\Delta^\prime\}$ given $((\R_+\times \R^d)\cup\{(\infty,\Delta)\})^{n+1}\times(\mathcal{A}\cup\{\Delta^\prime\})^n$.

Denote by $\mathbb{A}$ the space of all non-randomized policies 
$\mathbf{a}=(\mathbf{a}_0,\mathbf{a}_1,\ldots)\in\mathbb{A}^\prime$, i.e., 
for each $n\in\N_0$, $(t_0,e_0)$, $\ldots$, $(t_n,e_n)\in
(\R_+\times \R^d)\cup\{\infty,\Delta\}$,
 and
$\alpha_0$, $\ldots$, $\alpha_{n-1}\in\mathcal{A}\cup\{\Delta^\prime\}$,  
$\mathbf{a}_n(\dd \alpha_n\vert (t_0,e_0),\ldots, (t_n,e_n); \alpha_0,\ldots,\alpha_{n-1})$ is
a Dirac measure.
\end{definition}

\begin{remark}
Our (randomized) policies should not be confused with relaxed controls, which, in our context, would be functions
from $[0,T]$ to $\mathscr{P}(A)$. 
\end{remark}

{\color{black} Given an initial condition and a specified policy, we are now in the position to define 
the corresponding probability measure.}

\begin{definition}\label{D:contr:checkP}
Let $(s,x)\in (0,T)\times\Omega$ and $\mathbf{a}\in\mathbb{A}^\prime$. 
We  construct  a probability measure $\check{\Prob}^{s,x,\mathbf{a}}$ on
 $(\check{\Omega},\check{\cF})$
with  corresponding
expected value $\check{\Mean}^{s,x,\mathbf{a}}$ 
 recursively by defining its marginals  
 $\check{\Prob}^{s,x,\mathbf{a}}_n$ on $(\check{\Omega}_n,\check{\cF}_{n,\infty})$ with corresponding
expected values $\check{\Mean}^{s,x,\mathbf{a}}_n$, $n\in\N$, as follows
 {\color{black} (cf.~section~3.2 of \cite{Jacobsen})}:
\begin{equation}\label{E:contr:MPP:construction}
\begin{split}
&\check{\Mean}_1^{s,x,\mathbf{a}}[
\psi_0(T_0,E_0,\mathcal{A}_0)
]:=\int_0^\infty \int_E
\int_{\mathcal{A}} \psi_0(t_0,e_0,\alpha_0)\,\mathbf{a}_0(\dd \alpha_0)\,\delta_{x(s)}(\dd e_0)\,
\delta_s(\dd t_0),\\
&\check{\Mean}_2^{s,x,\mathbf{a}}[\bfone_{\{T_1<\infty\}}
\psi_1(T_0,E_0,\mathcal{A}_0;
T_1,E_1,\mathcal{A}_1)]\\ &\,:=
-\int_0^\infty\int_E\int_{\mathcal{A}}\int_0^\infty \int_E \int_{\mathcal{A}}
\psi_1(t_0,e_0,\alpha_0;t_1,e_1,\alpha_1)\\
&\,\qquad\qquad\qquad \mathbf{a}_1(\dd \alpha_1\vert (t_0,e_0),(t_1,e_1);\alpha_0)\\
&\,\qquad\qquad\qquad 
Q^{s,x,\alpha_0}(t_1,\dd e_1)\,\dd F^{s,x,\alpha_0}(t_1)\,
\mathbf{a}_0(\dd \alpha_0)\,\delta_{x(s)}(\dd e_0)\,\delta_s(\dd t_0),\\
&\check{\Mean}_{n+2}^{s,x,\mathbf{a}}[\bfone_{\{T_{n+1}<\infty\}}\psi_{n+1}
(T_0,E_0,\mathcal{A}_0;
\ldots;
T_{n+1},E_{n+1},\mathcal{A}_{n+1})]
\\&\,:=
- \int_{\check{\Omega}_{n+1}}
\Biggl\{
\int_0^\infty\int_E\int_{\mathcal{A}}
\psi_{n+1}(\check{\omega}_{n+1};
t_{n+1},e_{n+1},\alpha_{n+1})\\
&\qquad\qquad
\mathbf{a}_{n+1}(\dd \alpha_{n+1}\vert (t_0,e_0),\ldots,
(t_{n+1},e_{n+1}); \alpha_0,\ldots,\alpha_n)
\\ &\qquad\qquad
Q(t_{n+1},\dd e_{n+1}\vert
t_n,X^{s,x}_n(\check{\omega}_{n+1}),
\alpha_n)\,
\dd 
F(t_{n+1}\vert 
t_n,X^{s,x}_n(\check{\omega}_{n+1}),
\alpha_n
)\Biggr\}\\
&\qquad\qquad\check{\Prob}^{s,x,\mathbf{a}}_{n+1}(\dd\check{\omega}_{n+1}),
\end{split}
\end{equation}
where  $\psi_n: (\R_+\times \R^d)^{n+1}\times\mathcal{A}^{n+1}\to\R$, $n\in\N_0$, are Borel-measurable and bounded
and the notation $\check{\omega}_{n+1}=(t_j,e_j,\alpha_j)_{j=0}^n$ is used.

{\color{black} Thanks to the Kolmogorov extension theorem, the above construction really provides
a unique probability measure (besides section~3.2 in \cite{Jacobsen}, we refer also to Proposition~7.45 and its proof in \cite{BS}
for a detailed treatment).}
\end{definition}

\begin{definition}\label{D:valuefunctions}
The expected cost $J:(0,T]\times\Omega\times\mathbb{A}^\prime\to\R$ and
 the value functions  $V^\prime:(0,T]\times\Omega\to\R$
and $V:(0,T]\times\Omega\to\R$ are defined by
\begin{equation}\label{E:Cost:and:Value}
\begin{split}
J(s,x,\mathbf{a})&:=
\check{\Mean}^{s,x,\mathbf{a}}\Biggl[\int_s^T\ell(t,X^{s,x},\bar{\alpha}(t))\,\dd t+h(X^{s,x})\Biggr],\\
V^\prime(s,x)&:=\inf_{\mathbf{a}\in\mathbb{A}^\prime} J(s,x,\mathbf{a}),\\
V(s,x)&:=\inf_{\mathbf{a}\in\mathbb{A}} J(s,x,\mathbf{a}).
\end{split}
\end{equation}
Here, $\bar{\alpha}:\R_+\times\check{\Omega}\to A\cup\{\Delta^\prime\}$ is given by
\begin{align*}
\bar{\alpha}(t):=\sum_{n=0}^\infty \bfone_{(T_n,T_{n+1}]}(t)\,\mathcal{A}_n(t)+\bfone_{\left(
\cup_{n=0}^\infty (T_n,T_{n+1}]
\right)^c}(t)\,a^\circ
\end{align*}
(recall $a^\circ$ from the beginning of section~\ref{SS:Canon:Marked}).
\end{definition}

\subsection{A related discrete-time optimal control problem and the dynamic programming principle}\label{SS:value}
{\color{black} We will construct a  
path-dependent discrete-time inf\-in\-ite-horizon decision model, which is closely related to our continuous-time problem.
This will enable us to establish a dynamic programming principle for the value functions
$V$ and $V^\prime$ of our continuous-time problem (see Definition~\ref{D:valuefunctions}).
(Our approach is inspired by  similar
 work in the Markovian case, see, e.g., section~3 of  \cite{BR10PDMP}, section~8.2 in \cite{BR11book}, and \cite{RTT17PDMP}.)

In this  {\color{black} subsection,} we  fix $(s,x)\in (0,T]\times\Omega$. }

{\color{black} We consider a path-dependent infinite-horizon model 
with the following data
(cf.~Definitions~8.1 and 9.1  in \cite{BS} for the non-path-dependent case):
\begin{itemize} 
\item State space: $S=(\R_+\times \R^d)\cup\{(\infty,\Delta)\}$.
\item  Control space: $U=\mathcal{A}\cup\{\Delta^\prime\}$.
\item Disturbance space: $S$.
\item Disturbance kernel 
\begin{align*}
\vec{Q}_n(\dd t\,\dd e\vert (t_0,e_0),\ldots,(t_n,e_n);\alpha_0,\ldots,\alpha_n)\text{  at stage $n$, $n\in\N_0$,}
\end{align*}
is a Borel-measurable  stochastic kernel on $S$  given
  $S^{n+1}\times U^{n+1}$ defined
by
\begin{align*}
&\int_S \bfone_{[0,\infty)}(t)\,\psi(t,e)\, 
\vec{Q}_n(\dd t\,\dd e\vert
(t_0,e_0),\ldots,(t_n,e_n);\alpha_0,\ldots,\alpha_n)\\&:=\begin{cases}
&-\int_{\R_+\times E} 
\psi(t,e)\,
Q(t,\dd e\vert t_n,X^{s,x}_n(\check{\omega}_{n+1}),\alpha_n)\,
\dd F(t\vert  t_n,X^{s,x}_n(\check{\omega}_{n+1}),\alpha_n)\\
&\phantom{\psi(\infty,\Delta)}\qquad\text{if $\check{\omega}_{n+1}=(t_j,e_j,\alpha_j)_{j=0}^n\in\check{\Omega}_{n+1}$,}\\
&\psi(\infty,\Delta)\qquad\text{otherwise,}
\end{cases}
\end{align*}
for every bounded Borel-measurable function $\psi:S\to\R$. 
\item Cost function
\begin{align*}
g_n:S^{n+1}\times U^{n+1}\to\R_+\cup\{\infty\} \quad \textup{at stage $n$, $n\in\N_0$,}
\end{align*}
is a lower semi-analytic function defined by
\begin{equation*}\label{3E:vecln}
g_n((t_0,e_0),\ldots,(t_n,e_n);\alpha_0,\ldots,\alpha_n)
:=\begin{cases} \vec{\ell}_n((t_j,e_j,\alpha_j)_{j=0}^n)\text{ if $(t_j,e_j,\alpha_j)_{j=0}^n\in\check{\Omega}_{n+1}$,}\\
\infty\text{ otherwise,}
\end{cases}
\end{equation*}
where  $\vec{\ell}_n:\check{\Omega}_{n+1} \to\R_+$ is defined by 
\begin{equation}\label{E:vecln}
\begin{split}
\vec{\ell}_n(\check{\omega}_{n+1}) 
&:=\int_{t_n\wedge T}^T \chi^{t_n,X^{s,x}_n(\check{\omega}_{n+1}),\alpha_n}(t)\,
\ell(t,X^{s,x}_n(\check{\omega}_{n+1}), \alpha_n(t))\,\dd t\\
&\qquad +\bfone_{[0,T]}(t_n)\, \chi^{t_n,X^{s,x}_n(\check{\omega}_{n+1}),\alpha_n}(T)\,h(X^{s,x}_n(\check{\omega}_{n+1})).
\end{split}
\end{equation}
Here, $\check{\omega}_{n+1}= (t_0,e_0,\alpha_0;\ldots;t_n,e_n,\alpha_n)$.
\end{itemize}

\begin{remark}
Borel-measurability of  $\vec{Q}_n$ and lower semi-analyticity of $g_n$, 
$n\in\N_0$,
follow from Remarks~\ref{R:AQ}, \ref{R:coeff:meas}, and \ref{R:Xsx}. Keep in mind that  $(s,x)$ is fixed here.
\end{remark}

Depending on the starting stage $k\in\N_0$, 
we use as policy space the set $\mathbb{A}^k$ of all
policies  $\mathbf{a}^k=(\mathbf{a}^k_k,\mathbf{a}^k_{k+1},\ldots)$ such that
each 
\begin{align*}
 \mathbf{a}^k_j=\mathbf{a}^k_j(\dd\alpha_j\vert (t_0,e_0),\ldots,(t_j,e_j);\alpha_0,\ldots,\alpha_{j-1}),\quad
j\in\{k,k+1,\ldots\},
\end{align*}
 is a universally measurable stochastic kernel on $U$ given $S^{j+1}\times U^j$
 (\cite[section~10.1]{BS}).
 
\begin{remark}
Note that $\mathbb{A}^0$ coincides with $\mathbb{A}^\prime$ from Definition~\ref{D:randomizedPolicy}.
\end{remark}

Given $k\in\N_0$, a policy $\mathbf{a}^k=(\mathbf{a}^k_k,\mathbf{a}^k_{k+1},\ldots)\in\mathbb{A}^k$, and
a probability measure $p^k$ on $S^{k+1}\times U^k$,  define marginals $r^k_N(\mathbf{a}^k,p^k)$ on $(S\times U)^N$,
$N\in\{k+1,k+2,\ldots\}$ by
\begin{equation}\label{E:new:marginals}
\begin{split}
&\int\psi\,\dd r^k_N(\mathbf{a}^k,p^k)\\
&:=\int_{S^{k+1}\times U^k}\int_U\int_S\cdots\int_S\int_U 
\psi((t_0,e_0),\ldots,(t_{N-1},e_{N-1});\alpha_0,\ldots,\alpha_{N-1})\\
&\qquad\qquad \mathbf{a}^k_{N-1}(\dd \alpha_{N-1}\vert
 (t_0,e_0),\ldots,(t_{N-1},e_{N-1});\alpha_0,\ldots,\alpha_{N-2}))\\
&\qquad\qquad \vec{Q}_{N-1}(\dd t_{N-1}\,\dd e_{N-1}\vert
(t_0,e_0),\ldots,(t_{N-2},e_{N-2});\alpha_0,\ldots,\alpha_{N-2})\cdots \\
&\qquad\qquad\vec{Q}_{k}(\dd t_{k+1}\,\dd e_{k+1}\vert
(t_0,e_0),\ldots,(t_{k},e_{k});\alpha_0,\ldots,\alpha_{k})\\
&\qquad\qquad \mathbf{a}^k_k(\dd \alpha_k\vert
 (t_0,e_0),\ldots,(t_k,e_k);\alpha_0,\ldots,\alpha_{k-1})\\
&\qquad\qquad p^k(\dd t_0\,\dd e_0\ldots \dd t_{k}\,\dd e_{k}\,\dd \alpha_0\ldots\dd \alpha_{k-1})
\end{split}
\end{equation}
for each bounded universally measurable function $\psi:S^N\times U^N\to\R$.
}

{\color{black} For each $k\in\N_0$, $\mathbf{a}^k\in\mathbb{A}^k$, $(t_0,e_0)$, $\ldots$, $(t_k,e_k)\in S$,
and $\alpha_0$, $\ldots$, $\alpha_{k-1}\in U$, define the corresponding cost
\begin{equation}\label{0E:OriginalModel:Cost:kOriginating}
\begin{split}
&J_{\mathbf{a}^k}(k;(t_0,e_0),\ldots,(t_k,e_k);\alpha_0,\ldots,\alpha_{k-1})\\ &:= \sum_{j=k}^\infty \int g_j
\,\dd r_{j+1}^k(\mathbf{a}^k,\delta_{((t_0,e_0),\ldots,(t_k,e_k);\alpha_0,\ldots,\alpha_{k-1})}).
\end{split}
\end{equation}

For each $k\in\N_0$,  $(t_0,e_0)$, $\ldots$, $(t_k,e_k)\in S$,
and $\alpha_0$, $\ldots$, $\alpha_{k-1}\in U$, define the  corresponding optimal cost
\begin{equation}\label{E:new:optimal:cost}
\begin{split}
&J^\ast(k;(t_0,e_0),\ldots,(t_k,e_k);\alpha_0,\ldots,\alpha_{k-1})\\
&:=\inf_{\mathbf{a}^k\in\mathbb{A}^k} J_{\mathbf{a}^k}(k;(t_0,e_0),\ldots,(t_k,e_k);\alpha_0,\ldots,\alpha_{k-1}).
\end{split}
\end{equation}

Next, we show that, for appropriate values, the just defined optimal cost function $J^\ast$ coincides 
with the value function $V^\prime$ defined by \eqref{E:Cost:and:Value}. 

\begin{theorem}\label{0L:value:DPP}
Fix $(s,x)\in (0,T)\times\Omega$ and
$(t,e,\alpha)\in (s,T]\times \R^d\times\mathcal{A}$.  Then
\begin{align}\label{0E:T:value:DPP:VJ}
V^\prime(s,x)&=J^\ast(0;s,x(s)),\\ \label{0E:L:value:DPP}
V^\prime(t,\phi^{s,x,\alpha}\otimes_t e)&= J^\ast(1;(s,x(s)),(t,e);\alpha).
\end{align}
\end{theorem}
\begin{proof}
(i) First, we establish \eqref{0E:T:value:DPP:VJ}. To this end, let $\mathbf{a}\in\mathbb{A}^\prime$. Then, by \eqref{E:Cost:and:Value},
\begin{align*}
J(s,x,\mathbf{a})&=
\sum_{n=0}^\infty \check{\Mean}^{s,x,\mathbf{a}}\Bigl[
\int_{T_n}^{T_{n+1}} \bfone_{[0,T]}(\tau)\,\ell(\tau,X^{s,x},\mathcal{A}_n(\tau))\,\dd \tau+
\bfone_{[T_n,T_{n+1})} (T)\,h(X^{s,x})
\Bigr].
\end{align*}
Since, by Assumption~\ref{A:data:controlled}, $h(x)=h(x(\cdot\wedge T))$ for every $x\in\Omega$, 
we have 
\begin{align*}
{\color{black}
\bfone_{(T,\infty]}(T_{n+1})\,h(X^{s,x}(\check{\omega}))=\bfone_{(T,\infty]}(T_{n+1})\,h(X^{s,x}_n(\check{\omega}_{n+1}))
}
\end{align*}
 for 
 every $\check{\omega}=(\check{\omega}_{n+1},(t_j,e_j,\alpha_j)_{j=n+1}^\infty)\in\check{\Omega}$
and every $n\in\N_0$ (to see this, recall \eqref{E:Xn} for the definition of $X^{s,x}_n$ and Remark~\ref{R:Xsx}).
Thus, for each $n\in\N_0$, 
\begin{align*}
&\check{\Mean}^{s,x,\mathbf{a}}\Bigl[
\int_{T_n}^{T_{n+1}} \bfone_{[0,T]}(\tau)
\,\ell(\tau,X^{s,x},\mathcal{A}_n(\tau))\,\dd \tau+
\bfone_{[T_n,T_{n+1})} (T)\,h(X^{s,x})\Bigr]\\
&=\check{\Mean}^{s,x,\mathbf{a}}_{n+2}
\Bigl[
\int_{T_n\wedge T}^T \bfone_{(\tau,\infty]}(T_{n+1})
\ell(\tau,X_n^{s,x},\mathcal{A}_n(\tau))\,\dd \tau+
\bfone_{[0,T]} (T_n)
\bfone_{(T,\infty]}(T_{n+1})\,h(X_n^{s,x})\Bigr]\\
&= \int_{\check{\Omega}_{n+1}}\Big(
\int_{t_n\wedge T}^T \chi^{t_n,X^{s,x}_n(\check{\omega}_{n+1}),\alpha_n}(\tau)\,
\ell(\tau,X^{s,x}_n(\check{\omega}_{n+1}),a_n(\tau))\,\dd \tau\\
&\qquad\qquad\qquad
+ \bfone_{[0,T]}(t_n)\,
\chi^{t_n,X^{s,x}_n(\check{\omega}_{n+1}),\alpha_n}(T)\,h(X^{s,x}_n(\check{\omega}_{n+1}))\Big)\,
\check{\Prob}^{s,x,\mathbf{a}}_{n+1}(\dd\check{\omega}_{n+1}).
\end{align*}
Thus, recalling the definition of $\vec{\ell}_n$ in \eqref{E:vecln}, we have
\begin{align}\label{0E:Jsxa=sumln}
J(s,x,\mathbf{a})
=\sum_{n=0}^\infty  \check{\Mean}_{n+1}^{s,x,\mathbf{a}} \left[\vec{\ell}_n\right].
\end{align}
By the definitions of the marginals $r^k_N(\mathbf{a}^k,p)$ in \eqref{E:new:marginals}
and $\check{\Prob}^{s,x,\mathbf{a}}_n$ in \eqref{E:contr:MPP:construction},
\begin{align*}
\int g_0\,\dd r^0_1\left(\mathbf{a},\delta_{(s,x(s)}\right)&=
\int_{S\setminus(\{0\}\times E)} \int_U  \vec{\ell}_0(t_0,e_0,\alpha_0)\,\mathbf{a}_0(\dd  \alpha_0)\,\delta_{(s,x(s))}(\dd t_0\,\dd e_0)\\
&\qquad +\infty.\delta_{(s,x(s))}(\{0\}\times E)\qquad\text{(note that $s>0$)}\\
&=\check{\Mean}^{s,x,\mathbf{a}}_1\left[\vec{\ell}_0\right],\\
\int g_1\,\dd r^0_2\left(\mathbf{a},\delta_{(s,x(s))}\right)&=
\int_S\int_U\int_S\int_U g_1((t_0,e_0),(t_1,e_1);\alpha_0,\alpha_1)\\
&\qquad\mathbf{a}_1(\dd\alpha_1\vert (t_0,e_0),(t_1,e_1);\alpha_0)\,
\vec{Q}_0(\dd t_1\,\dd e_1\vert (t_0,e_0),\alpha_0)\\
&\qquad \mathbf{a}_0(\dd  \alpha_0)\,\delta_{(s,x(s))}(\dd t_0\,\dd e_0)\\
&=\check{\Mean}^{s,\mathbf{a}}_2\left[\vec{\ell}_1\right]\qquad\text{because 
$\left[r^0_2\left(\mathbf{a},\delta_{(s,x(s))}\right)\right](\check{\Omega}_2)=1$,}\\
\int g_N \,\dd r^0_{N+1}\left(\mathbf{a},\delta_{(s,x(s))}\right)&=
\check{\Mean}^{s,x,\mathbf{a}}_{N+1}\left[\vec{\ell}_N\right]\qquad\text{(can be shown by induction)}. 
\end{align*}
Hence, by \eqref{0E:Jsxa=sumln} and \eqref{0E:OriginalModel:Cost:kOriginating},   we have
\begin{align}\label{0E:J:equal:J}
J(s,x,\mathbf{a})=J_{\mathbf{a}}(0;s,x(s)).
\end{align}
Since $\mathbf{a}$ was an arbitrary policy in $\mathbb{A}^\prime=\mathbb{A}^0$,
we can conclude that  
\eqref{0E:T:value:DPP:VJ} holds.

(ii) Now, we prove \eqref{0E:L:value:DPP}.
Let
\begin{align*}
y:=\phi^{s,x,\alpha}\otimes_t  e
\end{align*} 
and note that, for every $\mathbf{a}\in\mathbb{A}^\prime$,
we can deduce in the same way as \eqref{0E:Jsxa=sumln} that
\begin{align}\label{0E:Jtya:sum}
J(t,y, 
\mathbf{a})
=\sum_{n=0}^\infty
\check{\Mean}^{t,y, 
\mathbf{a}}_{n+1}\left[
\vec{\ell^1_n}
\right],
\end{align}
where the functions $\vec{\ell_n^1}:\check{\Omega}_{n+1}\to\R_+$, $n\in\N_0$,
are defined by
\begin{equation}\label{0E:vecln1}
\begin{split}
\vec{\ell_n^1}(\check{\omega}_{n+1})&:=
\int_{t_n\wedge T}^T \chi^{t_n,X^{t,y}_n(\check{\omega}_{n+1}),\alpha_n}(\tau)\,
\ell(t,X^{t,y}_n(\check{\omega}_{n+1}), \alpha_n(\tau))\,\dd \tau\\
&\qquad +\bfone_{[0,T]}(t_n)\, \chi^{t_n,X^{t,y}_n(\check{\omega}_{n+1}),\alpha_n}(T)\,
h(X^{t,y}_n(\check{\omega}_{n+1}))
\end{split}
\end{equation}
for each $\check{\omega}_{n+1}=(t_0,e_0,\alpha_0;\ldots;t_n,e_n,\alpha_n)\in
 \check{\Omega}_{n+1}$.

First note that, for every $\mathbf{a}^1\in\mathbb{A}^1$,
\begin{align*}
\int g_1\,\dd r^1_2(\mathbf{a}^1,\delta_{
((s,x(s)),(t,e);\alpha)
})&=
\int_{S^2\times U}\int_U
g_1((t_0,e_0),(t_1,e_1);\alpha_0,\alpha_1)\\
&\qquad\qquad\mathbf{a}^1_1(\dd\alpha_1\vert (t_0,e_0),(t_1,e_1);\alpha_0)\\
&\qquad\qquad\delta_{
((s,x(s)),(t,e);\alpha)
}(\dd t_0\,\dd e_0\,\dd t_1\,\dd e_1\,\dd \alpha_0)\\
&=\int_U \vec{\ell}_1((s,x(s),\alpha),(t,e,\alpha_1))\\
&\qquad\qquad\mathbf{a}^1_1(\dd\alpha_1\vert (s,x(s)),(t,e);\alpha).
\end{align*}
By the definition of $\vec{\ell}_1$ in \eqref{E:vecln}, the terms
$\vec{\ell}_1((s,x(s),\alpha),(t,e,\alpha_1))$ depend on
\begin{align*}
X_1^{s,x}(s,x(s),\alpha;t,e,\alpha_1)\overset{\text{\eqref{E:ContrState}},\eqref{E:Xn}}{{\color{black}=}} 
\phi^{t,\phi^{s,x,\alpha}\otimes_t\,e,\alpha_1}\overset{\text{\eqref{E:ContrState}},\eqref{E:Xn}}{{\color{black}=}}
X_0^{t,y}(t,y(t),\alpha_1).
\end{align*}
Together with the definition of $\vec{\ell_0^1}$ in \eqref{0E:vecln1}, we obtain
\begin{align*}
&\int g_1\,\dd r^1_2(\mathbf{a}^1,\delta_{
((s,x(s)),(t,e);\alpha)
})\\&=
\int_U \vec{\ell_0^1}(t,y(t),\alpha_1)\,
\mathbf{a}^1_1(\dd\alpha_1\vert (s,x(s)),(t,e);\alpha)\\
&=
\int_U \vec{\ell_0^1}(t,y(t),\alpha_0)\,
\mathbf{a}^1_1(\dd\alpha_0\vert (s,x(s)),(t,e);\alpha)\\
&\overset{\text{\eqref{E:contr:MPP:construction}}}{{\color{black}=}}
\check{\Mean}_1^{t,y,\mathbf{a}}\left[\vec{\ell_0^1}\right]
\end{align*}
with $\mathbf{a}=(\mathbf{a}_0,\mathbf{a}_1,\ldots)\in\mathbb{A}^\prime$ such
that 
\begin{align*}
\mathbf{a}_0(\dd\alpha_0\vert t_0,e_0)=
\mathbf{a}^1_1(\dd\alpha_0\vert (s,x(s)),(t_0,e_0);\alpha).
\end{align*}

Next, let $N\in\{2,3,\ldots\}$ and $\mathbf{a}^1\in\mathbb{A}^1$. Then 
\begin{align*}
&\int g_N\,\dd r^1_{N+1}(\mathbf{a}^1,\delta_{
((s,x(s)),(t,e);\alpha)
})\\&=\int_U\int_S\int_U\cdots\int_S\int_U
g_N((s,x(s)),(t,e),(t_2,e_2),\ldots,(t_N,e_N);\alpha,\alpha_1,\ldots,\alpha_N)\\
&\qquad
\mathbf{a}^1_N(\dd \alpha_N\vert
(s,x(s)),(t,e),(t_2,e_2),\ldots,(t_N,e_N);\alpha,\alpha_1,\ldots,\alpha_{N-1}
)\\
&\qquad \vec{Q}_{N-1}(\dd t_N\,\dd e_N\vert
(s,x(s)),(t,e),(t_2,e_2)\,\ldots,(t_{N-1},e_{N-1});\alpha,\alpha_1,\ldots,\alpha_{N-1}
)\cdots\\
&\qquad\mathbf{a}^1_2(\dd\alpha_2\vert
(s,x(s)),(t,e),(t_2,e_2);\alpha,\alpha_1
)\\
&\qquad \vec{Q}_1(\dd t_2\,\dd e_2\vert (s,x(s)),(t,e);\alpha,\alpha_1)\,
\mathbf{a}^1_1(\dd\alpha_1\vert (s,x(s)),(t,e);\alpha).
\end{align*}
\normalsize
Since (up to an $r^1_{N+1}(\mathbf{a}^1,\delta_{
((s,x(s)),(t,e);\alpha)
})$-null set)
\begin{equation}\label{00E:L:value:DPP}
\begin{split}
&g_N((s,x(s)),(t,e),(t_2,e_2),\ldots,(t_N,e_N);\alpha,\alpha_1,\ldots,\alpha_N)\\
&\qquad =\vec{\ell}_N(
(s,x(s),\alpha),(t,e,\alpha_1),(t_2,e_2,\alpha_2),\ldots,(t_N,e_N,\alpha_N)
)
\end{split}
\end{equation}
and the right-hand side of \eqref{00E:L:value:DPP} 
depends on
\begin{align*}
X^{s,x}_N(
s,x(s),\alpha;t,e,\alpha_1;t_2,e_2,\alpha_2;\ldots;t_N,e_N,\alpha_N
),
\end{align*}
which equals
\begin{align*}
X^{t,y}_{N-1}(t,y(t),\alpha_1;t_2,e_2,\alpha_2;\ldots;t_N,e_N,\alpha_N),
\end{align*}
we have (again up to an $r^1_{N+1}(\mathbf{a}^1,\delta_{
((s,x(s)),(t,e);\alpha)
})$-null set)
\begin{align*}
&g_N((s,x(s)),(t,e),(t_2,e_2),\ldots,(t_N,e_N);\alpha,\alpha_1,\ldots,\alpha_N)\\
&\qquad = 
\overrightarrow{\ell^1_{N-1}}((t,y(t),\alpha_1);(t_2,e_2,\alpha_2);\ldots;(t_N,e_N,\alpha_N)).
\end{align*}
Thus 
\begin{align*}
&\int g_N\,\dd r^1_{N+1}(\mathbf{a}^1,\delta_{
((s,x(s)),(t,e);\alpha)
})\\&=\int_U\int_S\int_U\cdots\int_S\int_U
\overrightarrow{\ell^1_{N-1}}((t,y(t),\alpha_1);(t_2,e_2,\alpha_2);\ldots;(t_N,e_N,\alpha_N))\\
&\,\,\,\,
\mathbf{a}^1_N(\dd \alpha_N\vert
(s,x(s)),(t,y(t)),(t_2,e_2),\ldots,(t_N,e_N);\alpha,\alpha_1,\ldots,\alpha_{N-1}
)\\
&\,\,\,\,
Q(t_N,\dd e_N\vert
t_{N-1},
\underbrace{
X^{s,x}_{N-1}(
s,x(s),\alpha;t,e,\alpha_1;t_2,e_2,\alpha_2;\ldots;t_{N-1},e_{N-1},\alpha_{N-1}
)
}_{=X^{t,y}_{N-2}(t,y(t),\alpha_1;t_2,e_2,\alpha_2;\ldots;t_{N-1},e_{N-1},\alpha_{N-1})},
\alpha_{N-1})\\
&\,\,\,\, [-\dd F(t_N\vert t_{N-1},
X^{t,y}_{N-2}(t,y(t),\alpha_1;t_2,e_2,\alpha_2;\ldots;t_{N-1},e_{N-1},\alpha_{N-1}),
\alpha_{N-1})]\cdots\\
&\,\,\,\,\mathbf{a}^1_2(\dd\alpha_2\vert
(s,x(s)),(t,y(t)),(t_2,e_2);\alpha,\alpha_1
)\\
&\,\,\,\, \underbrace{Q(t_2,\dd e_2\vert t,\overbrace{
X^{s,x}_1(s,x(s),\alpha;t,e,\alpha_1)
}^{
=X^{t,y}_0(t,y(t),\alpha_1)=\phi^{t,y,\alpha_1}
},\alpha_1)}_{
=Q^{t,y,\alpha_1}(t_2,\dd e_2)
}\,
[-\dd F^{t,y,\alpha_1}(t_2)]\\
&\,\,\,\,\mathbf{a}_1^1(\dd\alpha_1\vert(s,x(s)),(t,y(t));\alpha)\\
&=^{(\text{after relabeling of variables})}\\
&\,\,\,\,\int_U\int_S\int_U\cdots\int_S\int_U
\overrightarrow{\ell^1_{N-1}}((t,y(t),\alpha_0);(t_1,e_1,\alpha_1);
\ldots;(t_{N-1},e_{N-1},\alpha_{N-1}))\\
&\,\,\,\,
\mathbf{a}^1_{N}(\dd \alpha_{N-1}\vert
(s,x(s)),(t,y(t)),(t_1,e_1),\ldots,(t_{N-1},e_{N-1});\alpha,\alpha_0,\ldots,\alpha_{N-2}
)\\
&\,\,\,\,
Q(t_{N-1},\dd e_{N-1}\vert
t_{N-2},
X^{t,y}_{N-2}(t,y(t),\alpha_0;t_1,e_1,\alpha_1;\ldots;t_{N-2},e_{N-2},\alpha_{N-2}),
\alpha_{N-2})\\
&\,\,\,\, [-\dd F(t_{N-1}\vert t_{N-2},
X^{t,y}_{N-2}(t,y(t),\alpha_0;t_1,e_1,\alpha_1;\ldots;t_{N-2},e_{N-2},\alpha_{N-2}),
\alpha_{N-2})]\cdots\\
&\,\,\,\,\mathbf{a}^1_2(\dd\alpha_1\vert
(s,x(s)),(t,y(t)),(t_1,e_1);\alpha,\alpha_0
)\, 
Q^{t,y,\alpha_0}(t_1,\dd e_1)\,
[-\dd F^{t,y,\alpha_0}(t_1)]\\
&\,\,\,\,\mathbf{a}_1^1(\dd\alpha_0\vert(s,x(s)),(t,y(t));\alpha)\\
&=\check{\Mean}_N^{t,y,\mathbf{a}}\Bigl[
\overrightarrow{\ell^1_{N-1}}
\Bigr]
\end{align*}
\normalsize
with $\mathbf{a}=(\mathbf{a}_0,\mathbf{a}_1,\ldots)\in\mathbb{A}^\prime$ such that,
for each $n\in\N$,
\begin{align*}
&\mathbf{a}_{n-1}(\dd\alpha_{N-1}\vert
(t_0,e_0),\ldots,(t_{N-1},e_{N-1});\alpha_0,\ldots,\alpha_{N-1})\\
&\qquad =\mathbf{a}^1_n(\dd\alpha_{N-1}\vert
(s,x(s)),(t_0,e_0),\ldots,(t_{N-1},e_{N-1});\alpha,\alpha_0,\ldots,\alpha_{N-1}).
\end{align*}
Similarly, one can show that, for every $\mathbf{a}\in\mathbb{A}^\prime$, there exists
an $\mathbf{a}^1\in\mathbb{A}^1$ such that 
\begin{align*}
\check{\Mean}_N^{t,y,\mathbf{a}}\Bigl[
\overrightarrow{\ell^1_{N-1}}
\Bigr]=\int g_N\,\dd r^1_{N+1}(\mathbf{a}^1,\delta_{
((s,x(s)),(t,e);\alpha)}).
\end{align*}

Consequently, we have
\begin{align*}
V^\prime(t,\phi^{s,x,\alpha}\otimes_t e)&=\inf_{\mathbf{a}\in\mathbb{A}^\prime}  J(t,y,\mathbf{a})
\overset{\text{\eqref{0E:Jtya:sum}}}{{\color{black}=}}
\inf_{\mathbf{a}\in\mathbb{A}^\prime} 
\sum_{n=0}^\infty
\check{\Mean}^{t,y, 
\mathbf{a}}_{n+1}\left[
\vec{\ell^1_n}
\right]\\
&=\inf_{\mathbf{a}^1\in\mathbb{A}^1} \sum_{n=0}^\infty
\int g_{n+1}\,\dd r^1_{n+2}(\mathbf{a}^1,\delta_{
((s,x(s)),(t,e);\alpha)
})\\
&\overset{\text{\eqref{0E:OriginalModel:Cost:kOriginating}}}{{\color{black}=}}
\inf_{\mathbf{a}^1\in\mathbb{A}^1} J_{\mathbf{a}^1}(1;(s,x(s)),(t,e);\alpha)\\
&\overset{\text{\eqref{E:new:optimal:cost}}}{{\color{black}=}}
J^\ast(1;(s,x(s)),(t,e);\alpha).
\end{align*}
This concludes 
the proof.
\end{proof}

}

{\color{black} 
\begin{theorem}\label{0T:OneStepDPP}
Let $(t_0,e_0)\in S$.  Then \small
\begin{align*}\label{0E:T:OneSteppDPP}
J^\ast(0;t_0,e_0)=\inf_{\alpha\in U} \Bigl\{
g_0((t_0,e_0);\alpha)+\int_S  J^\ast(1;(t_0,e_0),(t_1,e_1);\alpha) \,\vec{Q}_0(\dd t_1\,\dd e_1\vert (t_0,e_0);\alpha)
\Bigr\}.
\end{align*}
\normalsize
\end{theorem}
\begin{proof}
This result is a specific instance of Theorem~\ref{T:OneStepDPP}, whose proof is provided
in appendix~\ref{S:Appendix:MDP}. 
\end{proof}

\begin{theorem}\label{0T:value:DPP}
Let $(s,x)\in (0,T]\times\Omega$. Then 
\begin{equation}\label{0E:T:value:DPP}
\begin{split}
V^\prime(s,x)&=\inf_{a\in\mathcal{A}} \Bigl\{\chi^{s,x,a}(T)\, h(\phi^{s,x,a}) + \int_s^T \Bigl(\chi^{s,x,a}(t)\,\ell^{s,x,a}(t)
\\& \qquad+\int_{\R^d}\lambda^{s,x,a}(t)\chi^{s,x,a}(t) V^\prime(t,\phi^{s,x,a}\otimes_t\,e)\,Q^{s,x,a}(t,\dd e)
\Bigr)\,\dd t
\Bigr\}\\
&=:[G^\circ V^\prime](s,x).
\end{split}
\end{equation}
\end{theorem}
\begin{proof}
By Theorem~\ref{0T:OneStepDPP} (below the first $\Delta^\prime$ can be omitted because we have $\alpha_0=\Delta^\prime$ if and only if $t_0=\infty$ but
$t_0=s\in (0,T]$) and by noting that the running cost 
vanishes if $t>T$ due to \eqref{E:vecln}, we have
\begin{align*}
&J^\ast(0;s,x(s))=\inf_{\alpha\in \mathcal{A}\cup\{\Delta^\prime\}} \Bigl\{
g_0((s,x(s));\alpha)\\&\qquad\qquad+\int_{
(\R_+\times E)\cup\{(\infty,\Delta)\}
}  J^\ast(1;(s,x(s)),(t,e);\alpha) \,\vec{Q}_0(\dd t\,\dd e \vert (s,x(s));\alpha)
\Bigr\}\\
&=\inf_{\alpha\in \mathcal{A}} \Bigl\{
g_0((s,x(s));\alpha)\\&\qquad\qquad+\int_{
[0,T]\times E} 
 J^\ast(1;(s,x(s)),(t,e);\alpha) \,\vec{Q}_0(\dd t\,\dd e \vert (s,x(s));\alpha)
\Bigr\}.
\end{align*}
Together  with 
Theorem~\ref{0L:value:DPP},
we get \eqref{0E:T:value:DPP}.
\end{proof}

For the next theorem, recall the set $\mathbb{A}$ of non-randomized policies
from Definition~\ref{D:randomizedPolicy}.

\begin{theorem}\label{0T:AppendixB:NonRandomized}
Let $(t_0,e_0)\in S$.  Then
$J^\ast(0;t_0,e_0)=\inf_{\mathbf{a}\in\mathbb{A}} J_{\mathbf{a}}(0;t_0,e_0)$.
\end{theorem}
\begin{proof}
This result is a specific instance of Theorem~\ref{T:AppendixB:NonRandomized}, whose proof is provided
in appendix~\ref{S:Appendix:MDP}. 
\end{proof}

\begin{theorem}\label{0T:value:value}
$V^\prime=V$.
\end{theorem}
\begin{proof}
Let $(s,x)\in (0,T]\times\Omega$.
By Theorem~\ref{0T:AppendixB:NonRandomized}, $J^\ast(0;s,x(s))=\inf_{\mathbf{a}\in\mathbb{A}} J_\mathbf{a}(0;s,x(s))$.
Together  with Theorem~\ref{0L:value:DPP} and \eqref{0E:J:equal:J}, we can deduce that
 $V^\prime(s,x)=V(s,x)$.
\end{proof}

Given Theorems~\ref{0T:value:DPP} and \ref{0T:value:value}, the proof of 
the following dynamic programming principle is standard (see, e.g., Proposition~III.2.5 in
\cite{BardiCapuzzoDolcetta}).
\begin{cor}\label{0C:value:DPP}
Let $(s,x)\in (0,T]\times\Omega$ and $s_1\in [s,T]$ . Then
\begin{equation}\label{0E:C:value:DPP}
\begin{split}
V(s,x)&=\inf_{a\in\mathcal{A}} \Bigl\{\chi^{s,x,a}(s_1)\, V(s_1,\phi^{s,x,a}) + \int_s^{s_1} \Bigl(\chi^{s,x,a}(t)\,\ell^{s,x,a}(t)
\\& \qquad+\int_{\R^d}\lambda^{s,x,a}(t)\chi^{s,x,a}(t) V(t,\phi^{s,x,a}\otimes_t\,e)\,Q^{s,x,a}(t,\dd e)
\Bigr)\,\dd t
\Bigr\}.
\end{split}
\end{equation}
\end{cor}

}

\section{Non-local  path-dependent Hamilton--Jacobi--Bellman equations}\label{S:HJB}
In this section,  the following additional assumption is in force. This will allow us to
use the results of section~7 in  {\color{black}our companion paper} \cite{BK1}.
\begin{assumption}
The control action space  $A$ is a countable union of compact metrizable subsets of $A$.
\end{assumption}

We  study 
the terminal-value problem 
\begin{equation}\label{E:PPDEinitial}
\begin{split}
-\partial_t u(t,x)-{\color{black} H^u(t,x,u(t,x), \partial_x u(t,x))}&=0,\quad (t,x)\in [0,T)\times\Omega,\\
 u(T,x)&=h(x),\quad x\in\Omega,
\end{split}
\end{equation}
where  {\color{black} $(t,x,y,  z,\psi)\mapsto H^\psi(t,x,y,z)$,
$[0,T]\times\Omega\times\R\times\R^d\times \mathrm{BLSA}([0,T]\times\Omega)\to\R$},  is defined by
\begin{equation}\label{E:Def:Hamiltonian:F}
\begin{split}
{\color{black} H^\psi(t,x,y,  z)}&:= 
{\color{black}\inf_{a\in A}}\Bigl\{
\ell(t,x,a)+(f(t,x,a),z)\\ 
&\qquad\qquad+\lambda(t,x,a)\,\int_{\R^d} \left[\psi(t,x\otimes_t e)-y\right]\,Q(t,x,a, \dd e)
\Bigr\}.
\end{split}
\end{equation}

{\color{black}

\begin{remark}
Recall from section~\ref{SS:CanonicalPathSpace}
that $[0,T]\times\Omega$ and its subsets are considered to be equipped with
 the pseudo-metric defined in \eqref{E:SS:CanonicalPathSpace:PseudoMetric}.
 This implies that a function $u:[0,T]\times\Omega\to\R$ that is semicontinuous
 (with respect to this pseudo-metric) is non-anticipating, i.e., $u(t,x)=u(t,x_{\cdot\wedge t})$
 for each $(t,x)\in [0,T]\times\Omega$.
\end{remark}

Let us now outline our approach for establishing existence and uniqueness of \eqref{E:PPDEinitial}. 
First note that \eqref{E:PPDEinitial} formally corresponds to the HJB equation
associated to the deterministic optimal control problem that is specified by the value function 
\begin{align}\label{E:recursiveOptimalControl}
v(s,x)=\inf_{\alpha\in\mathcal{A}} \Bigl\{ \chi^{s,x,\alpha}(T)\,  h(\phi^{s,x,\alpha})
+\int_{s}^T \chi^{s,x,\alpha}(t)\,
\ell^v(t,\phi^{s,x,\alpha},\alpha(t))\,
\dd t\Bigr\}
\end{align}
where \emph{the  running cost}
\begin{align*}
\ell^v(t,x,a):=\ell(t,x,a)+\int_{\R^d} v(t,x\otimes_t e)\lambda(t,x,a)\,Q(t,x,a,\dd e)
\end{align*} 
\emph{depends on the value function $v$ itself} (cf.~Theorem~\ref{0T:value:DPP}).
To tackle this issue related to the running cost, we consider related
deterministic optimal control problems specified by value functions
$v^\psi:[0,T]\times\Omega\to\R$ parametrized by functions $\psi\in\mathrm{BLSA}([0,T]\times\Omega)$.
These value functions are defined as follows:
\begin{align}\label{E:nonRecursiveOptimalControl}
v^\psi(s,x):=\inf_{\alpha\in\mathcal{A}} \Bigl\{ \chi^{s,x,\alpha}(T)\,  h(\phi^{s,x,\alpha})
+\int_{s}^T \chi^{s,x,\alpha}(t)\,
\ell^\psi(t,\phi^{s,x,\alpha},\alpha(t))\,
\dd t\Bigr\}.
\end{align}
Here, the running costs $\ell^\psi:[0,T]\times\Omega\to A$ are given by
\begin{align*}
\ell^\psi(t,x,a):=\ell(t,x,a)+\int_{\R^d} \psi(t,x\otimes_t e)\lambda(t,x,a)\,Q(t,x,a,\dd e).
\end{align*}
Each $v^\psi$ is expected to be a unique solution of  the terminal value problem
\renewcommand{\theequation}
{$\text{\textrm{TVP}}_\psi$}
\begin{equation}\label{E:PPDEpsi}
\begin{split}
-\partial_t u(t,x)-{\color{black} H^\psi(t,x,u(t,x),\partial_x u(t,x))}&=0,\quad (t,x)\in [0,T)\times\Omega,\\
 u(T,x)&=h(x),\quad x\in\Omega.
\end{split},
\end{equation}
This topic is addressed  in Theorem~\ref{T:EUpsi} at the end of subsection~\ref{SS:minimax}.
Next, we establish in subsection~\ref{SS:G} that the operator\footnote{
See \eqref{E:nonRecursiveOptimalControl} for the definition of $v^\psi$.}
\renewcommand{\theequation}{\thesection.\arabic{equation}}
\begin{align}\label{E:0:Operator:G}
G:\psi\mapsto v^\psi,\, \mathrm{BLSA}([0,T]\times\Omega)\to \R^{[0,T]\times\Omega},
\end{align}
has a unique fixed point $\bar{\psi}$ in $\mathrm{BLSA}([0,T]\times\Omega)$, 
which satisfies \eqref{E:recursiveOptimalControl} in place of $v$. 
This result together with existence and uniqueness for \eqref{E:PPDEpsi}
will fairly quickly lead to existence and uniqueness for minimax solutions of \eqref{E:PPDEinitial}
in subsection~\ref{SS:EU}.
Also note that $\bar{\psi}$ coincides on $(0,T]\times\Omega$
 with the value function $V$ defined in \eqref{E:Cost:and:Value} of our stochastic optimal control problem treated in
section~\ref{S:OptimalControl}, i.e., 
\begin{align*}
-\partial_t u(t,x)-{\color{black}H^u(t,x,u(t,x), \partial_x u(t,x))}=0
\end{align*}
is the HJB equation associated to the stochastic optimal control problem for path-dependent PDPs specified
by \eqref{E:Cost:and:Value}.
 Finally, we will obtain a comparison principle between minimax sub- and and supersolutions
of \eqref{E:PPDEinitial} in subsection~\ref{SS:comparison}.
}

{\color{black}The next subsections~\ref{SS:classical} and \ref{SS:minimax} are devoted to  the introduction of
basic material such as path derivatives for functions on $[0,T]\times\Omega$
and  notions of solutions for 
\eqref{E:PPDEinitial}.}

\subsection{Path derivatives and classical solutions}\label{SS:classical}
{\color{black}  The concept of derivatives on $[0,T]\times\Omega$ that we are using here
is a slight modification of the notion of  coinvariant derivatives due to
A.~V.~Kim \cite{Kim85}. 
\begin{definition}
 The space $C_b^{1,1}([0,T]\times\Omega)$ is the set of all  $u\in C_b([0,T]\times\Omega)$,
 for which there exist functions $\partial_t u\in C_b([0,T]\times\Omega)$ and $\partial_x u\in C_b([0,T],\R^d)$
 also called \emph{path derivatives} of $u$ such that, for every $(t_0,x_0)\in [0,T)\times\Omega$,
 every $x\in\Omega$ that is Lipschitz continuous on $[t_0,T]$ and coincides with $x_0$ on $[0,t_0]$,
 and every $t\in (t_0,T]$, 
 \begin{align*}
 u(t,x)-u(t_0,x_0)=\int_{t_0}^t (\partial_t u(s,x)+\partial_x u(s,x)\cdot x^\prime(s))\,\dd s.
 \end{align*}
\end{definition}

Note that if $u\in C^{1,1}_b([0,T]\times\Omega)$, then its path derivatives  $\partial_t u$ and $\partial_x u$
are uniquely determined. Moreover, there is a 
 close relationship of the here defined
path derivatives with Dupire's
explicitly defined horizontal and vertical path derivatives introduced
in \cite{dupirefunctional}. We refer to section~3 of our companion paper \cite{BK1} for details.

\begin{definition}
A function $u:[0,T]\times\Omega\to\R$ is a \emph{classical solution} of \eqref{E:PPDEinitial} if
$u\in C_b^{1,1}([0,T]\times\Omega)$, $u(T,\cdot)=h$, and, for every $(t,x)\in [0,T)\times\Omega$,
\begin{align*}
-\partial_t u(t,x)-{\color{black}H^u(t,x,u(t,x), \partial_x u(t,x))}=0.
\end{align*}
\end{definition}

\subsection{Minimax solutions}\label{SS:minimax}
Here, we define our notion of nonsmooth solutions of the terminal value problems \eqref{E:PPDEinitial} and 
\eqref{E:PPDEpsi}, $\psi\in \mathrm{BLSA}([0,T]\times\Omega)$. To this end we need the following path spaces. 
Given $(s_0,x_0)\in [0,T)\times\Omega$ and $z\in\R^d$, define 
\begin{align*}
\mathcal{X}(s_0,x_0)&:=\Big \{x\in\Omega:\, x=x_0\text{ on $[0,s_0]$, $x\vert_{[s_0,T]}$
is absolutely continuous 
with}\\
&\qquad\qquad \abs{x^\prime(t)}\le C_f(1+\sup_{s\le t} \abs{x(s)})\text{ a.e.~on $(s_0,T)$}\Big \}.
\end{align*}
Given  $(s_0,x_0)\in [0,T)\times\Omega$, $z\in\R^d$, and $\psi\in \mathrm{BLSA}([0,T]\times\Omega)$, define
\begin{equation}\label{E:Ypsi}
\begin{split}
&\mathcal{Y}_\psi (s_0,x_0,y_0,z):=\Big \{(x,y)\in\mathcal{X}(s_0,x_0)\times C([s_0,T]):\\
&\qquad
  y(t)=y_0+\int_{s_0}^t (x^\prime(s),z)-{\color{black}H^\psi(s,x,y(s),z)}\,\dd s\text{ on $[s_0,T]$}\Big \}.
\end{split}
\end{equation}
}

{\color{black}
\begin{definition}\label{D:mm:PPDEu}
Let  $u:[0,T]\times\Omega\to\R$.

(i) $u$ is a \emph{minimax supersolution} of \eqref{E:PPDEinitial} if  $u\in {\color{black}\mathrm{BLSC}} 
([0,T]\times\Omega)$,
if $u=u(s,x)$ is Lipschitz in $x$, 
if   $u(T,\cdot)\ge h$,  and if,
 for all $(s_0,x_0,z)\in [0,T)\times\Omega\times\R^d$ and all
 $y_0\ge u(s_0,x_0)$, there is an
$(x,y)\in \mathcal{Y}_u (s_0,x_0,y_0,z)$ such that $y(t)\ge u(t,x)$ for each $t\in[s_0,T]$.

(i) $u$ is a \emph{minimax subsolution} of \eqref{E:PPDEinitial} if 
  $u\in   {\color{black}\mathrm{BUSC}}
  ([0,T]\times\Omega)$,
if $u=u(s,x)$ is Lipschitz in $x$, 
 if $u(T,\cdot)\le h$, and if
  for all $(s_0,x_0,z)\in [0,T)\times\Omega\times\R^d$ and all
 $y_0\le u(s_0,x_0)$, there is an
$(x,y)\in \mathcal{Y}_u (s_0,x_0,y_0,z)$ such that $y(t)\le u(t,x)$ for each $t\in[s_0,T]$.

(iii) $u$ is a \emph{minimax solution} of \eqref{E:PPDEinitial} if $u$ is both a \emph{minimax supersolution}
and a \emph{minimax subsolution} of  \eqref{E:PPDEinitial}.
\end{definition}}

{\color{black}
\begin{theorem}
Let $u\in C_b^{1,1}([0,T]\times\Omega)$. 
Suppose  in addition that $u=u(t,x)$ is Lipschitz in $x$ and that all mappings
\begin{align}\label{E:F:continuity:consistency}
(t,x)\mapsto {\color{black} H^u(t,x,u(t,x), z)},\quad [0,T]\times\Omega\to\R,\quad z\in\R^d,
\end{align}
 are continuous.
Then $u$  is a classical solution of \eqref{E:PPDEinitial} if and only if $u$ is a minimax solution of \eqref{E:PPDEinitial}.
\end{theorem}

\begin{proof}
(i) First, assume that $u$ is a classical solution of \eqref{E:PPDEinitial}. We  show that
$u$ is a minimax supersolution of \eqref{E:PPDEinitial}.
To this end,   fix $(s_0,x_0,z)\in [0,T)\times\Omega\times\R^d$. 
{\color{black} 
By Proposition~5.1 of \cite{Lukoyanov03}, there is a pair
$(x,y)\in\mathcal{X}(s_0,x_0)\times C([s_0,T])$ such that 
\begin{equation}\label{E1:Compatibility}
\begin{split}
y(t)&=u(s_0,x_0)+
\int_{s_0}^t (x^\prime(s),z)-{\color{black}H^u (s,x,u(s,x),z)}\,\dd s\text{ and}\\
y(t)&=u(t,x)
\end{split}
\end{equation}
for every $t\in [s_0,T]$. We proceed by considering the following two cases.}

\textit{Case 1.} 
Let $y_0=u(s_0,x_0)$. 
{\color{black} By \eqref{E1:Compatibility}, we have $(x,y)\in\mathcal{Y}_u(s_0,x_0,y_0,z)$ and 
$y(t)=u(t,x)$ on $[s_0,T]$.} 

\textit{Case 2.} 
Let $y_0>u(s_0,x_0)$. {\color{black}  Now, let  $\tilde{y}$ be a solution of 
\begin{equation}\label{E2:Compatibility}
\begin{split}
\tilde{y}^\prime(t)&= x^\prime(t)\cdot z-{\color{black}H^u(t,x,\tilde{y}(t),z)}\quad\text{a.e.~on $(s_0,T)$},\\
\tilde{y}(s_0)&=y_0,
\end{split}
\end{equation}
Then $(x,\tilde{y})\in \mathcal{Y}_u(s_0,x_0,y_0,z)$. We claim that
\begin{align}\label{E:claim:Compatibility}
\tilde{y}(t)> u(t,x) \text{ on $(s_0,T)$.}
\end{align}
For the sake of a contradiction, assume that \eqref{E:claim:Compatibility} does not hold, i.e., 
 there is a smallest time $s_1\in (s_0,T)$ such that $\tilde{y}(s_1)= u(t,s_1)$
and $\tilde{y}(t)> u(t,x)$ on $(s_0,s_1)$. 
But then, by \eqref{E1:Compatibility}, \eqref{E2:Compatibility}, and \eqref{E:Def:Hamiltonian:F},  we have
\begin{align*}
0&=\tilde{y}(s_1)-u(s_1,x)\\&=\tilde{y}(s_0)-u(s_0,x_0)+\int_{s_0}^{s_1}  
\Bigl[{\color{black}H^u(t,x,u(t,x),z)}
{\color{black} -H^u(t,x,\tilde{y}(t),z)}\Bigr]\,\dd t\\
&\ge \tilde{y}(s_0)-u(s_0,x_0)\\ &>0,
\end{align*}
which is a contradiction (cf.~section~6.2 of \cite{BK1}). Thus, 
\eqref{E:claim:Compatibility} holds. 

By Cases~1 and 2, we can conclude that   $u$ is a minimax supersolution of \eqref{E:PPDEinitial}.
Similarly, one can show that $u$ is a minimax subsolution of \eqref{E:PPDEinitial}.
Hence, $u$ is a minimax solution of \eqref{E:PPDEinitial}.

(ii) Now, assume that $u$ is a minimax solution of \eqref{E:PPDEinitial}.
Fix $(s_0,x_0)\in [0,T)\times\Omega$. 
Then there is an $(x,y)\in\mathcal{Y}_u(s_0,x_0,u(s_0,x_0),\partial_x u(s_0,x_0))$ such
that $y(t)\ge u(t,x)$ on $[s_0,T]$, i.e., for every $t\in [s_0,T]$, we have
\begin{align*}
&u(s_0,x_0)+\int_{s_0}^t \left[{\color{black} x^\prime(s)\cdot \partial_x u(s_0,x_0)}-
{\color{black}H^u(s,x,y(s),\partial_x u(s_0,x_0))}\right]\,\dd s\\ &\qquad\ge u(t,x),
\end{align*}
 or, equivalently,
 \begin{align*}
&\int_{s_0}^t \left[{\color{black} x^\prime(s)\cdot \partial_x u(s_0,x_0)}-{\color{black}H^u(s,x,y(s),\partial_x u(s_0,x_0))}\right]\,\dd s\\
&\qquad\qquad\ge \int_{s_0}^t
\left[ \partial_t u(s,x)+x^\prime(s)\cdot \partial_x u(s,x)\right]\,\dd s.
\end{align*}
Hence, by \eqref{E:F:continuity:consistency},
\begin{align*}
-\partial_t u(s_0,x_0)-{\color{black}H^u(s_0,x_0,u(s_0,x_0),\partial_x u(s_0,x_0))}\ge 0.
\end{align*}
Similarly, one can see that the opposite inequality holds as well, i.e., $u$ is a classical solution of \eqref{E:PPDEinitial}.}
\end{proof}
}

{\color{black}
\begin{definition}\label{D2:MinimaxSolutions}
Let $\psi\in \mathrm{BLSA}([0,T]\times\Omega)$ and  $u:[0,T]\times\Omega\to\R$.

(i) $u$ is a \emph{minimax supersolution} of \eqref{E:PPDEpsi} if $u\in\mathrm{LSC}([0,T]\times\Omega)$,
if $u(T,\cdot)\ge h$ on $\Omega$, and if, for every $(s_0,x_0,z)\in [0,T)\times\Omega\times\R^d$, and
every $y_0\ge u(s_0,x_0)$, there exists an $(x,y)\in\mathcal{Y}_\psi(s_0,x_0,y_0,z)$ such that
$y(t)\ge u(t,x)$ for each $t\in[s_0,T]$.

(ii)  $u$ is a \emph{minimax subsolution} of \eqref{E:PPDEpsi} if $u\in\mathrm{USC}([0,T]\times\Omega)$,
if $u(T,\cdot)\le h$ on $\Omega$, and if, for every $(s_0,x_0,z)\in [0,T)\times\Omega\times\R^d$, and
every $y_0\le u(s_0,x_0)$, there exists an $(x,y)\in\mathcal{Y}_\psi(s_0,x_0,y_0,z)$ such that
$y(t)\le u(t,x)$ for each $t\in[s_0,T]$.

(iii) $u$ is a \emph{minimax solution} of \eqref{E:PPDEpsi} if $u$ is both a \emph{minimax supersolution}
and a \emph{minimax subsolution} of  \eqref{E:PPDEpsi}.
\end{definition}

The next result uses  the operator $G$ from \eqref{E:0:Operator:G}.

\begin{theorem}\label{T:EUpsi}
Let $\psi_0=\psi_0(s,x)\in \mathrm{BLSA}( [0,T]\times\Omega)$ be
Lipschitz in $x$. 

(a) Then $G^2\psi_0$ is the unique minimax
solution of  \eqref{E:PPDEpsi} with $\psi=G\psi_0$.

(b) If, in addition, $\psi_0$ is continuous, then $G\psi_0$ is the unique minimax 
solution of  \eqref{E:PPDEpsi} with $\psi=\psi_0$.
\end{theorem}

\begin{proof}
(a) First, note that $G\psi_0$ is continuous (to see this, one can proceed as in the proof of 
{\color{black} Proposition~7.4 in \cite{BK1}).}
Thus, together with Theorem 96, 146-IV, in \cite{DMprobPot}, we can deduce that the maps
\begin{align}\label{E:Proof:EUpsi}
t\mapsto [G\psi_0](t,x\otimes_t e)],\, [0,T]\to\R,\,(x,e)\in\Omega\times E,
\end{align}
are Borel-measurable. Next, define $\ell^{[G\psi_0]}:[0,T]\times\Omega\times A\to\R_+$ by
\begin{align*}
\ell^{[G\psi_0]}(t,x,a):=\ell(t,x,a)+\int_{\R^d} [G\psi_0](t,x\otimes_t e)\lambda(t,x,a)\,Q(t,x,a,\dd e).
\end{align*}
By Assumptions~\ref{A:data:controlled} and \ref{A:Q}, the Borel-measurability of  the maps in 
\eqref{E:Proof:EUpsi}, and the Lipschitz continuity of $\psi_0$ in $x$,  Theorem~{\color{black}7.7}
 in \cite{BK1} with 
$\ell^{[G\psi_0]}$ in place of $\ell$ yields part~(a) of the theorem.

\medskip 

(b) It suffices to follow the lines of part~(a) of this proof but with $G\psi_0$ replaced by $\psi_0$.
This concludes the proof.
\end{proof}

}




\subsection{Properties of the operator $G$}\label{SS:G}


{\color{black}  Let us first recall  the operator $G$ from \eqref{E:0:Operator:G}:}
For each  function $\psi\in \mathrm{BLSA}([0,T]\times\Omega)$ and $(s,x)\in [0,T]\times\Omega$, 
\begin{equation}\label{E:Operator:G}
\begin{split}
&[G\psi](s,x)=\inf_{a\in\mathcal{A}} \Bigl\{ \chi^{s,x,a}(T)\,  h(\phi^{s,x,a})\\
&
+\int_{s}^T \chi^{s,x,a}(t)\,
\Bigl[\ell^{s,x,a}(t)
+ \int_{\R^d}
 \psi(t,\phi^{s,x,a}\otimes_t e)\,\lambda^{s,x,a}(t)\,
Q^{s,x,a}(t,\dd e)\Bigr]\,\dd t
\Bigr\}.
\end{split}
\end{equation}

\begin{remark} \label{R:G}
{\color{black} Note that}
\begin{align*}\label{E2:G}
G: \mathrm{BLSA}([0,T]\times\Omega)\to\mathrm{BLSA}([0,T]\times\Omega).
\end{align*}
{\color{black} This follows from} 
Theorem~96, 146-IV, in \cite{DMprobPot}, Lemma~\ref{L:phi:meas}, Remark~\ref{R:coeff:meas},
Proposition~7.48 in \cite{BS}, and Proposition~7.47 in \cite{BS}.
\end{remark}

Let us introduce several operators related to $G$. 
{\color{black}Given $s_1$, $s_2\in\R$ 
such that $0\le s_1<s_2\le T$, $\eta\in \mathrm{BLSA}([s_2,T]\times\Omega)$,  $a\in\mathcal{A}$, 
{\color{black} $\psi \in \mathrm{BLSA}([s_1,s_2]\times\Omega)$,}
and $(s,x)\in [s_1,s_2]\times\Omega$,  define
\begin{equation}\label{E:G:eta}
\begin{split}
&[(G_{s_1,s_2;\eta,a})\psi](s,x):=
\chi^{s,x,a}(s_2)\,\eta(s_2,\phi^{s,x,a})+\int_s^{s_2}\Bigl[
\chi^{s,x,a}(t)\,\ell^{s,x,a}(t)\\ &\qquad
+\int_{\R^d} \psi(t,\phi^{s,x,a}\otimes_t e)\,\lambda^{s,x,a}(t) \chi^{s,x,a}(t)\,Q^{s,x,a}(t,\dd e)
 \Bigr]\,\dd t\qquad\text{and}\\
& [(G_{s_1,s_2;\eta})\psi](s,x):=\inf_{\tilde{a}\in\mathcal{A}} [(G_{s_1,s_2;\eta,\tilde{a}})\psi](s,x).
\end{split}
\end{equation}
Note that $G_{s_1,s_2;\eta,a}$ and $G_{s_1,s_2;\eta}$ map bounded lower semi-analytic functions
to bounded lower semi-analytic functions. This can be shown in the same way as the fact
that $G$ maps bounded lower semi-analytic functions
to bounded lower semi-analytic functions in  Remark~\ref{R:G}.}

\begin{theorem}\label{T:Claim1}
The following holds:

(i)  The operator $G$ 
has a unique fixed point  $\bar \psi$ in $\mathrm{BLSA}([0,T]\times\Omega)$, which coincides with {\color{black} the value function} $V$ {\color{black} from \eqref{E:Cost:and:Value}}
on $(0,T]\times\Omega$.

(ii)  If $0=r_0<r_1<\ldots<r_N=T$ with
$r_{k+1}-r_k<\frac{1}{2}$  for all $k$ and if  $\psi\in \mathrm{BLSA}([0,T]\times\Omega)$, then
$$(G^n_{r_{k-1},r_k;\bar{\psi}})(\psi\vert_{[r_{k-1},r_k]\times\Omega})\to
\bar{\psi}\vert_{[r_{k-1},r_k]\times\Omega}$$ uniformly as $n\to\infty$,
where the operators $G_{r_{k-1},r_k;\bar{\psi}}$ are defined by \eqref{E:G:eta}.
 \end{theorem}
\begin{proof}
%
(i) \textit{Step~1 (construction of $\bar{\psi}$).} 
Consider a partition $0=r_0<r_1<\ldots<r_N=T$ of $[0,T]$ with $r_k-r_{k-1}<\frac{1}{2}$ for all $k$. 
Then $G_{r_{N-1},r_N;h}$ is a contraction {\color{black} on  $\mathrm{BLSA}([r_{N_1},r_N]\times\Omega)$. To see this,
note first that, for every $\alpha\in\mathcal{A}$, there is, by \eqref{E:G:eta}, Assumption~\ref{A:Q}, and the non-negativity 
of $\lambda$ (see Assumption~\ref{A:data:controlled}),
 some probability measure $\Prob$ with expected value $\Mean^\Prob$ such that, 
 {\color{black} for each $\psi_1$, $\psi_2 \in \mathrm{BLSA}([r_{N_1},r_N]\times\Omega)$,} 
\begin{align*}
[G_{r_{N-1},r_N;h,\alpha}\,\psi_1](s,x)-[G_{r_{N-1},r_N;h,\alpha}\,\psi_2](s,x)&\le \int_{r_{N-1}}^{r_N} 
\Mean^\Prob \norm{\psi_1-\psi_2}_\infty\,\dd t\\ &\le \frac{1}{2}\norm{\psi_1-\psi_2}_\infty.
\end{align*}
Thus, for every $(s,x)\in [r_{N-1},r_N]\times\Omega$ and $\eps>0$,  there is an $\alpha^\eps\in\mathcal{A}$ with
\begin{align*}
&[G_{r_{N-1},r_N;h}\,\psi_1](s,x)-[G_{r_{N-1},r_N;h}\,\psi_2](s,x)\\ &\qquad\le 
[G_{r_{N-1},r_N;h}\,\psi_1](s,x)-[G_{r_{N-1},r_N;h,\alpha^\eps}\,\psi_2](s,x)+\eps\\
&\qquad\le 
[G_{r_{N-1},r_N;h,\alpha^\eps}\,\psi_1](s,x)-[G_{r_{N-1},r_N;h,\alpha^\eps}\,\psi_2](s,x)+\eps
\le \frac{1}{2}\norm{\psi_1-\psi_2}_\infty+\eps.
\end{align*}}{\color{black} Since $G_{r_{N-1},r_N;h}$ is a contraction on  $\mathrm{BLSA}([r_{N_1},r_N]\times\Omega)$, it} 
 has a unique fixed point $\bar{\psi}_N{\color{black}\in\mathrm{BLSA}([r_{N_1},r_N]\times\Omega)}$.

{\color{black} Next, we show that  the value function $V^\prime$ defined by \eqref{E:Cost:and:Value} 
coincides with $\bar{\psi}_N$ on $[r_{N-1},r_N]\times\Omega$. Formally, this would immediately follow
from Theorem~\ref{0T:value:DPP} and the uniqueness of the fixed point $\bar{\psi}$.
Note  however that at this point, we have not
established yet that $V^\prime$ is lower semi-analytic.  Thus we cannot apply $G$
(or the related operators defined by \eqref{E:G:eta})
to $V^\prime$, as $G$ is an operator on $\mathrm{BLSA}([0,T]\times\Omega)$.
 Instead, we  use $G^\circ$ from \eqref{0E:T:value:DPP}, which can be applied to $V^\prime$, and obtain}
\begin{align*}
\norm{V^\prime\vert_{[r_{N-1},r_N]\times\Omega}-\bar{\psi}_N}_\infty&=
\norm{G^\circ V^\prime\vert_{[r_{N-1},r_N]\times\Omega}-G_{r_{N-1},r_N;h} \bar{\psi}_N}_\infty\\ &\le {\color{black}\frac{1}{2}}
\norm{V^\prime\vert_{[r_{N-1},r_N]\times\Omega}-\bar{\psi}_N}_\infty,
\end{align*}
i.e., $\bar{\psi}_N=V^\prime\vert_{[r_{N-1},r_N]\times\Omega}$ and thus
$V^\prime$ is lower semi-analytic on $[r_{N-1},r_N]\times\Omega$.

Similarly, $G_{r_{N-2},r_{N-1};\bar{\psi}_N}$ is a contraction
{\color{black} on  $\mathrm{BLSA}([r_{N-2},r_{N-1}]\times\Omega)$} and thus has
a unique fixed point $\bar{\psi}_{N-1}{\color{black}\in\mathrm{BLSA}([r_{N-2},r_{N-1}]\times\Omega)}$, which coincides with
$V^\prime$ on $[r_{N-2},r_{N-1}]\times\Omega$ thanks to 
{\color{black} Theorem~\ref{0T:value:DPP}.}

Iterating this process yields a finite sequence
$(\bar{\psi})_{k=1}^N$ with the properties that each $\bar{\psi}_k$ is the unique fixed point
of $G_{r_{k-1},r_k;\bar{\psi}_{k+1}}$ {\color{black} on  $\mathrm{BLSA}([r_{k-1},r_{k}]\times\Omega)$}, where we set  $\bar{\psi}_{N+1}(s,x):=h(x)$, and
that  $\bar{\psi}_k$  coincides with $V^\prime$ on $[r_{k-1},r_k]\times\Omega$ if $k>2$ and on 
$(0,r_1]\times\Omega$ if $k=1$.

Next, define $\bar{\psi}:[0,T]\times\Omega\to\R$ by
\begin{align*}
\bar{\psi}(s,x):=\begin{cases}
\bar{\psi}_k(s,x)&\text{ if $s\in [r_{k-1},r_k)$ for some $k\in\{1,\ldots,N\}$,}\\
h(x)&\text{ if $s=T$.}
\end{cases}
\end{align*}
Together with Theorem~\ref{0T:value:value}, we obtain $\bar{\psi}=V^\prime=V$ on $(0,T]\times\Omega$.

\medskip 

\textit{Step~2 ($\bar{\psi}$ is a fixed point of $G$).} Consider the function $\bar{v}:[0,T]\times\Omega\to\R$
defined by $\bar{v}(s,x):=[G\bar{\psi}](s,x)$. Then $\bar{v}=\bar{\psi}_N$ on $[r_{N-1},r_N]\times\Omega$. 

Now, let
$(s,x)\in [r_{N-2},r_{N-1})\times\Omega$. Then, by the \textcolor{black}{dynamic programming principle} 
(cf.~Proposition~III.2.5 in
\cite{BardiCapuzzoDolcetta})
 and the fixed point
property of $\bar{\psi}_{N-1}$, we have
\begin{align*}
\bar{v}(s,x){\color{black}=[(G_{0,T;h})\bar{\psi}](s,x)}
=[(G_{r_{N-2},r_{N-1};\bar{v}})\bar{\psi}_{N-1}](s,x)=\bar{\psi}_{N-1}(s,x).
\end{align*}
Iterating this procedure yields $\bar{v}=\bar{\psi}$, i.e., $\bar{\psi}$ is a fixed point of $G$.

{\color{black}
Another  way to verify the desired fixed point property of $\bar{\psi}$ would be to just note that, by Step~1,
$\bar{\psi}$ coincides with $V$ and 
 is lower semi-analytic. Then, together with Corollary~\ref{0C:value:DPP},
$V=GV$ follows. However, note this argument relies on the rather technical section~\ref{SS:value} 
unlike the argument in the preceding paragraphs.}

\medskip 

\textit{Step~3 ($\bar{\psi}$ is the unique fixed point of $G$).} Let $\bar{\psi}^\prime$ be a fixed point of $G$.
On $[r_{N-1},r_N]\times\Omega$, we have
$\bar{\psi}^\prime=G\bar{\psi}^\prime=(G_{r_{N-1},r_N;h})\bar{\psi}^\prime$.
Thus $\bar{\psi}^\prime$ coincides on $[r_{N-1},r_N]\times\Omega$ with $\bar{\psi}_N$, the unique 
fixed point of $G_{r_{N-1},r_N;h}$. 

Next, let $(s,x)\in [r_{N-2},r_{N-1}]\times\Omega$.
Then, by the \textcolor{black}{dynamic programming principle} (cf.~Proposition~III.2.5 in
\cite{BardiCapuzzoDolcetta}),
\begin{align*}
\bar{\psi}^\prime(s,x)&=[(G_{0,T;h})\bar{\psi}^\prime](s,x)=[(G_{r_{N-2},r_{N-1};\bar{\psi}^\prime})\bar{\psi}^\prime](s,x)
=[(G_{r_{N-2},r_{N-1};\bar{\psi}_N})\bar{\psi}^\prime](s,x).
\end{align*}
Hence, $\bar{\psi}^\prime$ coincides on $[r_{N-2},r_{N-1}]\times\Omega$ with $\bar{\psi}_{N-1}$, the unique 
fixed point of $G_{r_{N-2},r_{N-1};\bar{\psi}_N}$. Iterating this process yields $\bar{\psi}^\prime=\bar{\psi}$.

\medskip 

(ii)  The convergence results follow  from Step~1 of part~(i) of this proof.
\end{proof}

Next, we show that the fixed point $\bar{\psi}$   
is sufficiently regular.
{\color{black} 
To this end, we need the following auxiliary result,
whose 
 proof is relegated to subsection~\ref{S:A2}.
}

\begin{lemma}\label{L:barpsi:reg:Step0}
{\color{black} Let $c$, $c^\prime\in\R_+$, $0\le s_1<s_2\le T$, and
 $\psi$, $\eta\in \mathrm{BLSA}([0,T]\times\Omega)$. Let} 
\begin{equation}\label{E:Lip:L:barpsi:reg:Step0}
{\color{black}
\abs{\psi(s,x)-\psi(s,\tilde{x})}\le c\norm{x-\tilde{x}}_s} \text{ and }
\abs{\eta(s,x)-\eta(s,\tilde{x})}\le c^\prime\norm{x-\tilde{x}}_s 
\end{equation}
{\color{black}
for all  $s\in [s_1,s_2]$, $x$, $\tilde{x}\in\Omega$. 
Let $\norm{\eta}_\infty\le\norm{\bar{\psi}}_\infty$.}
 Recall $G_{s_1,s_2;\eta}$ from \eqref{E:G:eta}.
Then 
\begin{equation}
\begin{split}
\label{E:GpsiStab:xe}
&\abs{
{\color{black} [(G_{s_1,s_2;\eta})\psi](s,x)-[(G_{s_1,s_2;\eta})\psi](s,\tilde{x})}}
\\&\qquad\qquad\le
{\color{black}\left[c^\prime\ee^{L_f T}+6\check{L}(s_2-s)(1+c)(1+\norm{\psi}_\infty)\right]\norm{x-\tilde{x}}_s}
\end{split}
\end{equation}
for all $(s,x,\tilde{x})\in [s_1,s_2]\times\Omega\times\Omega$, {\color{black} where}
\begin{align*}
{\color{black} \check{L}:=\ee^{L_f T}\max\{L_f,L_f\,C_f, {\color{black} L_f\norm{\bar{\psi}}_\infty,}\,
C_f,C_\lambda\,L_Q,C_\lambda,1\}}.
\end{align*}
\end{lemma}

\begin{theorem}\label{L:barpsi:reg:Step1}
There exists a constant $L_{\bar{\psi}}\in\R_+$ such that,
for every $s\in [0,T]$ and every $x$, $\tilde{x}\in\Omega$,  we have
\begin{align}\label{E:Gpsi:Lip}
{\color{black}
\abs{
\bar{\psi}(s,x)-\bar{\psi}(s,\tilde{x})
}\le 
L_{\bar{\psi}}\norm{x-\tilde{x}}_{s}}.
\end{align}
\end{theorem}
\begin{proof}
First,  given $c^\prime\in\R_+$, define $(c_n(c^\prime))_{n\in\N_0}$ recursively by $c_0(c^\prime)=0$ and
\begin{align*}
c_{n+1}(c^\prime):=\left(c^\prime\ee^{L_fT} +\frac{1}{2}\right)+\frac{1}{2} c_n(c^\prime).
\end{align*} 
By mathematical induction,
\begin{align}\label{E:cn}
c_n(c^\prime)\le 2 c^\prime\ee^{L_fT}+1\text{ for all $n\in\N$.}
\end{align}

Next, recall the partition $0=r_0<r_1<\ldots<r_N=T$ from Step~1 of  the proof of Theorem~\ref{T:Claim1}~(i).
Note that $r_{j+1}-r_j<\frac{1}{2}$ for all $j$. 

Let $\psi_0\equiv 0$ and consider a subpartition $(s_{j,k})_{j=0,\ldots,{N-1};k=0,\ldots,M}$
with a sufficiently large integer $M$ such that, for each
$j\in\{1,\ldots,N\}$, we have $$r_{j-1}=s_{j-1,0}<s_{j-1,1}<\ldots<s_{j-1,M}=r_j$$ and, for each $k\in\{0,\ldots,M-1\}$, 
\begin{align}\label{E1:barpsi:reg:Step1}
s_{j,k+1}-s_{j,k} \le \frac{1}{12\check{L}{\color{black}\left(1+\sup_{n\in\N_0} 
\norm{(G^n_{r_{j-1},r_j;\bar{\psi}}) \psi_0}_\infty\right)}}=:\check{m}_j.
\end{align}

\medskip

\textit{Step~1.}
Let $n\in\N_0$. By \eqref{E1:barpsi:reg:Step1} and
 Lemma~\ref{L:barpsi:reg:Step0} applied to $[(G^n_{r_{N-1},r_N;\bar{\psi}}) \psi_0]$
 restricted to $[s_{N-1,M-1},s_{N-1,M}]\times\Omega$, 
 $c_n(L_f)$, $L_f$, $h$, $s_{N-1,M-1}$, $s_{N-1,M}$ 
 in place of $\psi$, $c$, $c^\prime$, $\eta$, $s_1$, $s_2$
 and noting that 
 \begin{align*}
 &(G_{s_{N-1,M-1},s_{N-1,M};h})(\psi_0\vert_{[s_{N-1,M-1},s_{N-1,M}]\times\Omega})\\&\qquad=
 [(G_{r_{N-1},r_N;\bar{\psi}})\psi_0)\vert_{[s_{N-1,M-1},s_{N-1,M}]\times\Omega},
 \end{align*}
  we can see that
\begin{align*}
&\abs{[G_{r_{N-1},r_N;\bar{\psi}}^{n+1}\psi_0](s,x)-[G_{r_{N-1},r_N;\bar{\psi}}^{n+1}\psi_0](s,\tilde{x})}\le
  \left(L_f
\ee^{L_fT}+\frac{1+c_n(L_f)}{2}\right)\norm{x-\tilde{x}}_s\\
&\qquad =c_{n+1}(L_f)\norm{x-\tilde{x}}_s\le (2L_f\ee^{L_fT}+1)\norm{x-\tilde{x}}_s
\text{ (by \eqref{E:cn})}
\end{align*}
for all $s\in [s_{N-1,M-1},s_{N-1,M}]$, $x$, $\tilde{x}\in\Omega$.

By Theorem~\ref{T:Claim1}~(ii), letting $n\to\infty$  yields \eqref{E:Gpsi:Lip}
 on $[s_{N-1,M-1},s_{N-1,M}]\times\Omega
\times\Omega$ 
 with $c^\prime_N:=2L_f\ee^{L_fT}+1$
in place of $L_{\bar{\psi}}$. 

\medskip

\textit{Step~2.} Let $n\in\N_0$.  
Without loss of generality, 
 assume that
$r_{N-1}\le s_{N-1,M-2}$.
Consequently, by \eqref{E1:barpsi:reg:Step1} and
 Lemma~\ref{L:barpsi:reg:Step0} applied to $[(G^n_{r_{N-1},r_N;\bar{\psi}}) \psi]$ 
 restricted to $[s_{N-1,M-2},s_{N-1,M-1}]\times\Omega$,
  $c_n(c^\prime_N)$,
 $c^\prime_N$, $\bar{\psi}$, $s_{N-1,M-2}$, $s_{N-1,M-1}$ in place 
 of $\psi$, $c$, $c^\prime$, $\eta$, $s_1$, $s_2$,
 we can with Step~1 of this proof deduce that
\begin{align*}
&\abs{[G_{r_{N-1},r_N;\bar{\psi}}^{n+1}\psi_0](s,x)-[G_{r_{N-1},r_N;\bar{\psi}}^{n+1}\psi_0](s,\tilde{x})}\le
  \left(c^\prime_N
\ee^{L_fT}+\frac{1+c_n(c^\prime_N)}{2}\right)\norm{x-\tilde{x}}_s\\
&\qquad =c_{n+1}(c^\prime_N)\norm{x-\tilde{x}}_s\le (2c^\prime_N\ee^{L_fT}+1)\norm{x-\tilde{x}}_s
\text{ (by \eqref{E:cn})}
\end{align*}
for all $s\in [s_{N-1,M-2},s_{N-1,M-1}]$, $x$, $\tilde{x}\in\Omega$.
By Theorem~\ref{T:Claim1}~(ii), letting $n\to\infty$  yields \eqref{E:Gpsi:Lip}
 on $[s_{N-1,M-2},s_{N-1,M-1}]\times\Omega
\times\Omega$ 
with 
$2c^\prime_N\ee^{L_fT}+1$
in place of $L_{\bar{\psi}}$. 

\medskip 

\textit{Step~3.} Repeating the procedure of Step~2 finitely many times and making
the obvious adjustments every time we pass to a different interval $[r_{k-1},r_k]$, we obtain
 the existence of some constant
$L_{\bar{\psi}}\in\R_+$ such that \eqref{E:Gpsi:Lip} holds on $[0,T]\times\Omega\times\Omega$.
\end{proof}

\setcounter{equation}{0}

\renewcommand{\theequation}{\thesection.\arabic{equation}}


\subsection{Existence and uniqueness}\label{SS:EU}

Recall  {\color{black} from  Theorem~\ref{T:Claim1}} the 
 fixed point $\bar{\psi}$ of the operator $G$. 

\begin{theorem}\label{T:fixedPoint:uniqueMinimax} 
$\bar{\psi}$ is the unique minimax solution of \eqref{E:PPDEinitial}.
\end{theorem}
\begin{proof}
(i) Existence: By {\color{black}Theorem}~\ref{L:barpsi:reg:Step1}, $\bar{\psi}=\bar{\psi}(s,x)$ is Lipschitz continuous in $x$.
By Theorem \ref{T:Claim1}, $\bar{\psi}\in \mathrm{BLSA}([0,T]\times\Omega)$ and $\bar{\psi}=G\bar{\psi}=G^2\bar{\psi}$.
Thus, by Theorem~\ref{T:EUpsi}~(a), $\bar{\psi}$ is a minimax solution
of \eqref{E:PPDEpsi} with $\psi=\bar{\psi}$.
{\color{black} We can conclude that $\bar{\psi}$ is a minimax solution of \eqref{E:PPDEinitial}.}
  
  \medskip

(ii) Uniqueness: Let $u$ be a minimax solution of \eqref{E:PPDEinitial}.
Then $u$ is also a minimax solution of \eqref{E:PPDEpsi} with $\psi=u$. By Theorem~\ref{T:EUpsi}~(b),
$G u$ is the unique minimax solution of \eqref{E:PPDEpsi} with $\psi=u$.
Moreover, by Theorem~{\color{black}7.7}  in \cite{BK1}, $u=Gu$,  {\color{black} which} 
yields $u=\bar{\psi}$ according to Theorem~\ref{T:Claim1} {\color{black} (i)}. 
{\color{black} This concludes the proof.}
\end{proof}

 {\color{black} 
 The next result is an immediate consequence of  Theorems ~\ref{T:fixedPoint:uniqueMinimax}
 and ~\ref{T:Claim1}.
 \begin{cor}
 The value function $V$ from \eqref{E:Cost:and:Value} is the unique minimax solution of
 \eqref{E:PPDEinitial} on $(0,T]\times\Omega$.
 \end{cor}}


\subsection{The comparison principle}\label{SS:comparison}
Recall our notions of solution in Definition~\ref{D:mm:PPDEu}.

\begin{theorem}
Let $u_0$ be a minimax sub- and $v_0$ be a minimax supersolution of 
\eqref{E:PPDEinitial}. Then $u_0\le v_0$.
\end{theorem}

\begin{proof}
We adapt the idea of the proof of Theorem 5.3 in
\cite{ColaneriEtAl20SPA} to our context.
Consider a partition $0=r_0<r_1<\ldots<r_N=T$ as in Theorem~\ref{T:Claim1}~(ii).
Also recall the operators defined in \eqref{E:G:eta}. 

\medskip 

\textit{Step~1.} First, define functions
\begin{align*}
u_{N,n}:[r_{N-1},r_N]\times\Omega\to\R, (t,x)\mapsto u_{N,n}(t,x):=[G^n_{r_{N-1},r_N;h}\,u_0](t,x),\\
v_{N,n}:[r_{N-1},r_N]\times\Omega\to\R, (t,x)\mapsto v_{N,n}(t,x):=[G^n_{r_{N-1},r_N;h}\,v_0](t,x).
\end{align*}
for each $n\in\N$. As $u_0$ and $v_0$ are semicontinuous, we can
with the same argument as in Theorem~\ref{T:EUpsi}~(b) deduce that
$u_{N,1}$  (resp.~$v_{N,1}$) is the unique minimax solution of  \eqref{E:PPDEpsi} on $[r_{N-1},r_N]\times\Omega$
with $\psi=u_0$ (resp.,~$\psi=v_0$).

Note that $u_0$ (resp., $v_0$) is also a minimax subsolution
(resp., supersolution)
 to \eqref{E:PPDEpsi} with $\psi=u_0$.
Thus,  by the comparison principle for \eqref{E:PPDEpsi} (Corollary~{\color{black}~5.4} in \cite{BK1}),
$u_0\le u_{N,1}$ and $v_{N,1}\le v_0$ on $[r_{N-1},r_N]\times\Omega$.

By induction, one can similarly show that $(u_{N,n})$ is non-decreasing and that $(v_{N,n})$ is non-increasing.
Moreover, by Theorem~\ref{T:Claim1}~(ii), $u_{N,n}\to \bar{\psi}$ and $v_{N,n}\to\bar{\psi}$
uniformly on $[r_{N-1},r_N]\times\Omega$ as $n\to\infty$. 
Thus $u_0\le \bar{\psi}\le v_0$ on $[r_{N-1},r_N]\times\Omega$.

\medskip 

\textit{Step~2.} Let us next consider the interval $[r_{N-2},r_N]$. We can proceed as in Step~1
and arrive at the corresponding conclusion but need
to replace $h$ by $\bar{\psi}$.

\medskip 

\textit{Step~3.} Iterating the steps above yields $u_0\le\bar{\psi}\le v_0$ on $[0,T]\times\Omega$.
\end{proof}

\setcounter{equation}{0}

\renewcommand{\theequation}{\thesection.\arabic{equation}}

\appendix

\section{A general path-dependent discrete-time infinite-horizon decision model}\label{S:Appendix:MDP}
We set up a discrete-time infinite-horizon decision model with data that depend on past states and past controls.
Then we  formulate a corresponding discrete-time infinite-horizon stationary decision model (see also chapter 10 of \cite{BS}
regarding a related procedure concerning non-stationary models).
Finally, we  establish part of the Bellman equation for the former model by using the latter model.
{\color{black}  Note that here the terms stationary and non-stationary have the same meaning as
in \cite{BS}.}

\subsection{A 
 model with data that depend 
on past states and past controls}\label{SS:Appendix:MDP:pathDep}
Our notation and terminology follows largely \cite{BS}.

Consider a model with the following data (cf.~Definition~8.1 on p.~188 in \cite{BS}):
\begin{itemize}
\item State space $S$, a non-empty Borel space, which also serves as disturbance space.
\item Control space $U$, a non-empty Borel space.
\item Disturbance kernel 
$$
p_{k}(\dd w\vert x_0,\ldots,x_k;u_0,\ldots,u_k)\quad \textup{at stage} \,\,k,\quad k\in\N_0,
$$  a Borel-mea\-sur\-able stochastic kernel
 on $S$ given $S^{k+1}\times U^{k+1}$.
 \item System function 
 $$(x,u,w)\mapsto w,
\quad S\times U\times S\to S.
$$
 \item Discount factor $1$.
 \item Cost function $$
g_k:S^{k+1}\times U^{k+1}\to\R_+\cup\{+\infty\} \quad \textup{at stage}\,\, k, \quad k\in\N_0,
$$ a \textcolor{black}{lower semi-analytic}~function.
\end{itemize}   

\begin{remark}
Maybe more suggestive (and more in line with \cite{BS}),  we can write that 
our system function maps the $k$-th disturbance $w_k$ to the $(k+1)$-st state $x_{k+1}$.
\end{remark}

Our goal is to minimize a certain expected cost over the space  $\Pi^\prime$ 
of all  \emph{randomized policies} $\pi=(\mu_0,\mu_1,\ldots)$,
where each 
$$\mu_k(\dd u_k\vert x_0,\ldots,x_k;u_0,\ldots,u_{k-1}),  \quad k\in\N_0,
$$
 is a universally measurable stochastic kernel on $U$ given
$S^{k+1}\times U^k$ 
(note that $\mu_0$ is a kernel on $U$ given $S$). 

{\color{black}Later we will also use the space $\Pi$  of all \emph{non-randomized policies} $\pi=(\mu_0,\mu_1,\ldots)\in\Pi^\prime$,
i.e., for every $k\in\N_0$, every $x_0$, $\ldots$, $x_k\in S$, and every $u_0$, $\ldots$, $u_{k-1}\in U$, the measure
$\mu_k(du_k\vert x_0,\ldots,x_k;u_0,\ldots u_{k-1})$ is a Dirac measure.}

\begin{remark}\label{R:randomizedPolicy:discrete}
We use randomized policies for two reasons. First, by doing so, we operate immediately in the setting of \cite{BS}, whose
results we need to use. Second, the {\color{black} natural} notation 
for non-randomized policies would become very cumbersome in our model. E.g.,
(realized) policies would then be of the form $\mu_0(x_0)$, $\mu_1(x_0,x_1;\mu_0(x_0))$, 
$\mu_2(x_0,x_1,x_2;\mu_0(x_0), \mu_1(x_0,x_1;\mu_0(x_0)))$, and so on. 
{\color{black} Using the notation for randomized policies, we have instead}
$\mu_0(\dd u_0\vert x_0)$, $\mu_1(\dd u_1\vert x_0,x_1;u_0)\,\mu_0(\dd u_0\vert x_0)$, $\ldots$,
\begin{align*}
\mu_{n+1}(\dd u_{n+1}\vert x_0,\ldots,x_{n+1}; u_0,\ldots,u_n)\cdots \mu_0(\dd u_0\vert x_0).
\end{align*}
 We consider the notation for the latter
(realized) policies to be slightly less heavy than the notation for the former ones.
\end{remark}

Given a policy $\pi=(\mu_0,\mu_1,\ldots)\in\Pi^\prime$ and a probability measure $p$ on $S$, define a probability measure
 $r(\pi,p)$ on $(S\times U)^{\N_0}$ via
marginals $r_N(\pi,p)$ on $(S\times U)^N$, $N\in\N$, that are recursively defined by

\begin{equation}\label{E:OriginalModel:Marginals}
\begin{split}
\int h_0\,\dd r_1(\pi,p)&:=\int_S \int_U h_0(x_0;u_0)\,\mu_0(\dd u_0\vert x_0)\,p(\dd x_0),\\
\int h_1\,\dd r_2(\pi,p)&:=\int_S \int_U \int_S \int_U h_1(x_0,x_1;u_0,u_1)\\
&\qquad\qquad \mu_1(\dd u_1\vert x_0,x_1;u_0)\,
p_0(\dd x_1\vert x_0,u_0)\\ &\qquad\qquad\mu_0(\dd u_0\vert x_0)\,p(\dd x_0),\\
\int h_{N}\,\dd r_{N+1}(\pi,p)&:=\int\Bigl[\int_S\int_U h_{N}(x_0,\ldots,x_{N}; u_0,\ldots,u_{N})\\
&\qquad\qquad \mu_{N}(\dd u_{N}\vert x_0,\ldots,x_{N};u_0,\ldots,u_{N-1})\\
&\qquad\qquad p_{N-1}(\dd x_{N}\vert x_0,\ldots,x_{N-1};u_0,\ldots,u_{N-1})
\Bigr]\\ &\qquad 
r_N(\pi,p)(\dd x_0\,\dd u_0\,\ldots\,\dd x_{N-1}\,\dd u_{N-1}),\,N\in\N,
\end{split}
\end{equation}
for each bounded universally measurable function $h_k:S^{k+1}\times U^{k+1}\to\R$, $k\in\N_0$.

Given a policy $\pi\in\Pi^\prime$ and a state $x\in S$,
define the cost
\begin{equation}\label{E:OriginalModel:Cost}
\begin{split}
J_\pi(x)&:=\int \sum_{k=0}^\infty g_k\,\dd r(\pi,\delta_x)\\
&=\sum_{k=0}^\infty \int g_k(x_0,\ldots,x_k;u_0,\ldots,u_k)\,r_{k+1}(\pi,\delta_x)(\dd x_0\,\dd u_0\,\ldots\, \dd x_k\,\dd u_k).
\end{split}
\end{equation}

Given a state $x\in S$, define the optimal cost
\begin{align}\label{E:AppB:OptimalCost}
J^\ast(x):=\inf_{\pi\in\Pi^\prime} J_\pi(x).
\end{align}

We also use the spaces
$\Pi^{k}$, $k\in\N_0$, of all  policies $\pi^{k}=(\mu_{k},\mu_{k+1},\ldots)$,
where each 
$$
\mu_j(\dd u_j\vert x_0,\ldots,x_j;u_0,\ldots,u_{j-1}),\quad j\in\{k,k+1,\ldots\},
$$ is 
a universally measurable stochastic kernel on $U$ given $S^{j+1}\times U^j$ (cf.~chapter~10 of \cite{BS}).
Note that $\Pi^\prime=\Pi^0$.

Next, given $k\in\N_0$, a policy $\pi^k=(\mu_k,\mu_{k+1},\ldots)\in\Pi^k$, and a probability measure $p^k$ on $S^{k+1}\times U^k$,
define a probability measure  
$r^k(\pi^k,p^k)$ 
on $(S\times U)^{\N_0}$ via
marginals  $r^k_N(\pi^k,p^k)$ 
on $(S\times U)^{N}$, $N\in\{k+1,k+2,\ldots\}$, that are recursively defined by


\begin{equation}\label{E:OriginalModel:Marginals:kOriginating}
\begin{split}
\int h_k\,\dd r^k_{k+1}(\pi^k,p^k) &:=
\int_{S^{k+1}\times U^k} \int_U h_k(x_0,\ldots,x_k;u_0,\ldots,u_k)\, \\
&\qquad\qquad\mu_k(\dd u_k\vert x_0,\ldots,x_k;u_0,\ldots,u_{k-1})\\
&\qquad\qquad p^k(\dd x_0\,\ldots\, \dd x_k\,{\color{black}\dd u _0 \ldots\,\dd u_{k-1}}),\\
\int h_{N}\,\dd r^k_{N+1}(\pi^k,p^k)&:=\int\Big[
\int_S\int_U h_{N}(x_0,\ldots,x_{N};u_0,\ldots,u_{N})\\
&\qquad\qquad \mu_{N}(\dd u_{N}\vert x_0,\ldots,x_{N};u_0,\ldots,u_{N-1})\\
&\qquad\qquad 
p_{N-1}(\dd x_{N}\vert x_0,\ldots,x_{N-1};u_0,\ldots,u_{N-1})
\Big]\\
&\, r^k_N(\pi^k,p^k)(\dd x_0\,\dd u_0\,\ldots\,\dd x_{N-1}\,\dd u_{N-1})
\end{split}
\end{equation}
for every bounded universally measurable function $h_j:S^{j+1}\times U^{j+1}\to\R$,
$j\in\{k,k+1,\ldots\}$.

Given $k\in\N_0$,  $\pi^k\in\Pi^k$,  $x_0$, $\ldots$, $x_k\in S$, and
$u_0$, $\ldots$, $u_{k-1}\in U$,
define the cost
\begin{align}\label{E:OriginalModel:Cost:kOriginating}
J_{\pi^k}(k;x_0,\ldots,x_k;u_0,\ldots,u_{k-1})&:=\int \sum_{j=k}^\infty g_j
\,\dd r^k(\pi^k,\delta_{(x_0,\ldots,x_k;u_0,\ldots,u_{k-1})}).
\end{align}

Given $k\in\N_0$, $x_0$, $\ldots$, $x_k\in S$, and
$u_0$, $\ldots$, $u_{k-1}\in U$, define the optimal cost
\begin{align}\label{E:AppB:kOrig:OptimCost}
J^\ast(k;x_0,\ldots,x_k;u_0,\ldots,u_{k-1})):=\inf_{\pi^k\in\Pi^k} J_{\pi^k}(k;x_0,\ldots,x_k;u_0,\ldots,u_{k-1}).
\end{align}

Our goal is to establish the following  special case of the Bellman equation
(this is a counterpart of Proposition~10.1 on p.~246 in \cite{BS}). 

\begin{theorem}\label{T:OneStepDPP}
Let $x\in S$. Let $J^\ast$ and  $J^\ast(1;\cdot)$ be the functions  defined respectively  by   \eqref{E:AppB:OptimalCost} and \eqref{E:AppB:kOrig:OptimCost} with $k=1$. Then
\begin{align}\label{E:T:OneSteppDPP}
J^\ast(x)=\inf_{u\in U} \Bigl\{
g_0(x,u)+\int_S  J^\ast(1;x,x_1;u) \,p_0(\dd x_1\vert x,u)
\Bigr\}.
\end{align}
\end{theorem}
\begin{proof}
See section~\ref{SS:T:OneStepDPP}.
\end{proof}

Moreover, we show that minimization over non-randomized policies{\color{black}, i.e., over $\Pi$,}  is equivalent to
minimization over randomized policies{\color{black}, i.e., over $\Pi^\prime$}. 
We want to emphasize that we still work with the data introduced at the beginning of this section
(section~\ref{SS:Appendix:MDP:pathDep}). In particular, non-negativity of the  cost functions is important. This will enable us to apply
corresponding results in \cite{BS}.

\begin{theorem}\label{T:AppendixB:NonRandomized}
Let $x\in S$. Let $J^\ast$ and $J_\pi$  be the functions defined respectively by \eqref{E:AppB:OptimalCost} 
and \eqref{E:OriginalModel:Cost}. Then
\begin{align}\label{E:T:AppendixB:NonRandomizedMain}
J^\ast(x)=\inf_{\pi\in\Pi} J_\pi(x).
\end{align}
\end{theorem}
\begin{proof}
See section~\ref{SS:T:AppendixB:NonRandomized}.
\end{proof}

\subsection{A corresponding 
stationary 
 model}
Using the data from the previous subsection, consider a model with the following data:
\begin{itemize}
\item State space $\mathbf{S}:=\N_0\times S^{\N_0}\times U^{\N_0}${\color{black}, which is
a non-empty Borel space (see, e.g., Proposition~7.13 in \cite{BS}).} A typical state is denoted by
$$\mathbf{x}=(k;x^0,x^1,\ldots;u^0,u^1,\ldots)
$$ and $\mathbf{x}$ is to be understood of that form below.
\item Control space $U${\color{black}, which is a non-empty Borel space.}
\item Disturbance space $S${\color{black}, which is a non-empty Borel space.}
\item Disturbance kernel $$
\mathbf{p}(\dd w\vert \mathbf{x},u):=p_k(\dd w\vert x^0,\ldots,x^k;u^0,\ldots,u^{k-1},u),$$
{\color{black} which is a Borel-measurable stochastic kernel} on {\color{black} $S$}  given $\mathbf{S}\times U$.

\item System function $\mathbf{f}:\mathbf{S}\times U\times S\to\mathbf{S}$ defined by
\begin{align*}
\mathbf{f}(\mathbf{x},u,w):=(k+1;x^0,\ldots,x^k,w,x^{k+2},x^{k+3},\ldots;u^0,\ldots,u^{k-1},u,u^{k+1},u^{k+2},\ldots),
\end{align*}
{\color{black} which is a Borel measurable function.}
\item Discount factor $1$.
\item Cost function $\mathbf{g}:\mathbf{S}\times U\to\R_+$ defined by
\begin{align*}
\mathbf{g}(\mathbf{x},u):=g_k(x^0,\ldots,x^k;u^0,\ldots,u^{k-1},u),
\end{align*}
{\color{black} which is a lower semi-analytic function.}
\end{itemize}
{\color{black}
\begin{remark}
Note that our data satisfy all assumptions of the infinite-horizon model
 in Chapter 9 of \cite{BS}. This allows us to to apply the results of \cite{BS} later 
 in sections~\ref{SS:T:OneStepDPP} and \ref{SS:T:AppendixB:NonRandomized}.
\end{remark}}

The corresponding state transition kernel  $\mathbf{t}$ 
on $\mathbf{S}$ 
given $\mathbf{S}\times U$ is (as on~p.~189 in \cite{BS}) defined by
\begin{equation}\label{E:transition}
\begin{split}
\mathbf{t}(\underline{\mathbf{S}}\vert\mathbf{x},u)&:=\mathbf{p}(\{w\in S:\,\mathbf{f}(\mathbf{x},u,w)\in\underline{\mathbf{S}}\}\vert
\mathbf{x},u)\\
&=\delta_{k+1}(\underline{E})\cdot\prod_{j=0}^k \delta_{x^j}(\underline{F}_j)\cdot 
p_k(\underline{F}_{k+1}\vert x^0,\ldots,x^k;u^0,\ldots,u^{k-1},u)
\\&\qquad\qquad\cdot\prod_{j={k+2}}^\infty \delta_{x^j}(\underline{F}_j)\cdot
\prod_{j=0}^{k-1} \delta_{u^j}(\underline{G}_j)\cdot\delta_u (\underline{G}_k)\cdot
\prod_{j=k+1}^\infty \delta_{u^j}(\underline{G}_j),
\end{split}
\end{equation}
for each  
$\underline{\mathbf{S}}=\underline{E}\times\prod_{j=0}^\infty \underline{F}_j\times\prod_{j=0}^\infty \underline{G}_j\in\mathcal{B}(\mathbf{S})$, 
$\mathbf{x}=(k;x^0,x^1,\ldots;u^0,u^1,\ldots)\in\mathbf{S}$, and 
$u\in U$.

The set of policies, denoted by $\vec{\Pi}^\prime$, consists of all $\vec{\pi}=(\vec{\mu}_0,\vec{\mu}_1,\ldots)$,
where each $$\vec{\mu}_k(\dd u_k\vert \mathbf{x}_0,\,\ldots,\,\mathbf{x}_k;\,u_0,\,\ldots,u_{k-1}), \quad k\in\N_0, 
$$
is
a universally measurable stochastic kernel on $U$ given $\mathbf{S}^{k+1}\times U^k$. 


For each policy $\vec{\pi}=(\vec{\mu}_0,\vec{\mu}_1,\ldots)\in\vec{\Pi}^\prime$ and each probability measure $\mathbf{p}$ on $\mathbf{S}$,
define a probability measure $\mathbf{r}(\vec{\pi},\mathbf{p})$ on $(\mathbf{S}\times U)^{\N_0}$ via marginal distributions
$\mathbf{r}_N(\vec{\pi},\mathbf{p})$  on $(\mathbf{S}\times U)^N$, $N\in\N$, by
\begin{equation}\label{E:Augmented:Prob}
\begin{split}
\int\mathbf{h}_0\,\dd \mathbf{r}_1(\vec{\pi},\mathbf{p})
&:=\int_{\mathbf{S}}\int_U \mathbf{h}_0(\mathbf{x}_0;u_0)\,\vec{\mu}_0(\dd u_0\vert\mathbf{x}_0)\,\mathbf{p}(\dd \mathbf{x}_0),\\
\int\mathbf{h}_{N+1}\,\dd\mathbf{r}_{N+1}(\vec{\pi},\mathbf{p})
&:=\int\Biggl[
\int_{\mathbf{S}}\int_U
\mathbf{h}_{N}(\mathbf{x}_0,\ldots,\mathbf{x}_{N}; u_0,\ldots,u_{N})\\
&\qquad\qquad \vec{\mu}_{N}(\dd u_{N}\vert \mathbf{x}_0,\ldots,\mathbf{x}_{N};u_0,\ldots,u_{N-1})\\
&\qquad\qquad \mathbf{t}(\dd\mathbf{x}_{N}\vert \mathbf{x}_{N-1},u_{N-1})
\Biggr]\\
&\qquad \mathbf{r}_N(\vec{\pi},\mathbf{p})(\dd\mathbf{x}_0\,\dd u_0\,\ldots\,\dd\mathbf{x}_{N-1}\,\dd u_{N-1}),\,N\in\N,
\end{split}
\end{equation}
for every bounded universally measurable function $\mathbf{h}_k:\mathbf{S}^{k+1}\times U^{k+1}\to\R$, $k\in\N_0$.
\bibliographystyle{amsplain}

Now, we can define the cost functions $\mathbf{J}_{\vec{\pi}}:\mathbf{S}\to\R$, $\vec{\pi}\in\vec{\Pi}^\prime$, 
and the optimal cost function $\mathbf{J}:\mathbf{S}\to\R$ by
\begin{equation}\label{E:Augmented:Cost}
\begin{split}
\mathbf{J}_{\vec{\pi}}(\mathbf{x})&:=\sum_{k=0}^\infty\int \mathbf{g}(\mathbf{x}_k,u_k)\,
\mathbf{r}_{k+1}(\vec{\pi},\delta_{\mathbf{x}})(\dd\mathbf{x}_0\,\dd u_0\,\ldots\,\dd\mathbf{x}_k\,\dd u_k),\\
\mathbf{J}^\ast(\mathbf{x})&:=\inf_{\vec{\pi}\in\vec{\Pi}^\prime}\mathbf{J}_{\vec{\pi}}(\mathbf{x}).
\end{split}
\end{equation} 

{\color{black}
\begin{remark}
As an alternative to the stationary model introduced above, one could consider instead 
of our state space  $\mathbf{S}=\N_0\times S^{\N_0}\times U^{\N_0}$ the set
\begin{align*}
\bigcup_{k\in\N_0} \{(k;x^0,\ldots,x^k; u^0,\ldots,u^{k-1}):\,
(x^0,\ldots,x^k)\in S^{k+1},\,(u^0,\ldots,u^{k-1})\in U^k\}
\end{align*}
as state space and then appropriately adjust the remaining parts of the model
 such as the  disturbance kernel,
the system function, and so on. Doing so would be more in line with
section~10.1 of \cite{BS}, where one can find a similar procedure of transforming
a non-stationary to a stationary model.
\end{remark}
}

\subsection{Proof of Theorem~\ref{T:OneStepDPP}.}\label{SS:T:OneStepDPP}
Recall that $\mathbf{g}$ is non-negative. Thus, by Proposition~9.8 on p.~225 in \cite{BS} and Definition~8.5 on p.~195 in \cite{BS}, we have
\begin{align}\label{E:MDP:Bellman}
\mathbf{J}^\ast(\mathbf{x})=\inf_{u\in U}\Bigl\{\mathbf{g}(\mathbf{x},u)+\int \mathbf{J}^\ast(\mathbf{x}^\prime)\,\mathbf{t}(\dd\mathbf{x}^\prime\vert
\mathbf{x},u)\Bigr\}
\end{align}
for every $\mathbf{x}\in\mathbf{S}$. 

Now, fix $x\in S$ and some $u_\Delta\in U$, and  let $\mathbf{x}=(0;x,x\ldots;u_\Delta,u_\Delta,\ldots)$. By \eqref{E:MDP:Bellman},
\begin{align*}
\mathbf{J}^\ast(\mathbf{x})=\inf_{u\in U}\Bigl\{
g_0(x;u)+\int \mathbf{J}^\ast(1;x,w,x,x,\ldots;u,u_\Delta,u_\Delta,\ldots)\,p_0(\dd w\vert x,u)
\Bigr\}.
\end{align*}
Now, fix also $u\in U$ and $w\in S$. It remains to show that
\begin{equation}\label{E:MDP:Correspondence}
\begin{split}
\mathbf{J}^\ast(0;x,x,\ldots;u_\Delta,u_\Delta,\ldots)&=J^\ast(x)\quad\text{and}\\
\mathbf{J}^\ast(1;x,w,x,x,\ldots;u,u_\Delta,u_\Delta,\ldots)&=J^\ast(1;x,w;u),
\end{split}
\end{equation}
which we will do by proving the following four claims. 

\textit{Claim 1:} $\mathbf{J}^\ast(0;x,x,\ldots;u_\Delta,u_\Delta,\ldots)\ge J^\ast(x)$.

\textit{Claim 2:} $\mathbf{J}^\ast(1;x,w,x,x,\ldots;u,u_\Delta,u_\Delta,\ldots)\ge J^\ast(1;x,w;u)$.

\textit{Claim~3:} $\mathbf{J}^\ast(0;x,x,\ldots;u_\Delta,u_\Delta,\ldots)\le J^\ast(x)$.

\textit{Claim~4:} $\mathbf{J}^\ast(1;x,w,x,x,\ldots;u,u_\Delta,u_\Delta,\ldots)\le J^\ast(1;x,w;u)$.

\medskip 

Note that {\color{black}in the proofs of these claims,}
 we will use the following notational conventions for the components of elements 
$(\mathbf{x}_0,\mathbf{x}_1,\ldots)$
of $\mathbf{S}^{\N_0}$:
\begin{equation}\label{E:MDP:component:notation}
\mathbf{x}_k=(\iota_k;x^0_k,x^1_k,\ldots;u^0_k,u^1_k,\ldots),\quad k\in\N_0,
\end{equation}
{\color{black} where $\iota_k\in\N_0$ denotes the first component of $\mathbf{x}_k$,
$x^0_k\in S$ the second one, and so on.}

\textit{Proof of Claim~1.} Fix a policy 
$\vec{\pi}=(\vec{\mu}_0,\vec{\mu}_1,\ldots)\in\vec{\Pi}^\prime$. Next, define a
policy $\pi=(\mu_0,\mu_1,\ldots)\in\Pi^\prime$  as follows.
For every $k\in\N_0$, put
\begin{align*}
&\mu_k(\dd u_k\vert x_0,\ldots,x_k;u_0,\ldots,u_{k-1})\\&:=
\vec{\mu}_k(\dd u_k\vert (0;x,x,\ldots;u_\Delta,u_\Delta,\ldots),
(1;x,x_1,x,x,\ldots;u_0,u_\Delta,u_\Delta,\ldots),\\
&\qquad\qquad (2;x,x_1,x_2,x,x,\ldots;u_0,u_1,u_\Delta,u_\Delta,\ldots), 
\ldots,\\
&\qquad\qquad (k;x,x_1,\ldots,x_k,x,x,\ldots;u_0,\ldots,u_{k-1},u_\Delta,u_\Delta,\ldots);
u_0,\ldots,u_{k-1}).
\end{align*}
By \eqref{E:Augmented:Cost}, \eqref{E:Augmented:Prob} (cf.~also (4) on p.~191 in \cite{BS}),  \eqref{E:transition}, and \eqref{E:OriginalModel:Cost},
\begin{align*}
\mathbf{J}_{\vec{\pi}}(\mathbf{x})
&=\sum_{k=0}^\infty \int_{\mathbf{S}}\int_U\int_{\mathbf{S}}\int_U\cdots\int_{\mathbf{S}}\int_U \mathbf{g}(\mathbf{x}_k,u_k)\,
\vec{\mu}_k(\dd u_k\vert\mathbf{x}_0,\ldots,\mathbf{x}_k;u_0,\ldots, u_{k-1})\\
&
\qquad\mathbf{t}(\dd \mathbf{x}_k\vert \mathbf{x}_{k-1},u_{k-1})\cdots\vec{\mu}_1(\dd u_1\vert \mathbf{x}_0,\mathbf{x}_1;u_0)\,
\mathbf{t}(\dd \mathbf{x}_1\vert\mathbf{x}_0,u_0)\,
\vec{\mu}_0(\dd u_0\vert\mathbf{x}_0)\,\delta_{\mathbf{x}}(\dd \mathbf{x}_0)\\
&=\sum_{k=0}^\infty 
\int_U\int_{S}\int_U\cdots\int_{S}\int_U
g_k(x,x^1_1,\ldots,x^k_k; u_0,\ldots,u_{k-1},u_k)\\
&\qquad \vec{\mu}_k(\dd u_k\vert
 (0;x,x,\ldots;u_\Delta,u_\Delta,\ldots),
(1;x,x_1,x,x,\ldots;u_0,u_\Delta,u_\Delta,\ldots),\\
&\qquad\qquad (2;x,x^1_1,x^2_2,x,x,\ldots;u_0,u_1,u_\Delta,u_\Delta,\ldots), 
\ldots,\\
&\qquad\qquad (k;x,x^1_1,\ldots,x^k_k,x,x,\ldots;u_0,\ldots,u_{k-1},u_\Delta,u_\Delta,\ldots);
u_0,\ldots,u_{k-1})\\
&\qquad p_{k-1}(\dd x^k_k\vert  x,x^1_1,\ldots,x^{k-1}_{k-1};u_0,\ldots,u_{k-1}) 
\cdots\\
&\qquad \vec{\mu}_1(\dd u_1\vert  (0;x,x\ldots;u_\Delta,u_\Delta,\ldots), (1;x,x_1^1,x,x,\ldots;u_0,u_\Delta,u_\Delta,\ldots);u_0)\\
&\qquad
p_0(\dd x_1^1\vert x,u_0) 
\,\vec{\mu}_0(\dd u_0\vert (0;x,x,\ldots;u_\Delta,u_\Delta,\ldots))\\
&=J_\pi(x).
\end{align*}
This yields Claim~1.

\medskip 

\textit{Proof of Claim~2.}   Fix a policy 
$\vec{\pi}=(\vec{\mu}_0,\vec{\mu}_1,\ldots)\in\vec{\Pi}^\prime$ and
define  a policy $\pi^1=(\nu_1,\nu_2,\ldots)\in\Pi^1$ as follows.
For every $k\in\N_0$, put
\begin{align*}
&\nu_{k+1}(\dd u_{k+1}\vert x_0,\ldots,x_{k+1};u_0,\ldots,u_k)\\
&:=\vec{\mu}_k(\dd u_{k+1}\vert (1;x,w,x,x,\ldots;u,u_\Delta,u_\Delta,\ldots),
(2;x,w,x_2,x,x,\ldots;u,u_1,u_\Delta,u_\Delta,\ldots),\\
&\qquad (3;x,w,x_2,x_3,x,x,\ldots;u,u_1,u_2,u_\Delta,u_\Delta,\ldots),
\ldots,\\
&\qquad (k+1;x,w, x_2,\ldots,x_{k+1},x,x,\ldots; u,u_1,\ldots,u_k,u_\Delta,u_\Delta,\ldots);
u_1,\ldots,u_k).
\end{align*}
such that
\begin{align*}
&\nu_{k+1}(\dd u_k\vert x,w,x_2,\ldots,x_{k+1};u,u_0,\ldots,u_{k-1})\\
&=\vec{\mu}_k(\dd u_k\vert (1;x,w,x,x,\ldots;u,u_\Delta,u_\Delta,\ldots),
(2;x,w,x_2,x,x,\ldots;u,u_0,u_\Delta,u_\Delta,\ldots),\\
&\qquad (3;x,w,x_2,x_3,x,x,\ldots;u,u_0,u_1,u_\Delta,u_\Delta,\ldots), 
\ldots,\\
&\qquad (k+1;x,w, x_2,\ldots,x_{k+1},x,x,\ldots; u,u_0,\ldots,u_{k-1},u_\Delta,u_\Delta,\ldots);
u_0,\ldots,u_{k-1}).
\end{align*}
Hence, by \eqref{E:Augmented:Cost}, \eqref{E:Augmented:Prob},  \eqref{E:transition},  and
\eqref{E:OriginalModel:Cost:kOriginating},
\begin{align*}
&\mathbf{J}_{\vec{\pi}}(1;x,w,x,x,\ldots;u,u_\Delta,u_\Delta,\ldots)\\
&=\sum_{k=0}^\infty 
\int_U\int_{S}\int_U\cdots\int_{S}\int_U
g_{k+1}(x,w,x^2_1,\ldots,x^{k+1}_k; u,u_0,u_1,\ldots,u_{k-2},u_{k-1},u_k)\\
&\quad \vec{\mu}_k(\dd u_k\vert
 (1;x,w,x,x,\ldots;u,u_\Delta,u_\Delta,\ldots),
(2;x,w,x^2_1,x,x,\ldots;u,u_0,u_\Delta,u_\Delta,\ldots),\\
&\quad\quad (3;x,w,x,x^2_1,x^3_2,x,x,\ldots;u,u_0,u_1,u_\Delta,u_\Delta,\ldots), 
\ldots,\\
&\quad\quad (k+1;x,w,x,x^2_1,\ldots,x^{k+1}_k,x,x,\ldots;u,u_0,\ldots,u_{k-1},u_\Delta,u_\Delta,\ldots);
u_0,\ldots,u_{k-1})\\
&\quad p_k(\dd x^{k+1}_k\vert  x,w,x^2_1,\ldots,x^k_{k-1};u,u_0,\ldots,u_{k-1}) 
\cdots\\
&\quad \vec{\mu}_1(\dd u_1\vert  (1;x,w,x,x\ldots;u,u_\Delta,u_\Delta,\ldots), (2;x,w,x_1^2,x,x,\ldots;u,u_0,u_\Delta,u_\Delta,\ldots);u_0)\\
&\quad
p_1(\dd x_1^2\vert x,w;u,u_0) 
\,\vec{\mu}_0(\dd u_0\vert (1;x,w, x,x,\ldots;u,u_\Delta,u_\Delta,\ldots)),  
\end{align*}
where it might be also helpful to note that, by  \eqref{E:transition} and \eqref{E:MDP:component:notation}, we have
\begin{align*}
&\mathbf{t}(\dd\mathbf{x}_1\vert (1;x,w,x,x,\ldots;u,u_\Delta,u_\Delta,\ldots);u_0)
=\delta_2(\dd\iota_1)\,\delta_x(\dd x^0_1)\,\delta_w(\dd x^1_1)\\
&\qquad\cdot
p_1(\dd x^2_1\vert x,w;u,u_0)\,\prod_{j=3}^\infty \delta_x(\dd x^j_1)\,
\delta_u(\dd u^0_1)\, 
\delta_{u_0}(\dd u^1_1)\,\prod_{j=2}^\infty \delta_{u_\Delta}(\dd u^j_1),\text{ etc.}
\end{align*}
Next, changing the lower limit of summation  and then relabeling our integration variables, we obtain
\begin{align*}
&\mathbf{J}_{\vec{\pi}}(1;x,w,x,x,\ldots;u,u_\Delta,u_\Delta,\ldots)\\
&=\sum_{k=1}^\infty \int_U\int_S\int_U\cdots\int_S\int_U g_k(x,w,x^2,\ldots,x_{k-1}^k;u,u_0,u_{k-1})\\
&\qquad\qquad \nu_k(\dd u_{k-1}\vert x,w,x^2_1,\ldots,x^k_{k-1};u,u_0,\ldots,u_{k-2})\\
&\qquad\qquad p_{k-1}(\dd x^k_{k-1}\vert x,w,x^2_1,\ldots,x^{k-1}_{k-2};u,u_0,\ldots,u_{k-2})\cdots\\
&\qquad\qquad \nu_2(\dd u_1\vert x,w,x^2_1;u,u_0)\,p_1(\dd x^2_1\vert x,w;u,u_0)\,\nu_1(\dd u_0\vert x,w;u)\\
&=\sum_{k=1}^\infty \int_S\int_U\int_S\int_U\int_S\int_U\cdots\int_S\int_U
g_k(x_0,\ldots,x_k;u_0,\ldots,u_k)\\
&\qquad\qquad \nu_k(\dd u_k\vert x_0,\ldots,x_k;u_0,\ldots,u_{k-1})\,
p_{k-1}(\dd x_k\vert x_0,\ldots,x_{k-1};u_0,\ldots,u_{k-1})\cdots\\
&\qquad\qquad \nu_2(\dd u_1\vert x_0,x_1,x_2;u_0,u_1)\,p_1(\dd x_2\vert x_0,x_1;u_0,u_1)\,
\nu_1(\dd u_1\vert x_0,x_1;u_0)\\
&\qquad\qquad \delta_w(\dd x_1)\,\delta_u(\dd u_0)\,\delta_x(\dd x_0)\\
&=J_{\pi^1}(1;x,w;u).
\end{align*}
This yields Claim~2.

\medskip

\textit{Proof of Claim~3.} Fix a policy $\pi=(\mu_0,\mu_1,\ldots)\in\Pi^\prime$. Define a policy
$\vec{\pi}={\color{black}(\vec{\mu}_0,\vec{\mu}_1,\ldots)}\in\vec{\Pi}^\prime$ by  
\begin{align*}
\vec{\mu}_k(\dd u_k\vert \mathbf{x}_0,\ldots,\mathbf{x}_k;u_0,\ldots,u_{k-1}):=\mu_k(\dd u_k\vert x_0^0,\ldots,x_k^k;u_0,\ldots,u_{k-1}),\quad k\in\N_0.
\end{align*}
Then, similarly as in the proof of Claim~1, we have
\begin{align*}
J_\pi(x)&=\sum_{k=0}^\infty \int_U\int_S\int_U\cdots\int_S\int_U 
g_k(x,x^1_1,\ldots,x^k_k;u_0,\ldots,u_{k-1},u_k)\\
&\qquad \mu_k(\dd u_k\vert x,x^1_1,\ldots,x^k_k; u_0,\ldots,u_{k-1})\\
&\qquad 
p_{k-1}(\dd x^k_k\vert x,x^1_1,\ldots,x^{k-1},x_{k-1};u_0,\ldots,u_{k-1})\cdots\\
&\qquad \mu_1(\dd u_1\vert x,x^1_1;u_0)\,p_0(\dd x^1_1\vert x;u_0)\,\mu_0(\dd u_0\vert x)\\
&=\mathbf{J}_{\vec{\pi}}(\mathbf{x}).
\end{align*}
This yields Claim~3.

\medskip 

\textit{Proof of Claim~4.} Fix a policy $\pi^1=(\mu_1,\mu_2,\ldots)\in\Pi^1$. Define a policy $\vec{\pi}=(\vec{\nu}_0,\vec{\nu}_1,\ldots)\in\vec{\Pi}$ by 
\begin{align*}
{\color{black}\vec{\nu}_k}(\dd u_k\vert \mathbf{x}_0,\ldots,\mathbf{x}_k;u_0,\ldots,u_{k-1}):=
\mu_{k+1}(\dd u_k\vert x,w,x^2_1,\ldots,x^{k+1}_k;u,u_0,\ldots,u_{k-1})
\end{align*}
for each $k\in\N$. Then, similarly as in the proof of Claim~2, we have
\begin{align*}
J_{\pi^1}(1;x,w;u)&=\sum_{k=1}^\infty \int_U\int_S\int_U\cdots\int_S\int_U 
g_k(x,w,x_2,\ldots,x_k;u,u_1,\ldots,u_{k-1},u_k)\\
&\qquad  \mu_k(\dd u_k\vert x,w,x_2,\ldots,x_k;u,u_1,\ldots,u_{k-1})\\
&\qquad p_{k-1}(\dd x_k\vert x,w,x_2,\ldots,x_{k-1};u,u_1,\ldots,u_{k-1})\cdots\\
&\qquad \mu_2(\dd u_2\vert x,w,x_2;u,u_1)\,p_1(\dd x_2\vert x,w;u,u_1)\,\mu_1(\dd u_1\vert x,w;u)\\
&=\sum_{k=1}^\infty \int_U\int_S\int_U\cdots\int_S\int_U 
g_k(x,w,x^2_1,\ldots,x^k_{k-1};u,u_0,\ldots,u_{k-2},u_{k-1})\\
&\qquad  \mu_k(\dd u_{k-1}\vert x,w,x^2_1,\ldots,x^k_{k-1};u,u_0,\ldots,u_{k-2})\\
&\qquad p_{k-1}(\dd x^k_{k-1}\vert x,w,x^2_1,\ldots,x^{k-1}_{k-2};u,u_0,\ldots,u_{k-2})\cdots\\
&\qquad \mu_2(\dd u_1\vert x,w,x^2_1;u,u_0)\,p_1(\dd x^2_1\vert x,w;u,u_0)\,\mu_1(\dd u_0\vert x,w;u)\\
&=\sum_{k=0}^\infty \int_U\int_S\int_U\cdots\int_S\int_U 
g_{k+1}(x,w,x^2_1,\ldots,x^{k+1}_k;u,u_0,\ldots,u_{k-1},u_k)\\
&\qquad \mu_{k+1}(\dd u_k\vert x,w,x^2_1,\ldots,x^{k+1}_k;u,u_0,\ldots,u_{k-1})\\
&\qquad p_k(\dd x^{k+1}_k\vert x,w,x^2_1,\ldots,x^k_{k-1};u,u_0,\ldots,u_{k-1})\cdots\\
&\qquad \mu_2(\dd u_1\vert x,w,x^2_1;u,u_0)\,p_1(\dd x^2_1\vert x,w;u,u_0)\,\mu_1(\dd u_0\vert x,w;u)\\
&=\mathbf{J}_{\vec{\pi}}(1;x,w,x,x,\ldots;u,u_\Delta,u_\Delta,\ldots).
\end{align*}
This yields Claim~4.

Thanks to the four established claims above, we have  \eqref{E:MDP:Correspondence}.
This concludes the proof  of Theorem~\ref{T:OneStepDPP}. \qed

\subsection{Proof of Theorem~\ref{T:AppendixB:NonRandomized}.}\label{SS:T:AppendixB:NonRandomized}
We will make use of facts established in section~\ref{SS:T:OneStepDPP}.
Let $u_\Delta\in U$. Recall that, by \eqref{E:Augmented:Cost}, we have
\begin{align}\label{E:T:AppendixB:NonRandomized1}
\mathbf{J}^\ast(0;x,x,\ldots;u_\Delta,u_\Delta,\ldots)=
\inf_{\vec{\pi}\in\vec{\Pi}^\prime}\mathbf{J}_{\vec{\pi}}(0;x,x,\ldots;u_\Delta,u_\Delta,\ldots).
\end{align}
Since $\mathbf{g}$ is non-negative, the infimum in the right-hand side of \eqref{E:T:AppendixB:NonRandomized1}
 can, according to Proposition~9.19 of \cite{BS}, be taken over 
 \emph{non-randomized policies} $\vec{\pi}\in\vec{\Pi}^\prime$, i.e.,
 over every $\vec{\pi}=(\vec{\mu}_0,\vec{\mu}_1,\ldots)\in\vec{\Pi}^\prime$ such
 that, for each $k\in\N_0$, $\mathbf{x}_0$, $\ldots$, $\mathbf{x}_k\in\mathbf{S}$, and
 $u_0$, $\ldots$, $u_{k-1}\in U$, the measure
 $\vec{\mu}_k(\dd u_k\vert \mathbf{x}_0,\,\ldots,\,\mathbf{x}_k;\,u_0,\,\ldots,u_{k-1})$
 is a Dirac measure. 
 
 Next, following the proof of  Claim~1 in section~\ref{SS:T:OneStepDPP}, we can
 deduce via defining a corresponding non-randomized policy in $\Pi$ 
for any non-randomized policy in $\vec{\Pi}^\prime$ that
\begin{align}\label{E:T:AppendixB:NonRandomized2}
\mathbf{J}^\ast(0;x,x,\ldots;u_\Delta,u_\Delta,\ldots)\ge \inf_{\pi\in\Pi} J_\pi(x).
\end{align}

Similarly, following the proof of  Claim~3 in section~\ref{SS:T:OneStepDPP}, we can see that 
\begin{align}\label{E:T:AppendixB:NonRandomized3}
\mathbf{J}^\ast(0;x,x,\ldots;u_\Delta,u_\Delta,\ldots)\le \inf_{\pi\in\Pi} J_\pi(x).
\end{align}
Hence, by the first line of \eqref{E:MDP:Correspondence} together with \eqref{E:T:AppendixB:NonRandomized2} and
\eqref{E:T:AppendixB:NonRandomized3},  we immediately have \eqref{E:T:AppendixB:NonRandomizedMain}.
This concludes the proof of Theorem~\ref{T:AppendixB:NonRandomized}. \qed

\section{Some technical proofs}\label{A:TechnicalProofs}

\subsection{Proof of Lemma~\ref{L:phi:meas}}\label{S:A1}

First, we prove  the   auxiliary  Lemmas \ref{Claim1}, \ref{Claim2}, \ref{Claim3}. 
 
\begin{lemma}\label{Claim1}
Consider the maps 
$(s,x, {\color{black} \alpha})\mapsto\iota_i(s,x, {\color{black}\alpha})$,
$\R_+\times\Omega\times\mathcal{A}\to\Omega$, $i\in\{1,2\}$, and  
$\iota_3: (s,x,\alpha)\mapsto\iota_3(s,x,\alpha)$, 
 $\R_+\times\Omega\times\mathcal{A}\to\R_+\times\Omega{\color{black}\times\mathcal{A}}$,  respectively defined by 
\begin{align*}
[\iota_1(s,x,\alpha)](t)&:=\int_0^t f(s+r,x(\cdot\wedge s),\alpha(s+r))\,\dd r,\quad t\in\R_+,\\
[\iota_2(s,x,\alpha)](t)&:=\iota_2^t(s,x,\alpha),\quad t\in\R_+,\\
[\iota_3(s,x,\alpha)](t)&:=(s,x(t\wedge s),\alpha),\quad t\in\R_+,
\end{align*}
 where $\iota_2^t$ is given by 
 \begin{align}\label{E:iota2t}
\iota_2^t:(s,x,\alpha)\mapsto \int_0^t f(s+r,x,\alpha(s+r))\,\dd r,\quad \R_+\times\Omega\times\mathcal{A}\to\R^d.
\end{align}
We have that
\begin{itemize}
\item [(i)]the map $\iota_1$
 is measurable from $\mathcal{B}(\R_+)\otimes\cF^0\otimes\mathcal{B}(\mathcal{A})$ to $\cF^0$;
\item [(ii)]	the map $\iota_2$  is measurable from 
$\mathcal{B}(\R_+)\otimes\cF^0\otimes\mathcal{B}(\mathcal{A})$ to $\cF^0$;
\item [(iii)]the map $\iota_3$ is measurable from 
$\mathcal{B}(\R_+)\otimes\cF^0\otimes\mathcal{B}(\mathcal{A})$ to 
$\mathcal{B}(\R_+)\otimes\cF^0\otimes\mathcal{B}(\mathcal{A})$.
\end{itemize}
\end{lemma}
\noindent \textit{Proof of Lemma \ref{Claim1}.}
Item (i) directly follows from (ii)-(iii), noticing that 
 $\iota_1=\iota_2\circ\iota_3$. 
Item~(iii)   
{\color{black} follows from the measurability
of $(s,x,\alpha)\mapsto x(\cdot\wedge s)$, $\R_+\times\Omega\times\mathcal{A}\to\Omega$,  from 
$\mathcal{B}(\R_+)\otimes\cF^0\otimes\mathcal{B}(\mathcal{A})$ to $\cF^0$
due to} Theorem 96 (b), 146-IV, in 
\cite{DMprobPot} {\color{black}
together with Proposition~7.14 in \cite{BS}}. 
\noindent In order to prove item~(ii), we consider the map in \eqref{E:iota2t}.  
The proof  follows the line of the reasoning on 
p.~135
in \cite{Billingsley} regarding the measurability of path-space valued random elements,   using that

\medskip 

 (ii')
for any  $t\in\R_+$, the map $\iota_2^t$ is measurable from $\mathcal{B}(\R_+)\otimes \cF^0\otimes\mathcal{B}(\mathcal{A})$ to $\mathcal{B}(\R^d)$,

\medskip 

\noindent which we show next.

First note that 
the maps
\begin{align*}
s\mapsto\int_0^t f(s+r,x,\alpha(s+r))\,\dd r=\int_s^{s+t} f(r,x,\alpha(r))\,\dd r, 
\, \R_+\to\R^d,\, (x,\alpha)\in\Omega\times\mathcal{A},
\end{align*}
are continuous.

Moreover, for every  $s\in\R_+$, the map
\begin{align*}
 (x,\alpha)\mapsto \int_0^t f(s+r,x,\alpha(s+r))\,\dd r,\, 
\Omega\times\mathcal{A}\to\R^d,
\end{align*}
 is
$\cF^0_{t+s}\otimes\mathcal{B}(\mathcal{A})$- and hence $\cF^0\otimes\mathcal{B}(\mathcal{A})$-measurable. {\color{black}
To see this, consider the case that each component $f^j$ of $f=(f^1,\ldots,f^d)$ is of the form 
$$f^j(r,x,a)=f^j_1(x(\cdot\wedge (t+s)))\,f^j_2(r,a),$$ 
where $f^j_1$ is $\cF^0$- and $f^j_2$ is $\mathcal{B}(\R_+)\times\mathcal{B}
(A)$-measurable. Taking Remark~\ref{R:AAisBorelSpace}
into account, we can deduce the  $\cF^0\otimes\mathcal{B}(\mathcal{A})$-measurability of 
$$(x,\alpha)\mapsto \int_0^t f(s+r,x,\alpha(s+r))\,\dd r.$$ It remains to apply a monotone-class argument.
The reasoning here is very similar to the proof of Lemma~2 in \cite{Yushkevich80}.}

{\color{black} The above} two facts establish (ii') (see ~Proposition~1.13, p.~5, in \cite{KaratzasShreve}).
\qed
\normalcolor

\medskip 

\begin{lemma}\label{Claim2} The map $(s,x,\alpha)\mapsto \phi^{s,x,\alpha}_{(0)}$,
 $\R_+\times\Omega\times\mathcal{A}\to\Omega$,
defined by 
\begin{align*}
\phi^{s,x,\alpha}_{(0)}(t):=x(t).\bfone_{[0,s)}(t)+ \Bigl[x(t)+\int_s^t f(r,x(\cdot\wedge s),\alpha(r))\,\dd r\Bigr].\bfone_{
[s,\infty)}(t),\quad t\in\R_+,
\end{align*}
 is measurable from $\mathcal{B}(\R_+)\otimes\cF^0\otimes\mathcal{B}(\mathcal{A})$ to $\cF^0$.
\end{lemma}

\noindent \textit{Proof of Lemma \ref{Claim2}.} By Theorem 96 (d), 146-IV, in 
\cite{DMprobPot}, the map $\iota_4:(s,x,\tilde{x})\mapsto \iota_4(s,x)$, 
$\R_+\times\Omega\times\Omega\to\Omega$, defined by
\begin{align*}
[\iota_4(s,x,\tilde{x})](t):=x(t).\bfone_{[0,s)}(t)+\tilde{x}(t-s).\bfone_{[s,\infty)}(t),\quad
t\in\R_+,
\end{align*}
is measurable from $\mathcal{B}(\R_+)\otimes\cF^0\otimes\cF^0$ to $\cF^0$.
This  allows to conclude the proof, being 
$\phi^{s,x,\alpha}_{(0)}=\iota_4(s,x,\iota_1(s,x,\alpha))$,  
and recalling  by Lemma \ref{Claim1}-(i) the measurability of
$(s,x,\alpha)\mapsto (s,x,\iota_1(s,x,\alpha))$, $\R_+\times\Omega\times\mathcal{A}\to\R_+\times\Omega$,
from $\mathcal{B}(\R_+)\otimes\cF^0\otimes\mathcal{B}(\mathcal{A})$ 
to $\mathcal{B}(\R_+)\otimes\cF^0\otimes\cF^0$. 
\qed 

\begin{lemma} \label{Claim3} Consider the mappings
$(s,x,\alpha)\mapsto\phi^{s,x,\alpha}_{(n)}$, $\R_+\times\Omega\to\Omega$, $n\in\N$,
recursively defined by
\begin{align*}
\phi^{s,x,\alpha}_{(n+1)}(t):=x(t).\bfone_{[0,s)}(t)+
\Bigl[
x(t)+\int_s^t f(r,\phi^{s,x}_{\text{($n$)}},\alpha(r))\,\dd r
\Bigr].\bfone_{[s,\infty)}(t),\quad t\in\R_+. 
\end{align*}
Then, for every $n\in\N_0$, the map $(s,x,\alpha)\mapsto\phi^{s,x,\alpha}_{(n)}$
is measurable from $\mathcal{B}(\R_+)\otimes\cF^0\otimes\mathcal{B}(\mathcal{A})$ to $\cF^0$.
\end{lemma}

\noindent \emph{Proof of Lemma \ref{Claim3}}.
Fix $n\in\N_0$. Assume that  $(s,x,\alpha)\mapsto\phi^{s,x,\alpha}_{(n)}$
is measurable from $\mathcal{B}(\R_+)\otimes\cF^0\otimes\mathcal{B}(\mathcal{A})$ to $\cF^0$.
To show that $(s,x,\alpha)\mapsto\phi^{s,x,\alpha}_{(n+1)}$
is measurable from $\mathcal{B}(\R_+)\otimes\cF^0\otimes\mathcal{B}(\mathcal{A})$ to $\cF^0$,
one can proceed essentially as in the proof of Lemma~\ref{Claim2} but one should
use Lemma \ref{Claim1}-(ii) instead of Lemma \ref{Claim1}-(i) and note that
$$\phi^{s,x,\alpha}_{(n+1)}=\iota_4(s,x,\iota_2(s,\phi^{s,x,\alpha}_{(n)},\alpha)).
$$
Finally, note that, by Lemma \ref{Claim2}, 
$(s,x,\alpha)\mapsto\phi^{s,x,\alpha}_{(0)}$
is measurable from $\mathcal{B}(\R_+)\otimes\cF^0\otimes\mathcal{B}(\mathcal{A})$ to $\cF^0$.
The proof is concluded by  mathematical induction.
\qed

\medskip 

We can finally prove Lemma \ref{L:phi:meas}.

\medskip

\noindent \emph{Proof of Lemma \ref{L:phi:meas}}. 
By Lemma \ref{Claim3}, 
 for every $t\in\R_+$, 
the function $$
(s,x,\alpha)\mapsto\phi^{s,x,\alpha}(t), \,\,\R_+\times\Omega\times\mathcal{A}\to\R^d,
$$ is measurable from  
$\mathcal{B}(\R_+)\otimes\cF^0\otimes\mathcal{B}(\mathcal{A})$ to 
$\mathcal{B}(\R^d)$
as the limit of the functions
$$(s,x,\alpha)\mapsto\phi^{s,x,\alpha}_{(n)}(t), \,\,\R_+\times\Omega\times\mathcal{A}\to\R^d, n\in\N,
$$ which are all
 measurable from  $\mathcal{B}(\R_+)\otimes\cF^0\otimes\mathcal{B}(\mathcal{A})$
  to $\mathcal{B}(\R^d)$.
  
 Moreover, for each $(s,x,\alpha)\in\R_+\times\Omega\times\mathcal{A}$,
  the function $t\mapsto\phi^{s,x,\alpha}(t)$,
 $\R_+\to\R^d$, is right-continuous. Thus, following the proof of Proposition~1.13 on p.~5 in \cite{KaratzasShreve}, one can show that 
 $$(s,x,\alpha,t)\mapsto\phi^{s,x,\alpha}(t),\,\,\R_+\times\Omega\times A\times\R_+\to\R_+,
 $$ is 
 measurable from $\mathcal{B}(\R_+)\otimes\cF^0\otimes\mathcal{B}(\mathcal{A})\otimes \mathcal{B}(\R_+)$ 
 to $\mathcal{B}(\R^d)$.
 Hence, for every $t\in\R_+$, $$(s,x,\alpha)\mapsto\phi^{s,x,\alpha}(t),\,\,\R_+\times\Omega\to\R^d,
 $$
  is measurable from 
 $\mathcal{B}(\R_+)\otimes\cF^0\otimes\mathcal{B}(\mathcal{A})$ to $\mathcal{B}(\R^d)$,
 from which we finally can conclude, exactly as in the proof of Lemma \ref{Claim1}-(ii), that
 $(s,x,\alpha)\mapsto \phi^{s,x,\alpha}$ is measurable from 
 $\mathcal{B}(\R_+)\otimes\cF^0\otimes\mathcal{B}(\mathcal{A})$ to $\cF^0$.
 \qed

\subsection{Proof of Lemma~\ref{L:barpsi:reg:Step0}}\label{S:A2}
 Without loss of generality, let $s_1=0$ and $s_2=T$.
 
 
 Fix $s\in [0,T]$, $x$, $\tilde{x}\in \Omega$, 
 $a\in\mathcal{A}$, and $e\in\R^d$. Put
 \begin{align*}
 (\phi,\ell,\lambda,\chi,Q)&:=(\phi^{s,x,a},\ell^{s,x,a},\lambda^{s,x,a},\chi^{s,x,a},Q^{s,x,a}),\\
  (\tilde{\phi},\tilde{\ell},\tilde{\lambda},\tilde{\chi},\tilde{Q})&:=
  (\phi^{s,\tilde{x},a},\ell^{s,\tilde{x},a},\lambda^{s,\tilde{x},a},\chi^{s,\tilde{x},a},Q^{s,\tilde{x},a}).
 \end{align*}
By Assumption~\ref{A:data:controlled}, for all $t\in [s,T]$,
 \begin{equation}\label{E:Pf:barpsi:reg:Step0}
 \begin{split}
 \abs{\phi(t)-\tilde{\phi}(t)}&\le \ee^{L_f(t-s)}\norm{x-\tilde{x}}_s\text{ (Proposition~7.2~(i) in \cite{BK18JFA}),}\\
 \abs{\ell(t)-\tilde{\ell}(t)}&\le L_f\norm{\phi-\tilde{\phi}}_t\le L_f\ee^{L_f(t-s)}\norm{x-\tilde{x}}_s,\\
  \abs{\lambda(t)-\tilde{\lambda}(t)}&\le L_f\norm{\phi-\tilde{\phi}}_t\le L_f\ee^{L_f(t-s)}\norm{x-\tilde{x}}_s,\\
  \abs{\chi(t)-\tilde{\chi}(t)}&\le \int_s^t L_f|\phi(r)-\tilde{\phi}(r)|\,\dd r\text{ (p.~174 in \cite{DavisBook})}\\
  &\le L_f(t-s) \ee^{L_f(t-s)}\norm{x-\tilde{x}}_s.
  \end{split}
  \end{equation}
 Let us set
 \begin{align*}
 [(G_{0,T;\eta,a})\psi](s,x)-[(G_{0,T;\eta,a})\psi](s,\tilde{x})=I_1+I_2+I_3,
 \end{align*}
 where
 \begin{align*}
 	 I_1&:=\chi(T)\,[\eta(T,\phi)-\eta(T,\tilde{\phi})]+[\chi(T)-\tilde{\chi}(T)]\,\eta(T,\tilde{\phi}),\\
 	 I_2&:=\int_s^T \chi(t)[\ell(t)-\tilde{\ell}(t)]+[\chi(t)-\tilde{\chi}(t)]\tilde{\ell}(t)\,\dd t,\\
 	 I_3&:=J_1+J_2+J_3+J_4,
 \end{align*}
 with 
 	\begin{align*}
  J_1&:=\int_s^T \Bigl[
  \int_{\R^d} \psi(t,\phi\otimes_t e)\lambda(t)\chi(t)\,Q(t,\dd e)
  -\int_{\R^d} \psi(t,\phi\otimes_t e)\lambda(t)\chi(t)\,\tilde{Q}(t,\dd e)
  \Bigr]\,\dd t,\\
  J_2&:=\int_s^T \Bigl[
  \int_{\R^d} \psi(t,\phi\otimes_t e)\lambda(t)[\chi(t)-\tilde{\chi}(t)]\,\tilde{Q}(t,\dd e)
  \Bigr]\,\dd t,\\
   J_3&:=\int_s^T \Bigl[
  \int_{\R^d} \psi(t,\phi\otimes_t e)[\lambda(t)-\tilde{\lambda}(t)]\tilde{\chi}(t)\,\tilde{Q}(t,\dd e)
  \Bigr]\,\dd t,\\
   J_4&:=\int_s^T \Bigl[
  \int_{\R^d} [\psi(t,\phi\otimes_t e)-\psi(t,\tilde{\phi}\otimes_t e)]\tilde{\lambda}(t)\tilde{\chi}(t)\,\tilde{Q}(t,\dd e)
  \Bigr]\,\dd t.
 \end{align*}
  Then,
 \begin{align*}
 I_1&=\chi(T)\,[\eta(T,\phi)-\eta(T,\tilde{\phi})]+[\chi(T)-\tilde{\chi}(T)]\,\eta(T,\tilde{\phi})\\
 &\le c^\prime \ee^{L_f(T-s)}\norm{x-\tilde{x}}_s
 +L_f(T-s) \ee^{L_f(T-s)}\norm{x-\tilde{x}}_s\,  \norm{\eta}_\infty\\
 &\le \left[c^\prime\ee^{L_f T} + (T-s)\check{L}\right]\norm{x-\tilde{x}}_s
 \end{align*}
 thanks to \eqref{E:Pf:barpsi:reg:Step0} and Assumption~\ref{A:data:controlled}, and 
 \begin{align*}
 I_2&=\int_s^T \chi(t)[\ell(t)-\tilde{\ell}(t)]+[\chi(t)-\tilde{\chi}(t)]\tilde{\ell}(t)\,\dd t\\
 &\le L_f(T-s) \ee^{L_f(T-s)}\norm{x-\tilde{x}}_s+
 C_f(T-s) \ee^{L_f(T-s)}\norm{x-\tilde{x}}_s\\
  &\le  (T-s)\check{L}\norm{x-\tilde{x}}_s
 \end{align*}
  thanks to \eqref{E:Pf:barpsi:reg:Step0} and Assumption~\ref{A:data:controlled}.
  
Moreover, 
  \begin{align*}
  J_1&=\int_s^T \Bigl[
  \int_{\R^d} \psi(t,\phi\otimes_t e)\lambda(t)\chi(t)\,Q(t,\dd e)
  -\int_{\R^d} \psi(t,\phi\otimes_t e)\lambda(t)\chi(t)\,\tilde{Q}(t,\dd e)
  \Bigr]\,\dd t\\
  &\le C_\lambda(T-s) L_Q\,c\,\ee^{L_f(T-s)}\,\norm{x-\tilde{x}}_s\\
  & \le  (T-s)\check{L}\,c\,\norm{x-\tilde{x}}_s
  \end{align*}
  thanks to Assumptions~\ref{A:data:controlled} and \ref{A:Q} together with \eqref{E:Lip:L:barpsi:reg:Step0} and
  \eqref{E:Pf:barpsi:reg:Step0},
  \begin{align*}
  J_2&=\int_s^T \Bigl[
  \int_{\R^d} \psi(t,\phi\otimes_t e)\lambda(t)[\chi(t)-\tilde{\chi}(t)]\,\tilde{Q}(t,\dd e)
  \Bigr]\,\dd t\\
  &\le (T-s)\norm{\psi}_\infty\, C_\lambda\,L_f(T-s)\ee^{L_f(T-s)}\norm{x-\tilde{x}}_s\\
    & \le  (T-s)\check{L}\,\norm{\psi}_\infty\,\norm{x-\tilde{x}}_s
  \end{align*}
  thanks to \eqref{E:Pf:barpsi:reg:Step0} and Assumption~\ref{A:data:controlled},
   \begin{align*}
  J_3&=\int_s^T \Bigl[
  \int_{\R^d} \psi(t,\phi\otimes_t e)[\lambda(t)-\tilde{\lambda}(t)]\tilde{\chi}(t)\,\tilde{Q}(t,\dd e)
  \Bigr]\,\dd t\\
  &\le (T-s)\norm{\psi}_\infty\, L_f \ee^{L_f(T-s)}\,\norm{x-\tilde{x}}_s\\
      & \le  (T-s)\check{L}\,\norm{\psi}_\infty\,\norm{x-\tilde{x}}_s
  \end{align*}
  thanks to \eqref{E:Pf:barpsi:reg:Step0} and Assumption~\ref{A:data:controlled}, and
     \begin{align*}
  J_4&=\int_s^T \Bigl[
  \int_{\R^d} [\psi(t,\phi\otimes_t e)-\psi(t,\tilde{\phi}\otimes_t e)]\tilde{\lambda}(t)\tilde{\chi}(t)\,\tilde{Q}(t,\dd e)
  \Bigr]\,\dd t\\
  &\le (T-s)\,c\,\ee^{L_f(T-s)}\,\norm{x-\tilde{x}}_s\,C_\lambda\\
      & \le (T-s)\check{L}\,c\,\norm{x-\tilde{x}}_s
  \end{align*}
  thanks to  \eqref{E:Lip:L:barpsi:reg:Step0}, \eqref{E:Pf:barpsi:reg:Step0} and Assumption~\ref{A:data:controlled}. Therefore 
  we can see that \eqref{E:GpsiStab:xe} holds.
 \qed

 \section{On regularity of certain functions on Skorokhod space}\label{S:Appendix:D}

 \begin{example}\label{Example1}
This example is closely related to Example~2.15~(i) in \cite{Keller16SPA}. Let $d=1$.
Consider the function $u:[0,T]\times\Omega\to\R$ defined by $u(t,x):=\norm{x(\cdot\wedge t)}_\infty$.
 Fix $t_0\in [0,T)$ and consider the path $x_0:=(-2).\bfone_{[t_0,T)}$. 

 (i) Note that
$u(t_0,x_0\otimes_{t_0} (-1))=1$, but for every sufficiently small $\eps>0$,
$u(t_0+\eps,x_0\otimes_{t_0+\eps}(-1))=2$, i.e., the map
$t\mapsto u(t,x_0\otimes_t (-1))$ is not right-continuous at $t=t_0$.

 (ii) Note that $-u$ is continuous, but $(t,x)\mapsto -u(t,x\otimes_t (-1))$ is not lower semi-continuous
as, by (i),
\begin{align*} 
\liminf_{n\to\infty} \left[-u(t_0+n^{-1},x_0\otimes_{t_0+n^{-1}} (-1))\right]=-2 \not\ge -1=-u(t_0,x_0\otimes_{t_0} (-1))
\end{align*}
and $\mathbf{d}_\infty((t_0+n^{-1},x_0),(t_0,x_0))=n^{-1}\to 0$ as $n\to\infty$.
\end{example}

\begin{remark}
Continuity (or semi-continuity) of a function $u:[0,T]\times\Omega\to\R$ does not, in general, imply continuity
(or semi-continuity) of
the functions $(t,x)\mapsto u(t,x\otimes_t e)$, $e\in\R^d$ (see Example~\ref{Example1} as well as 
(an appropriate modification of) Example~2.15~(iv) in \cite{Keller16SPA}). Also note
that this topic (in a very similar but not identical context) has been addressed in section~2.4 of  \cite{Keller16SPA}.
\end{remark}

 \medskip

\medskip

\medskip

\bibliographystyle{amsplain}

\bibliography{PDP}
\end{document}